\documentclass[12pt,a4paper,fleqn]{article}

\usepackage{lscape,exscale,german,amsthm,enumerate,rawfonts,epsfig,slashbox,multirow,latexsym}
\usepackage[intlimits]{amsmath}
\usepackage[cp850]{inputenc}
\usepackage{natbib}

\newcommand{\bol}[1]{\mbox{\boldmath$#1$}}

\newcommand{\bSigma}{\boldsymbol{\Sigma}}
\newcommand{\bPsi}{\boldsymbol{\Psi}}

\newcommand{\bXi}{\boldsymbol{\Xi}}
\newcommand{\bOmega}{\boldsymbol{\Omega}}
\newcommand{\bomega}{\boldsymbol{\omega}}
\newcommand{\bu}{\mathbf{u}}
\newcommand{\bff}{\mathbf{f}}
\newcommand{\bY}{\mathbf{Y}}
\newcommand{\bS}{\mathbf{S}}
\newcommand{\bQ}{\mathbf{Q}}
\newcommand{\bd}{\mathbf{d}}
\newcommand{\bX}{\mathbf{X}}
\newcommand{\bA}{\mathbf{A}}
\newcommand{\bB}{\mathbf{B}}
\newcommand{\ba}{\mathbf{a}}
\newcommand{\bb}{\mathbf{b}}
\newcommand{\bL}{\mathbf{L}}

\newcommand{\bC}{\mathbf{C}}

\newcommand{\bU}{\mathbf{U}}

\newcommand{\bR}{\mathbf{R}}

\newcommand{\bv}{\mathbf{v}}

\newcommand{\bzero}{\mathbf{0}}
\newcommand{\bI}{\mathbf{I}}

\newcommand{\bF}{\mathbf{F}}

\newcommand{\bD}{\mathbf{D}}

\newcommand{\bV}{\mathbf{V}}

\newtheorem{theorem}{Theorem}

\newtheorem{lemma}{Lemma}

\newfont{\tabfont}{cmr7 at 7pt}

\usepackage{color}

\usepackage{natbib}
\setcitestyle{numbers,square}
\begin{document}

\begin{center}
\vspace*{0.7cm} \noindent {\bf \large Exact and Asymptotic Tests on a Factor Model in Low and Large Dimensions with Applications}\\
\vspace{0.5cm} \noindent {\sc Taras Bodnar$^{a}$\footnote{Corresponding Author: Taras Bodnar. E-Mail: taras.bodnar@math.su.se. Tel: +46 8 164562. Fax: +46 8 612 6717. This research was partly supported by the Deutsche Forschungsgemeinschaft via the Research Unit FOR 1735 ''Structural Inference in Statistics: Adaptation and Efficiency''. The first author appreciates the financial support of SIDA via the project 1683030302.} and Markus Rei{\upshape{\ss}}$^{b}$
}\\
\vspace{0.2cm}
{\it \footnotesize  $^a$ Department of Mathematics, Stockholm University, Roslagsv\"{a}gen 101, SE-10691 Stockholm, Sweden}\\
{\it \footnotesize  $^a$ Department of Mathematics, Humboldt-University of Berlin, Unter den Linden 6, D-10099 Berlin, Germany}\\
%{\it \footnotesize   E-mails: \{bodnar, mreiss\}@math.hu-berlin.de}
\end{center}

\vspace{0.2cm}
\begin{abstract}
In the paper, we suggest three tests on the validity of a factor model which can be applied for both, small-dimensional and large-dimensional data. The exact and asymptotic distributions of the resulting test statistics are derived under classical and high-dimensional asymptotic regimes. It is shown that the critical values of the proposed tests can be calibrated empirically by generating a sample from the inverse Wishart distribution with identity parameter matrix. The powers of the suggested tests are investigated by means of simulations. The results of the simulation study are consistent with the theoretical findings and provide general recommendations about the application of each of the three tests. Finally, the theoretical results are applied to two real data sets, which consist of returns on stocks from the DAX index and on stocks from the S\&P 500 index. Our empirical results do not support the hypothesis that all linear dependencies between the returns can be entirely captured by the factors considered in the paper.
\end{abstract}

\vspace{0.1cm}

\section{Introduction}

Factor models are widely spread in different fields of science, especially, in economics and finance where this type of models have been increasing in popularity recently. They are often used in forecasting mean and variance (see, e.g., \citet{StockWatson2002a,StockWatson2002b}, \citet{marcellino2003macroeconomic}, \citet{artis2005factor}, \citet{boivin2005understanding}, \citet{anderson2007forecasting} and the references therein), in macroeconomic analysis (see, \citet{bernanke2003monetary}, \citet{favero2005principal}, \citet{giannone2006vars}), in portfolio theory (see, \citet{ross1976arbitrage,ross1977capital}, \citet{engle1981one}, \citet{chamberlain1983funds}, \citet{chamberlain1983arbitrage}, \citet{diebold1989dynamics}, \citet{fama1992cross,fama1993common}, \citet{aguilar2000bayesian}, \citet{Bai2003inferential}, \citet{ledoit2003improved}). Factor models are also popular in physics, psychology, biology (e.g., \citet{rubin1982algorithms}, \citet{carvalho2008high}) as well as in multiple testing theory (e.g., \citet{friguet2009factor}, \citet{dickhaus2012simultaneous}, \citet{fan2012estimating}).

Another stream of research related to factor models deals with the estimation of high-dimensional covariance and precision matrices. This approach is motivated by a rapid development of high-dimensional factor models during the last years (\citet{BaiNg2002determining,BaiNg2008foundations}, \citet{bai2012statistical}, \citet{Bai2013fixed}). \citet{fan2008high}, \citet{fan2012vast}, \citet{fan2013large} among others have suggested several methods for estimating the covariance and precision matrices based on factor models in high dimensions and applied their results to portfolio theory, whereas \citet{ledoit2003improved} have proposed to combine the sample covariance matrix with the single-factor model based estimator in order to improve the estimate of the covariance matrix. Here, they use the capital asset pricing model (CAPM) as a single-factor model. \citet{ross1976arbitrage,ross1977capital} argues that the empirical success of the CAPM can be explained by the validity of the following three assumptions: i) there are many assets; ii) the market permits no arbitrage opportunity; iii) asset returns have a factor structure with a small number of factors. He also presents a heuristic argument that if an infinite number of assets is present on the market, then it is possible to construct sufficiently many riskless portfolios. In \citet{chamberlain1983funds}, conditions are derived under which this heuristic argument of Ross is justified. Furthermore, \citet{chamberlain1983arbitrage} suggest the so-called approximate $K$-factor structure model where the number of assets is assumed to be infinite, while \citet{fan2008high} and \citet{LiLiShi2013} extend this model by considering an asymptotically infinite number of known and unknown factors, respectively.

Let $X_{it}$ be the observation data for the $i$-th cross-section unit at time $t$. For instance, in the case of portfolio theory, $X_{it}$ represents the return of the $i$-th asset at time $t$. Let $\bX_t=(X_{1t},\ldots,X_{pt})^\top$ be the observation vector at time $t$ and let $\bff_t$ be a $K$-dimensional vector of common observable factors at time $t$. Then the factor model in vector form is expressed as
\begin{equation}\label{fm}
\bX_t = \bB \bff_t + \bu_t
\end{equation}
where $\bB$ is the matrix of factor loadings and $\bu_t$, $t = 1,\ldots, T$, are independent errors with covariance matrix $\bSigma_u$. It is also assumed that $\bff_t$ are independent in time as well as independent of $\bu_t$. The estimation of the factor model or the covariance matrix resulting from the factor model with observable factors is considered by \citet{fan2008high}, whereas \citet{Bai2003inferential}, \citet{bai2012statistical}, \citet{fan2013large} present the results under the assumption that the factors are unobservable. Moreover, \citet{BaiNg2002determining}, \citet{hallin2007determining}, \citet{kapetanios2010testing}, \citet{onatski2010determining}, \citet{ahn2013eigenvalue} among others deal with the problem of determining the number of factors $K$ used in (\ref{fm}). Note that not in all models the factors are observable. For example, this is not the case in many applications in psychology or in multiple testing theory, and, consequently, the results derived in the present paper cannot be directly applied. On the other hand, factor models with observable variables are usually considered in economics and finance where we also provide two empirical illustrations of the obtained theoretical results.

Under the generic assumption that $\bSigma_u$ is a diagonal matrix, the dependence between the elements of $\bX_t$ is fully determined by the factors $\bff_t$. This means that the precision matrix of $\bY_t = (\bX^\top_t, \bff^\top_t )^\top$ has the following structure
\begin{equation}\label{prec_matr}
\bOmega = \{{\rm cov}(\bY_t)\}^{-1} = \left(
  \begin{array}{cc}
    \bOmega_{11} & \bOmega_{12} \\
    \bOmega_{21} & \bOmega_{22} \\
  \end{array}\right)\,,
\end{equation}
where $\bOmega_{21}=\bOmega_{12}^\top$ is a $p\times K$ matrix and $\bOmega_{11}$ is a diagonal $p\times p$ matrix, if the factor model (\ref{fm}) is true, i.e., if all linear dependencies among the components of $\bX_t$ are fully captured by the factor vector $\bff_t$. As a result, the test on the validity of the factor model (\ref{fm}) is equivalent to testing
\begin{equation}\label{test_fm}
H_0 : \bOmega_{11} = {\rm diag}(\omega_{11},\ldots,\omega_{pp}) \qquad \text{versus} \qquad H_1 : \bOmega_{11} \neq {\rm diag}(\omega_{11},\ldots,\omega_{pp})
\end{equation}
for some positive constants $\omega_{11},\ldots,\omega_{pp}$.

We contribute to the existing literature on factor models by deriving exact and asymptotic tests on the validity of the factor model which are based on testing (\ref{test_fm}). Furthermore, the distributions of the suggested test statistics are obtained under both hypotheses and also they are analyzed in detail when the dimension of the factor model tends to infinity as the sample size increases such that $p/(T-K)\longrightarrow c \in (0,1]$. This asymptotic regime is known in the statistical literature as double asymptotic regime or high-dimensional asymptotics.

Alternatively to the test (\ref{test_fm}), one can apply the classical goodness-of-fit test which is based on the estimated residuals given by
\[\hat{\bu}_t=\bX_t - \widehat{\bB} \bff_t \,,\]
where $\widehat{\bB}$ is an estimate of the factor loading matrix. This approach, however, does not always lead to reliable results. To see this, let $\bX=(\bX_1,\ldots,\bX_T)$, $\bF=(\bff_1,\ldots,\bff_T)$, and $\widehat{\bU}=(\hat{\bu}_1,\ldots,\hat{\bu}_T)$. If $\bB$ is estimated by applying the least square method, i.e., $\widehat{\bB}=\bX\bF^\top(\bF\bF^\top)^{-1}$, then
\[\widehat{\bU}=\bX-\widehat{\bB}\bF=\bX(\bI_T-\bF^\top(\bF\bF^\top)^{-1}\bF)\,,\]
where $\bI_T$ is the $T$-dimensional identity matrix. Under the assumption of normality it holds that $\bU|\bF\sim\mathcal{N}_{p,n}(\mathbf{0},\bSigma_u\otimes \bI_n)$ ($p \times n$ dimensional matrix variate normal distribution with zero mean matrix and covariance matrix $\bSigma_u\otimes \bI_n$) and, consequently, $\widehat{\bU}|\bF\sim\mathcal{N}_{p,n}(\mathbf{0},\bSigma_u\otimes (\bI_n-\bF^\top(\bF\bF^\top)^{-1}\bF))$. Hence, $(\hat{\bu}_t)_{t=1,\ldots,T}$ are autocorrelated and their distribution depends on the factor matrix $\bF$, although the true residuals $(\bu_t)_{t=1,\ldots,T}$ are independent and their distribution does not depend on $\bF$. This unpleasant property of $\widehat{\bU}$ surely influences testing procedures based on $\hat{\bu}_t$. It is remarkable that in contrast to the test based on the covariance matrix of the residuals, the suggested approach which is based on the precision matrix does not suffer from this problem. Moreover, our tests can be applied without imposing an additional identifiability condition on the model, i.e., it is not assumed that the factors are orthogonal, since the estimator for the matrix of factor loadings does not play any role in the derived test theory.

The rest of the paper is organized as follows. In the next section, we provide the mathematical motivation for the testing procedures (considered in the paper). In Section 3, two finite sample tests are suggested which are constructed in two steps. First, marginal test statistics are constructed and then the maxima of the marginal test statistics are calculated. We further prove that the distributions of the maxima do not depend on $\omega_{11},\ldots,\omega_{pp}$ and, consequently, the corresponding critical values can be calibrated via simulations. In Section 4, the likelihood ratio test is investigated. Similarly to the tests of Section 3, the distribution of the likelihood ratio statistic does not depend on $\omega_{11},\ldots,\omega_{pp}$ used in (\ref{test_fm}) under the null hypothesis. The results are extended to the case of high-dimensional factor models in Section 5. Here, the high-dimensional asymptotic distributions of the test statistics considered in Sections 3 and 4 are derived. The cases of $c< 1$ and $c=1$ are treated separately in detail. The results of the simulation study in Section 6 illustrate the size and power of the suggested tests, whereas an empirical study is provided in Section 7. We summarize our findings in Section 8. Proofs are given in the appendix.

\section{Mathematical Motivation of Three Tests}

A test on the hypothesis (\ref{test_fm}) can be performed in different ways. Below, we provide a full mathematical motivation for the three approaches considered in the paper.

The first method is based on testing the hypothesis that all non-diagonal elements of $\bOmega_{11}$ are equal to zero, i.e.,
\begin{equation}\label{test_fm_nondiag}
H_0 : \omega_{ij} = 0 ~ \text{for}~ 1\le j<i\le p \qquad \text{versus} \qquad H_1 : \omega_{ij}\neq 0~ \text{for at least one $(i,j)$,}
\end{equation}
where $\bOmega=(\omega_{ij})_{i,j \in \{1,\ldots,p+T\}}$.

The second approach is based on the following result\\[0.2cm]

\begin{lemma}\label{lem1}
Let $\bA=(a_{ij})_{i,j=1,\ldots,q}$ be a symmetric positive-definite matrix and let $\mathbf{B}=\bA^{-1}=(b_{ij})_{i,j=1,\ldots,q}$. Then $a_{ii}b_{ii}\ge 1$ holds for all $i=1,\ldots,q$ and $\bA$ is a diagonal matrix if and only if $a_{ii}b_{ii}=1$ for all $i=1,\ldots,q$.
\end{lemma}

\vspace{0.8cm}
The proof of Lemma \ref{lem1} is given in the appendix. This result motivates the reformulation of the hypothesis (\ref{test_fm}) in the following way
\begin{equation}\label{test_fm_prod-diag}
H_0 : \omega_{jj}\omega_{jj}^{(-)} = 1 ~ \text{for}~ 1\le j\le p \qquad \text{versus} \qquad H_1 : \omega_{jj}\omega_{jj}^{(-)}> 1 ~ \text{for at least one $j$,}
\end{equation}
where $\bOmega_{11}^{-1}=(\omega_{ij}^{(-)})_{i,j \in \{1,\ldots,p\}}$.

The third procedure is based on Hadamard's inequality (see, e.g., Section 4.2.6 of \citet{lutkepohl1997handbook}): for any positive definite symmetric matrix $\bA$ it holds that
\[{\rm det}(\bA)\le \prod_{i=1}^p a_{ii}\,,\]
with equality only if $\bA$ is a diagonal matrix. This approach leads to the hypothesis expressed as
\begin{equation}\label{test_fm_det}
H_0 : \frac{\prod_{i=1}^p \omega_{ii}}{{\rm det}(\bOmega _{11})} = 1 \qquad \text{versus} \qquad H_1 : \frac{\prod_{i=1}^p \omega_{ii}}{{\rm det}(\bOmega _{11})} > 1 \,.
\end{equation}

The test statistics for the null hypotheses (\ref{test_fm_nondiag}) and (\ref{test_fm_prod-diag}) are presented in Section 3, whereas testing (\ref{test_fm_det}) leads to the likelihood ratio test of Section 4.

\section{Small Sample Tests: $p,T$ are Finite}

Let
\begin{equation}\label{sample_cov_matr}
\bS = \frac{1}{T} \bY\bY^\top
\end{equation}
be the sample covariance matrix calculated for the sample $\bY_1,\ldots,\bY_T$ with $\bY=(\bY_1,\ldots,\bY_T)$. It is used to estimate $\bSigma=\bOmega^{-1}$. In (\ref{sample_cov_matr}) the sample mean vector of $\bY_t$ is not subtracted since the population mean vector is zero following (\ref{fm}) and the assumptions that $E(\bff_t)=\mathbf{0}$ and $E(\bu_t)=\mathbf{0}$. If model (\ref{fm}) is extended by adding a mean vector, i.e., to
\begin{equation}\label{fm_ext}
\bX_t = \bol{\mu}+\bB \bff_t + \bu_t\,,
\end{equation}
then the covariance matrix should be estimated by
\begin{equation}\label{sample_cov_matr_ext}
\widetilde{\bS} = \frac{1}{T-1} \bY \left(\bI_T-\frac{1}{T}\mathbf{J}_T\right)\bY^\top\,,
\end{equation}
where $\mathbf{J}_T$ is the $T \times T$ matrix of ones.

Assuming that both $\{\bff_t\}$ and $\{\bu_t\}$ are independent and identically distributed sequences from a multivariate normal distribution, we get that $T\bS \sim \mathcal{W}_{p+K}(T,\bSigma)$ ($(p+K)$-dimensional Wishart distribution with $T$ degrees of freedom and covariance matrix $\bSigma$). Consequently, $\bV=(T\bS)^{-1} \sim \mathcal{W}^{-1}_{p+K}(T+p+K+1,\bOmega)$ for $p+K<T$ (see, Theorem 3.4.1 in \citet{GuptaNagar2000}).

Similarly, we get that $(T-1)\widetilde{\bS} \sim \mathcal{W}_{p+K}(T-1,\bSigma)$ in the case of model (\ref{fm_ext}). Consequently, without loss of generality, we put $\bol{\mu}=\mathbf{0}$ in the rest of the paper, since the derived test statistics are fully determined by the elements of $\bS$ and in the case of $\widetilde{\bS}$ only a minor adjustment is needed. We further note that the assumption of normality is not restrictive in many applications. For instance, the asset returns at weekly or smaller frequency are well described by the normal distribution (see, \citet{Fama1976}). Moreover, \citet{TuZhou2004} find no benefits of heavy tailed distributions for the mean-variance investor and pointed out that the application of the normal assumption instead of a heavy tailed distribution leads to a relative small amount of losses.

Let $\bV=(v_{ij})_{i,j=1,\ldots,p+K}$ and let $\bV$ be partitioned as
\begin{equation}\label{part_sampl_prec}
\bV=\left(
  \begin{array}{cc}
    \bV_{11} & \bV_{12} \\
    \bV_{21} & \bV_{22} \\
  \end{array}\right) \quad \text{with} \quad \bV_{11}:p \times p\,.
\end{equation}

\subsection{Test Based on Each Non-Diagonal Element of $\bOmega_{11}$}
Testing hypothesis (\ref{test_fm_nondiag}) can also be considered as the global test of the marginal tests with hypotheses given by
\begin{equation}\label{test_fm_nondiag_marg}
H_{0,ij} : \omega_{ij} = 0 ~ \qquad \text{versus} \qquad H_{1,ij} : \omega_{ij}=d_{ij}\neq 0
\end{equation}
for $1\le j <i \le p$. In terms of multiple testing theory we are thus interested in testing the global hypothesis $H_0=\bigcap_{1\le j<i \le p} H_{0,ij}$. For each hypothesis in \eqref{test_fm_nondiag_marg}, $1\le j<i \le p$, we consider the following test statistic
\begin{equation}\label{test_stat_nondiag_marg}
T_{ij}=(T-K-p+1)\frac{g_{ij}^2}{1-g_{ij}^2} \quad \text{with} \quad g_{ij}=\frac{v_{ij}}{\sqrt{v_{ii}v_{jj}}}\,,
\end{equation}
The expression of $T_{ij}$ corresponds to the statistic used in testing for the uncorrelatedness between two random variables (see Section 5 of \citet{Muirhead}), although differences in the normalizing factor and in the distribution of the test statistics are present.

Let $\mathcal{F}_{i,j}$ denote the $\mathcal{F}$-distribution with degrees $i$ and $j$ and let $F_{i,j}$ be the corresponding density function. In the following we make also use of the hypergeometric function given by (see, \citet{AbramowitzStegun1964})
\[ _2F_1(a,b,c;x) = \frac{\Gamma(c)}{\Gamma(a)\Gamma(b)} \; \sum_{i=0}^\infty \frac{\Gamma(a+i) \Gamma(b+i)}{\Gamma(c+i)} \; \frac{z^i}{i!} \; . \]

The distribution of the test statistic $T_{ij}$ is obtained both under $H_{0,ij}$ and under $H_{1,ij}$ and it is presented in Theorem \ref{th1a}.\\[0.2cm]

\begin{theorem} \label{th1a}
Let $\bX_t$ follow model (\ref{fm}) where $\bff_t$ and $\bu_t$ are independent and normally distributed. Then:
\begin{enumerate}[(a)]
\item The density of $T_{ij}$ is given by
\begin{eqnarray*}
F_{T_{ij}}(x)&=& F_{1,T-K-p+1}(x) (1+\lambda_{ij})^{-(T-K-p+2)/2}\\
&\times& \,_2F_1\Big(\frac{T-K-p+2}{2},\frac{T-K-p+2}{2},\frac{1}{2};
\frac{x}{T-K-p+1+x} \; \frac{\lambda_{ij}}{1+\lambda_{ij}}  \Big) \,,
\end{eqnarray*}
where $\lambda_{ij} = d_{ij}^2/\{\omega_{jj}(\omega_{ii}-d_{ij}^2/\omega_{jj})\}$.
\item Under $H_{0,ij}$ it holds that $T_{ij} \sim \mathcal{F}_{1,T-K-p+1}$.
\end{enumerate}
\end{theorem}

\vspace{0.8cm}
The proof of Theorem \ref{th1a} is given in the appendix. Since the test statistics $(T_{ij})_{1\le j < i \le p}$ under the global hypothesis $H_0$ have the same distribution, we consider single-step multiple tests for testing (\ref{test_fm_nondiag}). The test statistic is given by
\begin{equation}\label{T_el}
T_{el}=\max_{1\le j<i \le p} T_{ij}\,.
\end{equation}

The marginal critical values for the $ij$-marginal test are derived from the equality
\begin{equation}\label{T_el_cond_ij}
{\rm Pr}_{H_0}(T_{el}>c_{1-\alpha}^{(el)})\le \alpha \,.
\end{equation}
Solving (\ref{T_el_cond_ij}) is a challenging problem since the test statistics $(T_{ij})_{1\le j < i \le p}$ are dependent. The first possibility to deal with this problem is the application of a Bonferroni correction. This leads to
\begin{equation*}
c_{1-\alpha}^{(el,B)}=F_{1,T-K-p+1;1-2\alpha/p(p-1)}\,,
\end{equation*}
where $F_{1,T-K-p+1;1-2\alpha/p(p-1)}$ stands for the $\left\{1-2\alpha/p(p-1)\right\}$-quantile of the $\mathcal{F}$-distribution with $1$ and $T-K-p+1$ degrees of freedom.

The second possibility is based on the observation that the expressions of the test statistics $(T_{ij})_{1\le j<i\le p}$ remain the same if $\bV_{11}$ is replaced by $\bD\bV_{11}\bD$ for any diagonal matrix $\bD$ of an appropriate order. Hence, the joint distribution of $(T_{ij})_{1\le j < i \le p}$ under the global hypothesis $H_0$ does not depend on $\omega_{11},\ldots,\omega_{pp}$. As a result, the critical values of the marginal tests $c_{1-\alpha}^{(el)}$ can be calibrated via simulations by generating a sample from the inverse Wishart distribution with $T-K+p+1$ degrees of freedom and identity parameter matrix. Under the alternative hypothesis, however, the distribution of $T_{el}$ cannot be obtained explicitly and needs to be explored by simulations. This point is discussed in more detail in Section 6 where the powers of the suggested tests are compared with each other.

\subsection{Test Based on the Product of Diagonal Elements of $\bOmega_{11}$ and $\bOmega_{11}^{-1}$}
Testing hypothesis (\ref{test_fm_prod-diag}) can be considered as the global hypothesis of the multiple tests whose hypotheses are given by
\begin{equation}\label{test_fm_prod-diag_marg}
H_{0,j} : \omega_{jj}\omega_{jj}^{(-)} = 1  \qquad \text{versus} \qquad H_{1,j} : \omega_{jj}\omega_{jj}^{(-)}=d_j > 1 \,,
\end{equation}
for $j=1,\ldots,p$, i.e., $H_{0}=\bigcap_{1\le j\le p} H_{0,j}$.

Similarly to Section 3.1, we first consider a test for the marginal hypothesis $H_{0,j}$. Let
$\bV_{11}^{(-)} =(v_{ij}^{(-)})_{i,j=1,\ldots,p}$. Then the test statistic for testing (\ref{test_fm_prod-diag_marg}) is given by
\begin{equation}\label{test_stat_col}
T_j = \frac{T-K-p+1}{p-1} (v_{jj} v_{jj}^{(-)}-1)\,.
\end{equation}

In Theorem \ref{th1} we present the exact distribution of $T_j$ under the null $H_{0,j}$ as well as under the alternative $H_{1,j}$ hypotheses.\\[0.2cm]

\begin{theorem} \label{th1}
Let $\bX_t$ follow model (\ref{fm}) where $\bff_t$ and $\bu_t$ are independent and normally distributed. Then:
\begin{enumerate}[(a)]
\item The density of $T_j$ is given by
\begin{eqnarray*}
F_{T_j}(x)&=& F_{p-1,T-K-p+1}(x) (1+\lambda_j)^{-(T-K)/2}\\
&\times& \,_2F_1\Big(\frac{T-K}{2},\frac{T-K}{2},\frac{p-1}{2};
\frac{(p-1)x}{T-K-p+1+(p-1)x} \; \frac{\lambda_j}{1+\lambda_j}  \Big) \,,
\end{eqnarray*}
where $\lambda_j = (\omega_{jj}\omega_{jj}^{(-)}-1)$.
\item Under $H_{0_j}$ it holds that $T_j \sim \mathcal{F}_{p-1,T-K-p+1}$.
\end{enumerate}
\end{theorem}

\vspace{0.8cm}
For testing the global hypothesis $H_{0}=\bigcap_{1\le j\le p} H_{0,j}$ in (\ref{test_fm_prod-diag}) we consider
\begin{equation}\label{T_el}
T_{pr}=\max_{1\le j \le p} T_{j}\,,
\end{equation}
where the critical value is obtained as a solution of
\begin{equation}\label{T_el_cond}
{\rm Pr}_{H_0}(T_{pr}>c_{1-\alpha}^{(pr)})\le \alpha \,.
\end{equation}
Since the derivation of the joint distribution of $(T_{j})_{1\le j \le p}$ is a complicated task, we consider two procedures how $c_{1-\alpha}^{(pr)}$ can be determined. The first procedure makes use of a Bonferroni correction. In this case, using Theorem \ref{th1}.(b) we get
\begin{equation*}
c_{1-\alpha}^{(pr,B)}=F_{p-1,T-K-p+1;1-\alpha/p}\,.
\end{equation*}

The second procedure is based on the following result.\\[0.2cm]

\begin{theorem} \label{th2}
Let $\bX_t$ follow model (\ref{fm}) where $\bff_t$ and $\bu_t$ are independent and normally distributed with diagonal matrix $\bSigma_u$. Then the distribution of $T_{pr}$ under $H_0$ is independent of $\omega_{11},\ldots,\omega_{pp}$.
\end{theorem}

\vspace{0.8cm}
Therefore, the critical values $c_{1-\alpha}^{(pr)}$ for the multiple tests $T_{pr}$ can be calibrated via simulations by generating a sample from the $p$-dimensional inverse Wishart distribution with $T-K+p+1$ degrees of freedom and identity parameter matrix.

\section{Likelihood-Ratio Test}
In this section we derive a test statistics for testing (\ref{test_fm}) following the third approach outlined in Section 2. It is remarkable that this procedure leads to the likelihood ratio test.

Let ${\rm etr}(.)=\exp({\rm trace}(.))$ denote the exponential of the trace and let $\Gamma_p(.)$ be the $p$-dimensional gamma function defined by
\[\Gamma_p\left(\frac{n}{2}\right)=\pi^{\frac{p(p-1)}{4}}\prod_{i=1}^{p}
\Gamma\left(\frac{n-i+1}{2}\right)\,.\]
Then the density of $\bV=(T\bS)^{-1}$ is given explicitly by
\begin{eqnarray*}
f(\bV;\bOmega)&=& \frac{2^{-(p+K)T/2}}{\Gamma_{T+p}\left(\frac{T}{2}\right)}\frac{({\rm det} \bOmega)^{T/2}}{({\rm det} \bV)^{(T+p+K+1)/2}}
{\rm etr}\left(-\frac{1}{2}\bV^{-1}\bOmega\right)\\
&=& \frac{2^{-(p+K)T/2}}{\Gamma_{T+p}\left(\frac{T}{2}\right)}\frac{({\rm det} \bOmega_{11})^{T/2}\{{\rm det}(\bOmega_{22}-\bOmega_{21}\bOmega_{11}^{-1}\bOmega_{12})\}^{T/2}}{({\rm det} \bV)^{(T+p+K+1)/2}}
{\rm etr}\left(-\frac{1}{2}\bV_{11}^{-1}\bOmega_{11}\right)\nonumber\\
&\times&{\rm etr}\left\{-\frac{1}{2}(\bV_{22}-\bV_{21}\bV_{11}^{-1}\bV_{12})^{-1}(\bOmega_{22}-\bOmega_{21}\bOmega_{11}^{-1}\bOmega_{12})\right\}\nonumber\\
&\times&{\rm etr}\left\{-\frac{1}{2}(\bV_{22}-\bV_{21}\bV_{11}^{-1}\bV_{12})^{-1}(\bV_{21}\bV_{11}^{-1}-\bOmega_{21}\bOmega_{11}^{-1})\bOmega_{11}
(\bV_{21}\bV_{11}^{-1}-\bOmega_{21}\bOmega_{11}^{-1})^\top\right\}\,,\nonumber
\end{eqnarray*}
where the last equality is obtained by following the proof of Theorem 3 in \citet{BodnarOkhrin2008}. As the transformation from the set of parameters $(\bOmega_{11},\bOmega_{21},\bOmega_{22})$ to $(\bOmega_{11},\bPsi_{21}=\bOmega_{21}\bOmega_{11}^{-1}$, $\bPsi_{22}=\bOmega_{22}-\bOmega_{21}\bOmega_{11}^{-1}\bOmega_{12})$ is one-to-one (see, e.g., Proposition 5.8 in \citet{Eaton2007}), we rewrite the likelihood function of $\bV$ in terms of the parameters $(\bOmega_{11},\bPsi_{21},\bPsi_{22})$:
\begin{eqnarray}\label{F_bV}
f(\bV;\bOmega)
&=& \frac{2^{-(p+K)T/2}}{\Gamma_{T+p}\left(\frac{T}{2}\right)}\frac{({\rm det} \bOmega_{11})^{T/2}({\rm det} \bPsi_{22})^{T/2}}{({\rm det} \bV)^{(T+p+K+1)/2}}{\rm etr}\left(-\frac{1}{2}\bV_{11}^{-1}\bOmega_{11}\right)\\
&\times&{\rm etr}\left\{-\frac{1}{2}(\bV_{22}-\bV_{21}\bV_{11}^{-1}\bV_{12})^{-1}\bPsi_{22}\right\}\nonumber\\
&\times&{\rm etr}\left\{-\frac{1}{2}(\bV_{22}-\bV_{21}\bV_{11}^{-1}\bV_{12})^{-1}(\bV_{21}\bV_{11}^{-1}-\bPsi_{21})\bOmega_{11}
(\bV_{21}\bV_{11}^{-1}-\bPsi_{21})^\top\right\}\,.\nonumber
\end{eqnarray}

Let
\begin{eqnarray}\label{marg_tbW11}
g(\bV_{11};\bOmega_{11})=({\rm det}\bOmega_{11})^{\frac{T}{2}} {\rm etr} \left(-\frac{1}{2}\bV_{11}^{-1} \bOmega_{11} \right) \,.
\end{eqnarray}

It is noted that the third factor in (\ref{F_bV}) is always less than or equal to $1$ with equality if an only if $\bPsi_{21}=\bV_{21}\bV_{11}^{-1}$ for any given $\bOmega_{11}$. Moreover, using the multiplicative representation of the likelihood function and noting that no restrictions are imposed under $H_0$ in \eqref{test_fm} on $\bPsi_{21}$ and $\bPsi_{22}$, we get that the likelihood ratio test statistics is then given by
\begin{eqnarray}\label{LR}
T_{LR}^*&=&\frac{\sup_{\bOmega_{11}>0}g(\bV_{11};\bOmega_{11})}{\sup_{\omega_{11}>0,\ldots,\omega_{pp}>0} g\{\bV_{11};{\rm diag}(\omega_{11},\ldots,\omega_{pp})\}}\nonumber\\
&=&\frac{\sup_{\bOmega_{11}>0} ({\rm det} \bOmega_{11})^{\frac{T}{2}} {\rm etr} \left(-\frac{1}{2}\bV_{11}^{-1} \bOmega_{11} \right)}
{\sup_{\omega_{11}>0,\ldots,\omega_{pp}>0}\left(\prod_{i=1}^p \omega_{ii} \right)^{\frac{T}{2}} {\rm etr} \left\{-\frac{1}{2}\bV_{11}^{-1} {\rm diag}(\omega_{11},\ldots,\omega_{pp})\right\}}
\,.
\end{eqnarray}
The maximum of the numerator is reached at $\bOmega_{11}^*= T\bV_{11} $, whereas the maximum of the denominator is attained at $\omega_{ii}^*=T(v_{ii}^{(-)})^{-1}$ for $i=1,\ldots,p$ where $v_{ii}^{(-)}$ denotes the $i$-th diagonal element of $\bV_{11}^{-1}$. Hence,
\begin{equation}\label{LR_final_a}
T_{LR}^*= \frac{({\rm det}\bV_{11})^{\frac{T}{2}}}{\prod_{i=1}^p \{(v_{ii}^{(-)})^{-1}\}^{\frac{T}{2}}}
=\left(\frac{{\rm det}\bV_{11}^{-1}}{\prod_{i=1}^p v_{ii}^{(-)}}\right)^{-\frac{T}{2}}\,.%\sim \chi^2_{p(p-1)/2}~\text{as}~ T \longrightarrow \infty
\end{equation}

Due to $\bV_{11} \sim \mathcal{W}_p^{-1}(T-K+p+1,\bOmega_{11})$ (see, Theorem 3 in \citet{BodnarOkhrin2008}), we get that $\bV_{11}^{-1} \sim \mathcal{W}_p(T-K,\bOmega_{11}^{-1})$. The last statement motivates the use of $(T_{LR}^*)^{(T-K)/T}$ instead of $T_{LR}^*$ which appears to be a well-known test statistic in multivariate analysis (see, e.g., Section 11 in \citet{Muirhead}). It is used to test the null hypothesis that the $p$ elements of a normally distributed random vector are independent which equivalently can be expressed as
\begin{equation}\label{test_fm_lr}
\tilde{H}_0 : \bOmega_{11}^{-1} = {\rm diag}(\tilde{\omega}_{11},\ldots,\tilde{\omega}_{pp}) \qquad \text{versus} \qquad \tilde{H}_1 : \bOmega_{11}^{-1} \neq {\rm diag}(\tilde{\omega}_{11},\ldots,\tilde{\omega}_{pp})
\end{equation}
for some positive constants $\tilde{\omega}_{11},\ldots,\tilde{\omega}_{pp}$, whereas the sample of size $T-K$ is used. It is noted that if the null hypothesis in \eqref{test_fm} is true then the null hypothesis in \eqref{test_fm_lr} is true and vice versa. Finally, we point out that testing \eqref{test_fm_lr} is also equivalent to testing whether the correlation matrix related to the covariance matrix $\bOmega_{11}^{-1}$ is the identity matrix (see, Section 7.4.3 in \citet{Rencher2002}).

We use the test statistic given by
\begin{equation}
T_{LR}=2\rho \ln (T_{LR}^*)^{(T-K)/T} \quad \text{with} \quad \rho=1-\frac{2p+5}{6(T-K)} \,,
\end{equation}
which is asymptotically $\chi^2_f$-distributed with $f=p(p-1)/2$ degrees of freedom under the null hypothesis in \eqref{test_fm_lr} (see, e.g., Section 7.4.3 in \citet{Rencher2002}). Since the expression of $T_{LR}$ remains unchanged if $\bV_{11}$ is replaced by $\bD\bV_{11}\bD$ for any diagonal matrix $\bD$ of an appropriate order, the distribution of $T_{LR}$ does not depend on $\omega_{11},\ldots,\omega_{pp}$ under $H_0$. Hence, the critical value of this test can be calibrated by generating a sample from the inverse Wishart distribution with $T-K+p+1$ degrees of freedom and identity parameter matrix.

Finally, let us point out that the critical value only depends on the dimension $p$. The asymptotics that the number of factors $K$ tends to infinity such that $T-K \rightarrow \infty$ is thus covered as well if the dimension $p$ remains fixed.

\section{High-Dimensional Asymptotic Test}

In this section we derive the distribution of the test statistics $T_j$, $T_{el}$, $T_{ij}$, $T_{pr}$, and $T_{LR}$ in the case when both $p$ and $K$ tend to infinity as the sample size $T$ increases. This case is known in the statistical literature as the high-dimensional asymptotic regime. It is remarkable that in this case the results obtained under the standard asymptotic regime ($p$ is fixed) can deviate significantly from those obtained under high-dimensional asymptotics (see, e.g., \citet{BaiSilverstein2010}).

Several papers deal with the problem of estimating the covariance and the precision matrices from high-dimensional data. The results are usually obtained by applying the shrinkage technique (see, e.g., \citet{ledoit2003improved}, \citet{BodnarGuptaParolya2014b,BodnarGuptaParolya2014a,bodnar2016direct}) or by imposing some conditions on the structure of the covariance (precision) matrix (see, \citet{cai2011adaptive}, \citet{cai2011constrained}, \citet{agarwal2012noisy}, \citet{fan2008high}, \citet{fan2013large}). For instance, in \citet{agarwal2012noisy} an assumption is imposed that the covariance matrix can be presented as a sum of a sparse matrix and a low rank matrix. This structure of the covariance matrix is similar to the one obtained assuming a factor model (see, e.g., \citet{fan2013large} for discussion).

Although several tests on the covariance matrix under high-dimensional asymptotics have been suggested recently (see, e.g., \citet{johnstone2001distribution}, \citet{bai2009corrections}, \citet{chen2010tests}, \citet{cai2011limiting}, \citet{jiang2013central}, \citet{GuptaBodnar2014}), we are not aware of any test on the precision matrix in the literature. The latter problem is closely related to the test theory developed in this paper since the suggested tests can be presented as tests on the specific structure of the precision matrix. Their distributions under high-dimensional asymptotics are derived in this section.

Later on, we distinguish between two cases, $p/(T-K) \longrightarrow c \in (0,1)$ as $T-K \longrightarrow \infty$ and $p/(T-K) \longrightarrow 1_{-}$ such that $T-K-p \longrightarrow d \in (0,\infty)$ as $T-K \longrightarrow \infty$. The number of factors could be both asymptotically finite or infinite, but must remain smaller than the sample size $T$. Finally, it is also assumed that $p \le T-K$ to ensure the invertibility of $\bS$.

\subsection{Asymptotic Distributions of $T_{ij}$}

As the finite sample distribution of the test statistics $(T_{ij})_{1\le j <i \le p}$ depend on $p$, $K$, and $T$ through the difference $T-K-p$ only, we get the following result. \\[0.2cm]

\begin{theorem}\label{th3a}
Let $\bX_t$ follow model (\ref{fm}) where $\bff_t$ and $\bu_t$ are independent and normally distributed. Then:
\begin{enumerate}[(a)]
\item Under $H_{0_{ij}}$ we have
\begin{eqnarray}
T_{ij}\stackrel{d.}{\longrightarrow} \chi^2_1
\end{eqnarray}
for $p/(T-K) \rightarrow c \in (0,1)$ as $T-K \rightarrow \infty$, and
\begin{eqnarray}
T_{ij} \stackrel{d.}{\longrightarrow} \mathcal{F}_{1,d+1} ~~ \text{for} ~~ T-K-p \longrightarrow d \in (0,\infty) ~~ \text{as} ~~ T-K \longrightarrow \infty \,.
\end{eqnarray}
%%%%%%%%%%%%%%%%%%%%%%%%%%%%%%%%%%%%%%%%%%%%%%%%%%%%%%%%%%%%
\item Under $H_{1,{ij}}$ it holds that
\begin{eqnarray}
\left(\sqrt{T_{ij}}-\sqrt{T-K-p+2}\sqrt{\lambda_{ij}}\right)^2 \stackrel{d.}{\longrightarrow} \chi^2_{1}
\end{eqnarray}
for $p/(T-K) \rightarrow c \in (0,1)$ as $T-K \rightarrow \infty$, and
\begin{eqnarray}
\left(\sqrt{T_{ij}}-\sqrt{T-K-p+1}\frac{\sqrt{v_{jj}}}{\sqrt{\omega_{jj}}}\sqrt{T-K-p+2}\sqrt{\lambda_{ij}}\right)^2 \stackrel{d.}{\longrightarrow} \mathcal{F}_{1,d+1}
\end{eqnarray}
for $ T-K-p \longrightarrow d \in (0,\infty)$ as $T-K \longrightarrow \infty$ where $\lambda_{ij}$ is given in the statement of Theorem \ref{th1a}.
\end{enumerate}
\end{theorem}

\subsection{Asymptotic Distributions of $T_j$}
Let $\bv_{21,j}$ be the $j$th column of $\bV_{11}$ leaving out $v_{jj}$ and let $\bV_{22,j}$ be the $(p-1)\times(p-1)$ matrix obtained from $\bV_{11}$ by deleting its $j$th row and its $j$th column. We define $\bQ_{j}=\bV_{22,j}-\bv_{21,j}\bv_{21,j}^\top/v_{jj}$. Then using the results of Lemma \ref{lem2} from the appendix (see Section 9), we get with $\bL=\bI_p$ that
\[T_j=\frac{\omega_{11,j}\left(\frac{\bv^\top_{21,j}}{v_{jj}}\right)^\top \bQ_{j}^{-1}\left(\frac{\bv^\top_{21,j}}{v_{jj}}\right)/(p-1)}
{(\omega_{11,j}/v_{jj})/(T-K-p+1)} \stackrel{a.s.}{\longrightarrow} 1 \]
for $p/(T-K) \longrightarrow c \in (0,1)$ as $T-K \longrightarrow \infty$ under $H_{0,j}$ since if $\eta \sim \chi^2_{q,\lambda}$ then $\eta/q \stackrel{a.s.}{\longrightarrow} 1$ as
$q \longrightarrow \infty$.

In Theorem \ref{th3} we derive the weak limit under high-dimensional asymptotics of a transformation of $T_j$, $j=1,\ldots,p-1$, $p/(T-K) \longrightarrow c \in (0,1)$ as $T-K \longrightarrow \infty$ as well as the weak limit of $T_j$ for $T-K-p \longrightarrow d \in (0,\infty)$ as $T-K \longrightarrow \infty$.\\[0.5cm]

\begin{theorem}\label{th3}
Let $\bX_t$ follow model (\ref{fm}) where $\bff_t$ and $\bu_t$ are independent and normally distributed. Then:
\begin{enumerate}[(a)]
\item Under $H_{0,j}$ we have
\begin{eqnarray}
\sqrt{p-1}\left(T_j-1 \right)\stackrel{d.}{\longrightarrow} \mathcal{N}\left(0,\frac{2}{1-c}\right)
\end{eqnarray}
for $p/(T-K) \rightarrow c \in (0,1)$ as $T-K \rightarrow \infty$, and
\begin{eqnarray}
T_j \stackrel{d.}{\longrightarrow} \frac{d+1}{\chi^2_{d+1}} ~~ \text{for} ~~ T-K-p \longrightarrow d \in (0,\infty) ~~ \text{as} ~~ T-K \longrightarrow \infty \,.
\end{eqnarray}
%%%%%%%%%%%%%%%%%%%%%%%%%%%%%%%%%%%%%%%%%%%%%%%%%%%%%%%%%%%%
\item Under $H_{1,j}$ it holds that
\begin{eqnarray}
\sqrt{p-1}\left(\frac{\omega_{11,j}\left(\frac{\bv^\top_{21,j}}{v_{jj}}\right)^\top \bQ_{j}^{-1}\left(\frac{\bv^\top_{21,j}}{v_{jj}}\right)/(p-1)-\frac{\lambda_j}{c}}
{(\omega_{11,j}/v_{jj})/(T-K-p+1)}-1 \right)\stackrel{d.}{\longrightarrow} \mathcal{N}\left(0,\frac{2}{1-c}+4\frac{\lambda_j}{c}\right)
\end{eqnarray}
for $p/(T-K) \rightarrow c \in (0,1)$ as $T-K \rightarrow \infty$, and
\begin{eqnarray}
T_j \stackrel{d.}{\longrightarrow} \frac{\left(1+\frac{\lambda_j}{c}\right)(d+1)}{\chi^2_{d+1}} ~~ \text{for} ~~ T-K-p \longrightarrow d \in (0,\infty) ~~ \text{as} ~~ T-K \longrightarrow \infty
\end{eqnarray}
where $\lambda_j$ is given in the statement of Theorem \ref{th1}.
\end{enumerate}
\end{theorem}

\vspace{0.8cm}
The marginal test based on the statistic $\sqrt{p-1}\left(T_j-1 \right)$ rejects the null hypothesis $H_{0,j}$ if the value of the test statistic multiplied by $\sqrt{1-c}/\sqrt{2}$ is larger than $z_{1-\alpha}$ ($(1-\alpha)$-quantile of the standard normal distribution). Using that $T_j \sim \mathcal{F}_{p-1,T-K-p+1}$ (see Theorem \ref{th1}), a finite sample correction of the statistic $\sqrt{p-1}\left(T_j-1 \right)$ can be suggested. Since the expectation and the variance of a $\mathcal{F}_{p-1,T-K-p+1}$-random variable are given by
\[\mu_F=\frac{T-K-p+1}{T-K-p-1} \quad \text{and} \quad var_F=\frac{2(T-K-2)(T-K-p+1)^2}{(T-K-p-3)(T-K-p-1)^2}\,,\]
we get the following finite sample adjusted version:
\begin{equation*}
\sqrt{p-1}\frac{T_j-\mu_F}{\sqrt{var_F}}\stackrel{d.}{\longrightarrow} \mathcal{N}\left(0,1\right) \quad \text{under} \quad H_{0,j}\,.
\end{equation*}
Of course, under the high-dimensional asymptotics, we get $\mu_F \longrightarrow 1$ and $var_F\longrightarrow 2/(1-c)$ for $p/(T-K) \rightarrow c \in (0,1)$ as $T-K \rightarrow \infty$.

\subsection{Likelihood Ratio Test under High-Dimensional Asymptotics}

In this subsection we extend the results of Section 4 by deriving the asymptotic distribution of the likelihood ratio test statistics under the high-dimensional asymptotic regime. The results are obtained in case of $p/(T-K) \longrightarrow c \in (0,1)$ as $T-K \longrightarrow \infty$ as well as in case of $p/(T-K) \longrightarrow 1_{-}$ such that $T-K-p \longrightarrow d \in (0,\infty)$ as $T-K \longrightarrow \infty$.

First, we note that the statistic $T_{LR}^*$ can be further rewritten. Let $\bR$ be the correlation matrix calculated from the Wishart distributed matrix $\bV_{11}^{-1}$, that is
\[\bR={\rm diag}\{(v_{11}^{(-)})^{-1/2},\ldots,(v_{pp}^{(-)})^{-1/2}\}\bV_{11}^{-1}
{\rm diag}\{(v_{ii}^{(-)})^{-1/2},\ldots,(v_{pp}^{(-)})^{-1/2}\}\,.\]
Then the test statistic $T_{LR}^*$ can be presented by
\begin{equation}
T_{LR}^*={\rm det}(\bR)^{-\frac{T}{2}}\,.
\end{equation}

The asymptotic distribution in case of the likelihood ratio test under $H_0$ in \eqref{test_fm} is given in Theorem \ref{th4}.\\[0.2cm]

\begin{theorem}\label{th4}
Let $\bX_t$ follow model (\ref{fm}) where $\bff_t$ and $\bu_t$ are independent and normally distributed. Then under $H_0$ in \eqref{test_fm} for $p/(T-K) \rightarrow c \in (0,1)$ as $T-K \rightarrow \infty$ as well as for $T-K-p \longrightarrow d \in (0,\infty)$ as $T-K \longrightarrow \infty$ with $d\ge 4$, we get
\begin{eqnarray}
\frac{\dfrac{2}{T}\ln(T_{LR}^*)+\mu_{LR}}{\sigma_{LR}}\stackrel{d.}{\longrightarrow} \mathcal{N}\left(0,1\right)\,,
\end{eqnarray}
where
\begin{equation}\label{mu_LR}
\mu_{LR}=\left(p-1-(T-K)+\frac{3}{2}\right)\ln\left(1-\frac{p}{T-K}\right)-\frac{T-K-1}{T-K}p
\end{equation}
and
\begin{equation}\label{sigma_LR}
\sigma_{LR}=-2\left\{\frac{p}{T-K}+\ln\left(1-\frac{p}{T-K}\right)\right\}\,.
\end{equation}
\end{theorem}

\vspace{0.8cm}
The proof of the theorem follows directly from Corollary 1 of \citet{jiang2013central}, where it is shown that
\[\frac{\ln({\rm det}\bR)-\mu_{LR}}{\sigma_{LR}}\stackrel{d.}{\longrightarrow} \mathcal{N}\left(0,1\right)\,,\]
where $\mu_{LR}$ and $\sigma_{LR}$ are given in \eqref{mu_LR} and \eqref{sigma_LR}, respectively.

\section{Finite-Sample Performance}

In this section we investigate the power of the three tests suggested in the previous sections. The analysis is performed for both, small (Section 6.1) and large (Section 6.2) values of $p$.

The critical values of each test are obtained via simulations or by using a Bonferroni correction in case of $T_{el}$ and $T_{pr}$ as well as the asymptotic distribution for $T_{LR}$. Consequently, in all plots six lines are shown. The lines denoted by $T_{el}$, $T_{pr}$, and $T_{LR}$ correspond to the case of calibrated critical values of the tests, whereas the notations $T_{el-B}$, $T_{pr-B}$, and $T_{LR-as}$ mean that a Bonferroni correction or the asymptotic distribution was used. The critical values, which are based on simulations, are obtained by generating a sample of $10^5$ realizations from the inverse Wishart distribution with $T-K+p+1$ degrees of freedom  and identity parameter matrix. Based on this sample, the sample quantiles of the corresponding test statistics are calculated and used as critical values.

The situation is more complex if the aim is to access the power of the suggested tests, since the powers depend on the model specified under the alternative hypothesis. In order to investigate the powers of the tests, we simulate data following (\ref{fm}). Namely, the vector of factors and the residual vector are generated independently from each other as well as independently in each repetition from $\mathcal{N}_K(\bzero,\bI_k)$ in case of $\bff_t$ and from $\mathcal{N}_p(\bzero,\bSigma_u)$, $\bSigma_u={\rm diag}(\eta_1,\ldots,\eta_p)\mathbf{\Delta} {\rm diag}(\eta_1,\ldots,\eta_p)$ with $\mathbf{\Delta}=(\rho_{ij})_{1\le j<i \le p}$ and $\rho_{ii}=1$, $i=1,\ldots,p$, in case of $\bu_t$. Here, $\eta_1,\ldots,\eta_p$ determine the standard deviations of $u_1,\ldots,u_p$, respectively, whereas $\mathbf{\Delta}$ stands for the correlation matrix. In order to get reliable results which do not depend on one model only, we take different parameters for $\bB=(b_{ij})_{i=1,\ldots,p;j=1,\ldots,K}$ and $\eta_i$, $i=1,\ldots,p$ in each repetitions. Namely, we specify all these quantities randomly following $\eta_i\sim UNI[1,2]$ and $b_{ij}\sim UNI[-1,1]$. The correlation matrix $\mathbf{\Delta}$ has been chosen in three possible ways in order to account for the behaviour of the tests under different deviations from $H_0$. We further increase the number of factors in model (\ref{fm}) and perform the test assuming that a lower number of factors is present.

We present four scenarios for generating data in detail which are used in the investigation of the test powers.

\begin{itemize}
\item \textbf{Scenario 1:} Change in one correlation coefficient.\\
Here, it is assumed that $\rho_{12}=\rho_{21}=\rho$ with $\rho \in \{-0.5,-0.45,\ldots,0,\ldots,0.45,0.5\}$. The remaining correlations are set to zero.

\item \textbf{Scenario 2:} Change in one column.\\
Let $\mathbf{\Delta}^{-1}=(\rho_{ij}^{(-)})$ with
\begin{equation*}
\rho_{1j}^{(-)}=\rho_{j1}^{(-)}=\left\{
\begin{array}{ccc}
\frac{sign(\rho^{j-1})|\rho|}{\sqrt{1+3(p-1)\rho^2/2}} & \text{for} & j>1, \\
1 & \text{for} & j=1 \\
\end{array}
\right. \,.
\end{equation*}
The remaining correlation coefficients are zero.

\item \textbf{Scenario 3:} All correlation coefficients are changed.\\
Here, we put $\rho_{ij}=\rho^{|j-i|}$ for $i,j \in \{1,\ldots,p\}$.

\item \textbf{Scenario 4:} Change in the number of factors.\\
The number of factors in the true model is increased to $K+\tilde{K}$ with $\tilde{K}\in\{1,\ldots,10\}$.
\end{itemize}

These four scenarios lead to different types of factor models under the alternative hypothesis. For instance, in case of Scenario 1, a single change in the correlation matrix of residuals is assumed, whereas Scenario 2 leads to changes in the first column (row) of $\bOmega_{11}$. Scenario 3 corresponds to changes in all elements of $\bOmega_{11}$ although their magnitude becomes smaller as the difference between the row number and the column number increases. Here, the structure of $\mathbf{\Delta}$ corresponds to the structure of the correlation matrix of an AR(1)-process. Finally, Scenario 4 assumes that the true factor model consists of $K+\tilde{K}$ factors, whereas the factor model with $K$ factors is fitted.

For different scenarios, we expect different performances of the suggested three tests with respect to their powers. For the first scenario, the $T_{el}$ test is expected to be the best one, whereas the $T_{pr}$ test should outperform the competitors in case of Scenario 2. Finally, when changes in the entire correlation matrix are present, the likelihood ratio test ($T_{LR}$) should possess the best performance. Furthermore, the application of $T_{el}$ provides more information to the practitioners than in the case of $T_{pr}$ and $T_{LR}$. Only the conclusion about the validity of the factor model can be drawn when the test $T_{LR}$ is used, whereas the test $T_{pr}$ can indicate the columns in the precision matrix $\bol{\Omega}$ which are responsible for the rejection of the null hypothesis. In contrast, the testing procedure $T_{el}$ determines the pairs of variables for which the null hypothesis is rejected.

\subsection{Results for Small Dimension}

In this subsection, we present the results of our simulation study under the assumption that $p$ is much smaller than $T-K$ and/or all quantities $p$, $T$, and $K$ are finite. Different values of $K=5$, $p \in \{10, 20\}$, and $T \in \{30,60,100\}$ are considered. Moreover, we put $\rho \in \{-0.5,-0.45,\ldots,0,\ldots,0.45,0.5\}$ and $\tilde{K}\in\{1,\ldots,10\}$, as described above. Finally, the nominal size of the tests is set to $\alpha=0.05$.

The resulting powers are shown in Figures \ref{Fig:T30K5p10}-\ref{Fig:T100K5p20}. In Figures \ref{Fig:T30K5p10} and \ref{Fig:T60K5p20}, we present the results for small sample size, whereas Figures \ref{Fig:T100K5p10} and \ref{Fig:T100K5p20} correspond to large $T-K$ with respect to $p$. It is not surprising that if $T-K$ is relatively small with respect to $p$, then the $T_{LR}$ test based on the asymptotic distribution shows the probability of type 1 error larger than the nominal value of $5\%$. Consequently, a finite sample adjustment for this test is required. This is achieved by calibrating the critical values of this test following the results of Section 4.

\begin{center} Figures 1 and 4 above here \end{center}

The figures with the exception of Figure \ref{Fig:T60K5p20} confirm our expectation. In case of Scenario 1, the best approach is the $T_{el}$ test followed by the $T_{pr}$ test, whereas for the rest of the considered scenarios this test shows the worst performance in almost all of the considered cases. For Scenario 2, the best approach is based on the application of the $T_{pr}$ statistic, while in both, Scenario 3 and Scenario 4, the likelihood ratio test outperforms the competitors.

We also observe that the lines which correspond to the Bonferroni correction or which are obtained from the asymptotic distribution almost coincide with the corresponding lines obtained by calibrating the critical values under the null hypothesis if $T=100$. This indicates that under $H_0$ the event that two marginal test statistics are simultaneously beyond the critical value is negligible. In contrast, for smaller sample sizes ($T=30$ and $T=60$) this statement does not hold, especially for the $T_{LR}$-asymptotic test.

\subsection{Results for Large Dimension}

In this subsection we deal with the case of high-dimensional factor models. Two possible sets of values for $p$, $K$, and $T$ are considered, namely $\{p=100,K=10,T=500\}$ with $c \approx 0.2$ and $\{p=100,K=20,T=250\}$ with $c \approx 0.36$. The nominal size of the tests is set to $\alpha=0.05$. Similarly to the previous subsection, six lines are plotted in each figure. Three of them correspond to the tests based on the calibrated critical values, whereas for the other three lines the asymptotic results of Theorems \ref{th3a} and \ref{th4} together with the Bonferroni correction are used.

\begin{center} Figures 5 and 6 above here \end{center}

The results of Figures \ref{Fig:T500K10p100} and \ref{Fig:T300K20p100} are even more pronounced than the ones in case of small $p$. Namely, for Scenario 1, the best test is based on the $T_{el}$ statistic, clearly outperforming the rest of competitors. Here, a very poor performance of the likelihood ratio test is observed which is to be expected because the dimension of $\bOmega$ becomes
large and a change in a single entry has only a minor impact on the determinant. On the other side, the $T_{el}$ test possesses very small power for the rest of the considered scenarios. The test based on the $T_{pr}$ statistic is the best one in case of Scenario 2 and shows the same performance as the $T_{LR}$ approach for Scenario 4. Finally, in case of Scenario 3, the likelihood ratio test clearly outperforms the other approaches.

\begin{center} Figure 7 above here \end{center}

It is also noted that the Bonferroni correction does not work well in case of the $T_{pr}$ test. This is explained by the problem of approximating the $\mathcal{F}$-distribution with both degrees of freedom large by the normal distribution under high-dimensional asymptotics (see Figure \ref{Fig:Fdist}). Although the histograms for $p \in  \{100,1000,10000,100000\}$ look like the ones which correspond to the normal distribution, they are slightly moved to the left and do provide a good approximation only if $p\ge 10000$. Since the maximum of dependent $F$-statistics is taken in the definition of $T_{pr}$ under $H_0$, this effect becomes even more pronounced. It is documented in Figures \ref{Fig:T500K10p100} and \ref{Fig:T300K20p100} by red lines which significantly deviate from the corresponding black lines obtained for calibrated $p$-values. As a result, it is recommendable to apply the results of Section 5 only in case of very high-dimensional factor models. If $p$ is smaller than $1000$, then it is better to construct Bonferroni corrections based on the exact $\mathcal{F}$-distributions given in Theorems \ref{th1a} and \ref{th1} instead of the asymptotic ones from Theorems \ref{th3a} and \ref{th3}.

\section{Empirical Illustration}

In this section, we apply the theoretical results of the paper to test if the market indices can be used as factors in describing the dynamics of the asset returns. This idea corresponds to the capital asset pricing model (CAPM) and the arbitrage pricing theory (APT) which are widely used in portfolio analysis.

\citet{BaiFangLiang2014spectral} point out that Markowitz's portfolio selection theory (see, \citet{Markowitz,markowitz1991foundations}) has already set up the foundation for the CAPM. These ideas are further extended by \citet{sharpe1964capital} and \citet{lintner1965security} in case of the presence of a risk-free asset, whereas \citet{black1972capital} generalizes the CAPM to the case when a risk-free asset is not available by deriving the so-called zero-beta CAPM. As a proxy for the returns of the market portfolio, which plays a role of the factor in the CAPM, the returns of the market indices, like the DAX index or the S\&P 500 index, are usually used.

The APT is an extension of the CAPM model which was suggested by \citet{merton1973intertemporal} and \citet{ross1976arbitrage}. In contrast to the CAPM, which is based on a single factor only, several factors are used in the APT in order to fit the dynamics in the asset returns. These factors are usually presented by other market  or industry-sector indices, like the TecDAX index or the NASDAQ bank index. One of the main ideas behind the APT is that it is commonly not enough to model the asset returns by a single factor and, thus, further factors have to be included into the model. Finally, \citet{chamberlain1983arbitrage} suggest a high-dimensional factor model for capturing the dynamics in the asset returns (see, \citet{fan2013large}).

Estimation and testing the CAPM (APT) is an important topic in finance today (see, \citet{shanken1992estimation,shanken1986testing,shanken1992estimation}, \citet{velu1999testing}, \citet{shanken2007estimating}, \citet{sentana2009econometrics}, \citet{beaulieu2012identification}, \citet{reiss2014nonparametric}). Recently, \citet{sentana2009econometrics} provides a survey of mean-variance efficiency tests which play a special role in the CAPM and have increased their popularity after the seminal paper of \citet{gibbons1989test}. \citet{beaulieu2012identification} suggest exact simulation-based procedures for testing the zero-beta CAPM and constructing confidence intervals for the zero-beta rate.

We apply the theoretical results of the paper to test the validity of a factor model with specified factors in case of the returns on stocks included into the German DAX index (Section 7.1) as well as in case of the returns on stocks included into the USA S\&P index (Section 7.2). The first empirical study corresponds to a factor model with $p=20$, whereas the second one to the high-dimensional model with $p=100$.

\subsection{Analysis of Stocks Included into the DAX Index}
We perform the $T_{el}$, $T_{pr}$, and $T_{LR}$ tests on the validity of factor models fitted to the returns of $20$ stocks included into the DAX index. These $20$ stocks are chosen randomly out of all $30$ stocks which determine the value of the DAX index. Repeating this procedure $10^4$ times, $10^4$ models are fitted and tests on the validity of each model are performed. As factors, we use the returns of the DAX index in the first approach. In the second approach, we included three further factors, namely, the STOXX50E index, the TecDAX index, and the MDAX index. In all cases, weekly returns are considered from the 11th of June 2012 to the 10th of June 2014 ($T=104$ observations) obtained from the Yahoo! finance web-page.$^1$ \footnote{$^1$ It has to be noted that the distribution of monthly returns is closer to the normal distribution compared to the shorter term returns. However, the application of monthly data over longer periods of time may lead to biased results due to non-constant parameters. In contrast, the daily data cause problems with the assumption of normality. For this reason we opt for the weekly frequency, which is a trade-off between the two extremes.}

%%%%%%%%%%%%%%%%%%%%%%%%%%%%%%%%%%%%%%%%%%%%%%%%%%%%%%%%%%%%%%%%%%%%%%%%%%%%%%%%%%%%%%%%%%%%%%%%%%%%%%%%%%%%%%%%%%%%%%%%%%%%%%%%%%%%%%%%%%%%%%%%%%%%%%%%%%%%%%%%%%%%%%%%%%%%%%%%%%%%%%%%%%%%%%%%%%%%%%%%%
\begin{table}[ptbh]
\begin{center}
\begin{tabular}{ccccc}
%%%%%%%%%%%%%%%%%%%%%%%%%%%%%%%%%%%%%%%%%%%%%%%%%%%%%%%%%%%%%%%%%%%%%%%%%%%%%%%%%%%%%%%%%%%%%%%%%%%%%%%%%%%%%5
\hline \hline
\multicolumn{5}{c}{$K=1$}\\
\hline
Test$\setminus \alpha$ & 0.1& 0.05 & 0.01 & 0.005  \\
\hline
$T_{el}$&12.7205 &14.2748 &18.0389 &19.8171 \\
$T_{pr}$&2.2581 &2.4474 &2.8415 &2.9739 \\
$T_{LR}$&215.7571 &223.4439 &239.7306 &245.2959 \\
\hline
\multicolumn{5}{c}{$K=4$}\\
\hline
Test$\setminus \alpha$ & 0.1& 0.05 & 0.01 & 0.005  \\
\hline
$T_{el}$&12.9347 &14.3680 &17.8800 &19.4064  \\
$T_{pr}$&2.2821 &2.4524 &2.8275 &3.0234 \\
$T_{LR}$&216.2087 &223.3710 &236.5821 &242.6313  \\
\hline \hline
\end{tabular}
\end{center}
\caption{Critical values of the $T_{el}$, $T_{pr}$, and $T_{LR}$ tests for $\alpha \in \{0.1,0.05,0.01,0.05\}$. We put $p=20$, $T=104$, and $K\in \{1,4\}$.}%
\label{tab1}%
\end{table}

Using $p=20$, $T=104$ as well as $K=1$ for one-factor models and $K=4$ for four-factors models, the critical values of the considered test are calibrated by generating a sample of size $10^5$ from the inverse Wishart distribution with $T-K+p+1$ degrees of freedom and the identity parameter matrix. These critical values are shown in Table \ref{tab1}. The resulting samples of test statistics are used in the determination of the empirical distribution functions of the test statistics which are then applied to the calculation of the $p$-values. The most important quantiles of the obtained $p$-values, namely the minimum and the maximum values, the lower and the upper quartiles as well as the median, are shown in Table \ref{tab2}. Here, we observe that most of the calculated $p$-values are equal to zero which shows that the null hypothesis of the validity of a factor model with the selected factors is rejected in most cases for both, $K=1$ and $K=4$. Only the $T_{el}$ test fails to reject the null hypothesis in a few cases, which is in-line with the results of the previous section where it is shown that this test is less powerful in many cases.

%%%%%%%%%%%%%%%%%%%%%%%%%%%%%%%%%%%%%%%%%%%%%%%%%%%%%%%%%%%%%%%%%%%%%%%%%%%%%%%%%%%%%%%%%%%%%%%%%%%%%%%%%%%%%%%%%%%%%%%%%%%%%%%%%%%%%%%%%%%%%%%%%%%%%%%%%%%%%%%%%%%%%%%%%%%%%%%%%%%%%%%%%%%%%%%%%%%%%%%%%
\begin{table}[ptbh]
\begin{center}
\begin{tabular}{cccccc}
%%%%%%%%%%%%%%%%%%%%%%%%%%%%%%%%%%%%%%%%%%%%%%%%%%%%%%%%%%%%%%%%%%%%%%%%%%%%%%%%%%%%%%%%%%%%%%%%%%%%%%%%%%%%%5
\hline \hline
\multicolumn{6}{c}{$K=1$}\\
\hline
Test$\setminus$ Quantile & Minimum &Lower Quartile& Median & Upper Quartile & Maximum  \\
\hline
$T_{el}$&0.0000 &0.0000 &0.0000 &0.0002 &0.7296  \\
$T_{pr}$&0.0000 &0.0000 &0.0000 &0.0000 &0.0270  \\
$T_{LR}$&0.0000 &0.0000 &0.0000 &0.0000 &0.0000 \\
\hline
\multicolumn{6}{c}{$K=4$}\\
\hline
Test$\setminus$ Quantile & Minimum &Lower Quartile& Median & Upper Quartile & Maximum  \\
\hline
$T_{el}$&0.0000 &0.0002 &0.0002 &0.0011 &0.9241  \\
$T_{pr}$&0.0000 &0.0000 &0.0000 &0.0000 &0.1024  \\
$T_{LR}$&0.0000 &0.0000 &0.0000 &0.0000 &0.0000 \\
\hline \hline
\end{tabular}
\end{center}
\caption{Quantiles of the $p$-values calculated from the empirical distribution functions $T_{el}$, $T_{pr}$, and $T_{LR}$ with $p=20$, $T=104$, and $K\in \{1,4\}$.}%
\label{tab2}%
\end{table}

In order to get a better understanding of the obtained results, we also plot the histograms for the values of the test statistics in Figure \ref{Fig:DAX} for $K=1$ (left hand-side plots) and for $K=4$ (right hand-side plots). Here, we observe that most of the values are much larger than the corresponding critical values presented in Table \ref{tab1}.

\begin{center} Figure 8 above here \end{center}

\subsection{Analysis of Stocks Included into the S\&P 500 Index}

In this subsection, we perform an analysis similar to the one provided in Section 7.1. In contrast to the models from Section 7.1, however, high-dimensional factor models are considered. These models are applied to model the dynamics in $100$ returns on stocks included into the S\&P 500 index where $100$ stocks are chosen randomly out of $500$ stocks included into the S\&P 500 index. As a result, $10^4$ models are fitted for which the high-dimensional tests of Section 5 are performed. We consider two types of factor models with one factor, the return of the S\&P 500 index, and nine factors (the S\&P 500 index, the NASDAQ-100, the NASDAQ bank index, the NASDAQ Composite index, the NASDAQ Biotechnology index, the NASDAQ Industrial index, the NASDAQ Transportation index, the NASDAQ Computer index, and the NASDAQ Telecommunications index). The weekly data are taken from the 11th of June, 2004 to the 10th of June, 2014 ($T=518$) from the Yahoo! finance web-page.

%%%%%%%%%%%%%%%%%%%%%%%%%%%%%%%%%%%%%%%%%%%%%%%%%%%%%%%%%%%%%%%%%%%%%%%%%%%%%%%%%%%%%%%%%%%%%%%%%%%%%%%%%%%%%%%%%%%%%%%%%%%%%%%%%%%%%%%%%%%%%%%%%%%%%%%%%%%%%%%%%%%%%%%%%%%%%%%%%%%%%%%%%%%%%%%%%%%%%%%%%
\begin{table}[ptbh]
\begin{center}
\begin{tabular}{ccccc}
%%%%%%%%%%%%%%%%%%%%%%%%%%%%%%%%%%%%%%%%%%%%%%%%%%%%%%%%%%%%%%%%%%%%%%%%%%%%%%%%%%%%%%%%%%%%%%%%%%%%%%%%%%%%%5
\hline \hline
\multicolumn{5}{c}{$K=1$}\\
\hline
Test$\setminus \alpha$ & 0.1& 0.05 & 0.01 & 0.005  \\
\hline
$T_{el}$&17.4888 &18.9975 &22.4416 &23.5609 \\
$T_{pr}$&3.6521 &3.9673 &4.6190 &4.9366 \\
$T_{LR}$&1.2979 &1.6562 &2.3115 &2.5480 \\
\hline
\multicolumn{5}{c}{$K=9$}\\
\hline
Test$\setminus \alpha$ & 0.1& 0.05 & 0.01 & 0.005  \\
\hline
$T_{el}$&17.6366 &19.1353 &22.5938 &23.7658 \\
$T_{pr}$&3.6266 &3.9389 &4.5929 &4.8550 \\
$T_{LR}$&1.2746 &1.6037 &2.3308 &2.5761 \\
\hline \hline
\end{tabular}
\end{center}
\caption{Critical values of the $T_{el}$, $T_{pr}$, and $T_{LR}$ tests for $\alpha \in \{0.1,0.05,0.01,0.05\}$. We put $p=100$, $T=518$, and $K\in \{1,9\}$.}%
\label{tab3}%
\end{table}

In Table \ref{tab3}, we show the critical values of the considered tests which are calculated via simulations based on $10^5$ independent samples from the inverse Wishart distribution. The resulting samples of the test statistics are used to determine the corresponding empirical distribution functions which are then applied to the calculation of the $p$-values. The most important quantiles of the obtained $p$-values are shown in Table \ref{tab4}. In contrast to Section 7.1, here all maxima of $p$-values equal zero, meaning that the null hypothesis of the validity of the considered factor models are rejected by all tests in all of the considered cases. %%%%%%%%%%%%%%%%%%%%%%%%%%%%%%%%%%%%%%%%%%%%%%%%%%%%%%%%%%%%%%%%%%%%%%%%%%%%%%%%%%%%%%%%%%%%%%%%%%%%%%%%%%%%%%%%%%%%%%%%%%%%%%%%%%%%%%%%%%%%%%%%%%%%%%%%%%%%%%%%%%%%%%%%%%%%%%%%%%%%%%%%%%%%%%%%%%%%%%%%
\begin{table}[ptbh]
\begin{center}
\begin{tabular}{cccccc}
%%%%%%%%%%%%%%%%%%%%%%%%%%%%%%%%%%%%%%%%%%%%%%%%%%%%%%%%%%%%%%%%%%%%%%%%%%%%%%%%%%%%%%%%%%%%%%%%%%%%%%%%%%%%%5
\hline \hline
\multicolumn{6}{c}{$K=1$}\\
\hline
Test$\setminus$ Quantile & Minimum &Lower Quartile& Median & Upper Quartile & Maximum  \\
\hline
$T_{el}$&0.0000 &0.0000 &0.0000 &0.0000 &0.0000  \\
$T_{pr}$&0.0000 &0.0000 &0.0000 &0.0000 &0.0000  \\
$T_{LR}$&0.0000 &0.0000 &0.0000 &0.0000 &0.0000 \\
\hline
\multicolumn{6}{c}{$K=9$}\\
\hline
Test$\setminus$ Quantile & Minimum &Lower Quartile& Median & Upper Quartile & Maximum  \\
\hline
$T_{el}$&0.0000 &0.0000 &0.0000 &0.0000 &0.0000  \\
$T_{pr}$&0.0000 &0.0000 &0.0000 &0.0000 &0.0000  \\
$T_{LR}$&0.0000 &0.0000 &0.0000 &0.0000 &0.0000 \\
\hline \hline
\end{tabular}
\end{center}
\caption{Quantiles of the $p$-values calculated from the empirical distribution functions $T_{el}$, $T_{pr}$, and $T_{LR}$ with $p=20$, $T=104$, and $K\in \{1,9\}$.}%
\label{tab4}%
\end{table}

In Figure \ref{Fig:SP}, we also plot the histograms for the values of the test statistics in case of $K=1$ (left-hand side plots) and $K=9$ (right-hand side plots). The histograms document that the values of the calculated test statistics are much larger than the critical values presented in Table \ref{tab3}. These findings do not support the hypothesis that the linear dependencies between the asset returns can be fully explained by the selected factors.

\begin{center} Figure 9 above here \end{center}

\section{Summary}

Factor models of both small and large dimensions are a very attractive and popular modeling device nowadays. They are applied in different fields of science, like econometrics, economics, finance, biology, psychology, etc. While a lot of papers are devoted to the estimation of the parameters of factor models as well as to the determination of the number of factors, testing the validity of factor models has not been discussed widely in literature up to now. A notable exception is the test on the CAPM in low dimensions which is a special case of factor models.

In the present paper, we derive exact and asymptotic tests on the validity of factor models when the factors are observable. The results are obtained for both small-dimensional and high-dimensional factor models. The distributions of the suggested test statistics are derived under the assumption of normality and it is shown that they are independent of the diagonal elements of the precision matrix constructed from the dependent variables and factors. In order to investigate the powers of the considered tests, an extensive simulation study is performed. Its conclusion is that none of the tests performs uniformly better than the others and, consequently, the application of each test depends on the deviations to be detected under the alternative hypothesis. Finally, we apply the theoretical results of the paper in two empirical studies where factor models with different number of factors are fitted to the returns on stocks included into the DAX as well as the S\&P index. Our empirical results do not support the hypothesis that all linear dependencies between the returns can be entirely captured by the considered factors. As a result, the factor models, which are based on the considered market indices, are not in general valid in practice and the investor can apply them with care only because they are not able to explain all linear dependencies between the asset returns.

It is remarkable that the tests suggested in the paper are also distribution-free for a large class of matrix-variate distributions. For instance, an application of Theorem 5.12 in \citet{GuptaVargaBodnar2013} shows that the distribution of the considered test statistics is the same if data follow a matrix-variate elliptically contoured distribution. This family of distributions includes plenty of well-known models, like the normal distribution, mixture of normal distributions, the multivariate $t$-distribution, Pearson types II and VII distributions (see \citet{GuptaVargaBodnar2013}). Elliptically contoured distributions have been already applied in portfolio theory. \citet{owen1983class} extend Tobin's separation theorem and Bawa's rules of ordering certain prospects to elliptically contoured distributions. \citet{chamberlain1983characterization} shows that elliptical distributions imply mean-variance utility functions, whereas \citet{berk1997necessary} argues that one of the necessary conditions for the CAPM is an elliptical distribution for the asset returns. Moreover, \citet{zhou1993asset} generalizes the test of \citet{gibbons1989test} on the efficiency of a given portfolio to elliptically distributed returns. \citet{hodgson2002testing} propose a test for the CAPM under elliptical assumptions (see, also the textbook of \citet{GuptaVargaBodnar2013} for further results and applications to financial data). Finally, we point out that, since in the derivation of the high-dimensional asymptotic distributions of the test statistics their finite sample distributions are used, the above result holds true for both, low-dimensional and high-dimensional factor models.

The suggested tests and their distributions are derived under the assumption that the factors are observable which is motivated by the application of the CAPM and the APT. An important question is how to extend the suggested testing procedures to the case when the factors are unobservable, especially, when the number of factors is unknown as well. It is noted that the unknown factors can be estimated very accurately in high dimensions as shown in \citet{BaiNg2002determining} and \citet{baing2013}. Consequently, the estimation of unknown factors is expected to have no large impact on the testing procedures suggested in the paper. The above two generalizations of our results are very attractive both, from a theoretical and a practical point of view and they will be treated in a consequent paper.

\section{Appendix}

In this section the proofs of lemmas and theorems are given.

\subsection*{Proof of Lemma \ref{lem1}}
\begin{proof}
In the proof we deal with the case $i=1$ only and note that the other equalities can be derived similarly. Let $\bA$ and $\bB$ be partitioned as
\[
\bA=\left(
      \begin{array}{cc}
        a_{11} & \ba_{12} \\
        \ba_{12}^\top & \bA_{22} \\
      \end{array}
    \right)
\quad \text{and} \quad
\bB=\bA^{-1}=\left(
      \begin{array}{cc}
        b_{11} & \bb_{12} \\
        \bb_{12}^\top & \bB_{22} \\
      \end{array}
    \right)
\]

The application of the inverse formula for the partitioned matrix (see Theorem 8.5.11 of \citet{Harville1997}) yields
\[
b_{11}=a_{11}^{-1}+a_{11}^{-1}\ba_{12}\left(\bA_{22}-\frac{\ba_{12}^\top\ba_{12}}{a_{11}}\right)^{-1}\ba_{12}^\top a_{11}^{-1}
\]

Since $\left(\bA_{22}-\frac{\ba_{12}^\top\ba_{12}}{a_{11}}\right)$ is positive definite, we deduce that its inverse is positive definite and, hence,
\[a_{11}^{-1}\ba_{12}\left(\bA_{22}-\frac{\ba_{12}^\top\ba_{12}}{a_{11}}\right)^{-1}\ba_{12}^\top a_{11}^{-1} \ge 0\,,\]
i.e., $b_{11}\ge a_{11}^{-1}$, where the equality is present only if $\ba_{12}=\mathbf{0}$.
\end{proof}

\vspace{0.5cm}
In the proofs of Theorems \ref{th1} and \ref{th2}, we use the result of Lemma \ref{lem2}. In the following we consider several partitions of $\bV_{11}$ defined in (\ref{part_sampl_prec}) which are constructed with respect
to its diagonal elements. In case of the first diagonal elements we get
\begin{equation}\label{part_sampl_V11}
\bV_{11}=\left(
  \begin{array}{cc}
    v_{11} & \bv_{12,1} \\
    \bv_{21,1} & \bV_{22,1} \\
  \end{array}\right)\,,
\end{equation}
whereas for the $j$-th diagonal element, a similar partition is considered where the vector $\bv_{21,j}$ is
obtained by deleting the $j$th element form the $j$th column of $\bV_{11}$ and $\bV_{22,j}$ is calculated by
deleting the $j$th column and the $j$th row of $\bV_{11}$.

Let $\bOmega_{11}$ be partitioned similar to (\ref{part_sampl_prec}) whose elements we denote by $\omega_{jj}$, $\bomega_{21,j}$, and $\bOmega_{22,j}$ for $j = 1,\ldots, p$. Next, we consider the test statistic
\begin{equation}\label{test_stat_col}
Z_j = \frac{T-p-K+1}{q} \frac{\bv^\top_{21,j}\bL^\top (\bL\bQ_{j}\bL^\top)^{-1}\bL\bv_{21,j}}{v_{jj}}
\end{equation}
for $j = 1,\ldots, p$ with $\bQ_j =\bV_{22,j}-\bv_{21,j}\bv_{21,j}^\top/v_{jj}$ in order to test the hypotheses
\begin{equation}\label{test_col}
H_{0,j}: \bL \bomega_{21,j}= \mathbf{0} \qquad \text{versus} \qquad
H_{1,j}: \bL \bomega_{21,j}=\bd_j \neq \mathbf{0} \qquad \text{for} \qquad
 j = 1,\ldots,p \,,
\end{equation}
where $\bL: q\times (p-1)$ is a matrix of constants.

Both test statistics $T_{ij}$ and $T_j$ for $j=1,\ldots,p$ and $1 \le j <i \le p$ can be obtained from $Z_j$ for some choices of the matrix $\bL$. Later on, we make use of this result for proving Theorems \ref{th1a} and \ref{th1}.

In Lemma 2, the distribution of $Z_j$ is derived under both the null and the alternative hypotheses.\\[0.2cm]

\begin{lemma}\label{lem2}
Let $\bX_t$ follow model (\ref{fm}) where $\bff_t$ and $\bu_t$ are independent and normally distributed. Then:
\begin{enumerate}[(a)]
\item The density of $Z_j$ is given by
\begin{eqnarray*}
F_{Z_j}(x)&=& F_{q,T-K-p+1}(x) (1+\lambda_j)^{-(T-K-p+1+q)/2}\\
&\times& \,_2F_1\Big(\frac{T-K-p+1+q}{2},\frac{T-K-p+1+q}{2},\frac{q}{2};
\frac{qx}{T-K-p+1+qx} \; \frac{\lambda_j}{1+\lambda_j}  \Big) \,,
\end{eqnarray*}
where $\lambda_j = \omega_{jj}^{-1}\bd_j^\top (\bL\bXi_j \bL^\top )^{-1}\bd_j$ with $\bXi_j = \bOmega_{22,j} -\bomega_{21,j}\bomega^\top_{21,j}/\omega^\top_{11,j}$.
\item Under $H_{0_j}$ it holds that $Z_j \sim \mathcal{F}_{p-1,T-K-p+1}$.
\end{enumerate}
\end{lemma}

\begin{proof}
\begin{enumerate}[(a)]
\item We consider
\begin{equation}\label{app_Z_j}
Z_j=\frac{T-K-p+1}{q} \frac{\omega_{jj}\left(\bL\frac{\bv^\top_{21,j}}{v_{jj}}\right)^\top (\bL\bQ_{j}\bL^\top)^{-1}\left(\bL\frac{\bv^\top_{21,j}}{v_{jj}}\right)}
{\omega_{jj}/v_{jj}}
\end{equation}

From the proof of Theorem 3 in \citet{BodnarOkhrin2008} we get that
\begin{equation*}
\frac{\bv_{21,j}}{v_{jj}}|\bQ_j=\bD \sim  \mathcal{N}_{p-1}\left(\frac{\bomega_{21,j}}{\omega_{jj}},\omega_{jj}^{-1} \bD\right)
\end{equation*}
and, consequently,
\begin{equation*}
\omega_{jj}\left(\bL\frac{\bv^\top_{21,j}}{v_{jj}}\right)^\top (\bL\bQ_{j}\bL^\top)^{-1}\left(\bL\frac{\bv^\top_{21,j}}{v_{jj}}\right) | (\bL\bQ_j\bL^\top)^{-1}=\bC \sim \chi^2_{q,\lambda_j(\bC)}
\end{equation*}
with $\lambda_j(\bC)=\omega_{jj}^{-1} \bd_j^\top \bC \bd_j$ which is independent of $v_{jj}$ (see, e.g., Theorem 3 in \citet{BodnarOkhrin2008}). Furthermore, it holds that $v_{jj} \sim \mathcal{W}^{-1}_1(T+p+K+1-2(p+K-1),\omega_{jj})$ and, hence,
\[ \frac{\omega_{jj}}{v_{jj}} \; \sim \; \chi^2_{T-K-p+1} \; . \]
Putting these results together we get
\[ Z_j|(\bL\bQ_j\bL^\top)^{-1}=\bC \; \sim \; \mathcal{F}_{q,T-K-p+1,\lambda(\bC)} \;
. \]

Because $(\bL\bQ_j\bL^\top)^{-1} \sim \mathcal{W}_{q}(T-K-p+1+q,(\bL\bXi_j\bL^\top)^{-1})$, we get
\[F_{Z_j}(x) = \int_{{\bf C}>0} F_{q,T-K-p+1,\lambda_j(\bC)}(x) W_{q}(T-K-p+1+q,(\bL\bXi_j\bL^\top)^{-1})(\bC) d\bC \; , \]
where $F_{i,j,\lambda}$ denotes the density of the non-central $\mathcal{F}$-distribution with degrees $i$ and $j$ and noncentrality parameter $\lambda$; $W_{q}(i,\bol{\Lambda})$ stands for the density of the $q$-dimensional Wishart distribution with degrees $i$ and covariance matrix $\bol{\Lambda}$. If $\lambda=0$ we briefly write $F_{i,j}$. It holds that (e.g., Theorem 1.3.6 of \citet{Muirhead})
\begin{eqnarray*}
F_{q,T-K-p+1,\lambda(\mathbf{C})}(x)& = & F_{q,T-K-p+1}(x) \,
\exp{\left\{-\frac{\lambda(\mathbf{C})}{2}\right\}} \frac{\Gamma\left(q/2\right)}{\Gamma\left\{(T-K-p+1+q)/2\right\}}\\
&\times& \sum_{i=0}^\infty
\frac{\Gamma\left\{(T-K-p+1+q)/2+i\right\}}{\Gamma\left(q/2+i\right)}\frac{\lambda(\mathbf{C})^i}{i!}\left\{\frac{qx}{2(T-K-p+1+qx)}
 \right\}^i \, .
\end{eqnarray*}
Let us denote
\[k(i)\,=\, \frac{1}{i!}\frac{\Gamma\left\{(T-K-p+1+q)/2+i\right\}}{\Gamma\left\{(T-K-p+1+q)/2\right\}}\frac{\Gamma\left(q/2\right)}{\Gamma\left(q/2+i\right)}
\left\{\frac{qx}{2(T-K-p+1+qx)} \right\}^i.\]
Using the notation ${\rm etr}(\mathbf{A})=\exp({\rm trace}(\mathbf{A}))$ for a square matrix $\mathbf{A}$, we get
\begin{eqnarray*}
F_{Z_j}(x) &=& F_{q,T-K-p+1}(x) \, \sum_{i=0}^\infty k(i)
\int_{\mathbf{C}>0} \lambda_j(\mathbf{C})^i
\exp{\left\{-\frac{\lambda_j(\mathbf{C})}{2}\right\}} \, \frac{1}{2^{q(T-K-p+1+q)/2}
\Gamma_{q}\left(\frac{T-K-p+1+q}{2}\right)}\\
&\times& |\bL\bXi_j\bL^\top|^{\frac{T-K-p+1+q}{2}}
|\mathbf{C}|^{\frac{T-K-p}{2}} {\rm etr} \left\{-\frac{1}{2}(\bL\bXi_j\bL^\top)\mathbf{C}\right\}
d\mathbf{C}\\
 &=& F_{q,T-K-p+1}(x) \, \sum_{i=0}^\infty k(i)
\int_{\mathbf{C}>0}  \,|\bL\bXi_j\bL^\top|^{\frac{T-K-p+1+q}{2}} \frac{1}{2^{q(T-K-p+1+q)/2}
\Gamma_{q}(\frac{T-K-p+1+q}{2})}\\
&\times&  |\mathbf{C}|^{\frac{T-K-p}{2}}\left(\omega_{jj}^{-1}\bd^\top_j \mathbf{C}\bd_j\right)^i
{\rm etr}\left\{-\frac{1}{2}(\bL\bXi_j\bL^\top+ \omega_{jj}^{-1}\bd_j \bd^\top_j)\mathbf{C}\right\}d\mathbf{C}\\
&=& F_{q,T-K-p+1}(x) \, |\bL\bXi_j\bL^\top|^{\frac{T-K-p+1+q}{2}}
|\bL\bXi_j\bL^\top+ \omega_{jj}^{-1}\bd_j \bd^\top_j|^{-\frac{T-K-p+1+q}{2}}
 \\
&\times& \sum_{i=0}^\infty k(i) \omega_{jj}^{-i} \;
{\rm E}\left\{(\bd^\top_j \mathbf{\tilde{C}} \bd_j)^i\right\} \ ,
\end{eqnarray*}
where $\mathbf{\tilde{C}}\sim \mathcal{W}_{q}(T-K-p+1+q,(\bL\bXi_j\bL^\top+ \omega_{jj}^{-1}\bd_j \bd^\top_j)^{-1})$. From
Theorem 3.2.8 of \citet{Muirhead} we obtain that
\begin{eqnarray*}
{\rm E}\left\{(\bd^\top_j \mathbf{\tilde{C}}\bd_j)^i\right\}
 &=& 2^i \, \frac{\Gamma\left\{(T-K-p+1+q)/2+i\right\}}{\Gamma\left\{(T-K-p+1+q)/2\right\}} \,(\bd^\top_j \{\bL\bXi_j\bL^\top+ \omega_{jj}^{-1}\bd_j \bd^\top_j\}^{-1}\bd_j)^i\\
&=& 2^i \, \frac{\Gamma\left\{(T-K-p+1+q)/2+i\right\}}{\Gamma\left\{(T-K-p+1+q)/2\right\}} \left\{\frac{\bd^\top_j (\bL\bXi_j\bL^\top)^{-1}\bd_j}{1+\omega_{jj}^{-1}\bd^\top_j (\bL\bXi_j\bL^\top)^{-1}\bd_j}\right\}^i \,.
\end{eqnarray*}

Finally,
\begin{eqnarray*}
F_{Z_j}(x) &=& F_{q,T-K-p+1}(x) (1+\omega_{jj}^{-1}\bd^\top_j (\bL\bX_j\bL^\top)^{-1}\bd_j)^{-(T-K-p+1+q)/2}\\
& \times & \frac{\Gamma\left(q/2\right)}{\Gamma\left\{(T-K-p+1+q)/2\right\} \Gamma\left\{(T-K-p+1+q)/2\right\}}\\
&\times&\sum_{i=0}^\infty \frac{\Gamma\left\{(T-K-p+1+q)/2+i\right\} \Gamma\left\{(T-K-p+1+q)/2+i\right\}}{i!\Gamma\left(q/2+i\right)}\\
&\times& \left\{\frac{qx\omega_{jj}^{-1}\bd^\top_j (\bL\bXi_j\bL^\top)^{-1}\bd_j}{(T-K-p+1+qx)(1+\omega_{jj}^{-1}\bd^\top_j (\bL\bXi_j\bL^\top)^{-1}\bd_j)} \right\}^i \\
&=& F_{q,T-K-p+1}(x) (1+\lambda_j)^{-(T-K-p+1+q)/2}\\
&\times& \,_2F_1\Big(\frac{T-K-p+1+q}{2},\frac{T-K-p+1+q}{2},\frac{q}{2};
\frac{qx}{T-K-p+1+qx} \; \frac{\lambda_j}{1+\lambda_j}  \Big) \, .
\end{eqnarray*}
The result is proved.

\item The statement follows by noting that $\lambda_j=0$ under $H_{0,j}$ and
$$_2F_1\Big(\frac{T-K-p+1+q}{2},\frac{T-K-p+1+q}{2},\frac{q}{2};0 \Big)=1.$$
\end{enumerate}
\end{proof}

\subsection*{Proof of Theorem \ref{th1a}}
\begin{proof}
The proof is based on the observation that the test statistic $T_{ij}$ for each $1\le j<i \le p$ can be presented as $Z_j$ from \eqref{app_Z_j} with $q=1$ and $\bL=(0,\ldots,0,1,0,\ldots,0)$ (the vector of zeros with exception of the $(i-1)$-th element which is one). In order to show this, we consider
\[\bL\bv_{21,j}=v_{ij}
\quad \text{and} \quad
(\bL\bQ_{j}\bL^\top)^{-1}=v_{ii}-\frac{v_{ij}^2}{v_{jj}} \,.\]

Hence,
\begin{equation*}
Z_j=\frac{T-K-p+1}{1}\frac{v_{ij}^2}{v_{jj}\left(v_{ii}-\frac{v_{ij}^2}{v_{jj}}\right)}=T_{ij}
\end{equation*}
and an application of Lemma 2 leads to the statement of Theorem \ref{th1a} with
\[\lambda_{ij}=\frac{d_{ij}^2}{\omega_{jj}\left(\omega_{ii}-d_{ij}^2/\omega_{jj}\right)}\,.\]
\end{proof}

\subsection*{Proof of Theorem \ref{th1}}
\begin{proof}
For the $j$-th test statistic with $\bL=\bI_{p-1}$ we get
\begin{equation}\label{test_stat_col}
Z_j = \frac{T-K-p+1}{q} \frac{\bv^\top_{21,j}\bL^\top (\bL\bQ_{j}\bL^\top)^{-1}\bL\bv_{21,j}}{v_{jj}}
=\frac{T-p-K+1}{q}v_{jj}(v_{jj}^{(-)}-v_{jj}^{-1})=T_j\,,
\end{equation}
where $v_{jj}^{(-)}$ stands for the $j$-th diagonal element of $\bV_{11}^{-1}$ and the second equality is obtained from (see, Theorem 8.5.11 of
\citet{Harville1997})
\[v_{jj}^{(-)}=v_{jj}^{-1}+v_{jj}^{-1}\bv^\top_{21,j}\bL^\top (\bL\bQ_{j}\bL^\top)^{-1}\bL\bv_{21,j}v_{jj}^{-1}\,.\]

The rest of the proof follows from Lemma \ref{lem1} with
\[\lambda_{j}=(\omega_{jj}^{(-)}\omega_{jj}-1)\,.\]
\end{proof}

\subsection*{Proof of Theorem \ref{th2}}

\begin{proof}
Let $\bD={\rm diag}(\omega_{11},\ldots,\omega_{pp})$ and $\bD_j={\rm diag}(\omega_{11},\ldots,\omega_{j-1,j-1},\omega_{j+1,j+1},\ldots,\omega_{pp})$. We consider
\[\bV_{11}^*=\bD^{-1/2}\bV_{11}\bD^{-1/2} \sim \mathcal{W}_p^{-1}(T-K+p+1,\bI)\,.\]
Then, it holds that
\begin{eqnarray*}
v^*_{11,j}&=&\frac{v_{jj}}{\omega_{jj}}, \quad \bv_{21,j}^*=\omega_{jj}^{-1/2}\bD_j^{-1/2}\tilde{\bv}_{21,j} \,,\\
\bQ_j^*&=&\bV_{22,j}^*-\frac{\bv_{21,j}^*(\bv_{21,j}^*)^\top}{v^*_{11,j}}
=\bD_j^{-1/2}\bV_{22,j}\bD_j^{-1/2}-\frac{\bD_j^{-1/2}\bv_{21,j}\bv_{21,j}^\top\bD_j^{-1/2}/\omega_{jj}}{v_{jj}/\omega_{jj}}\\
&=&\bD_j^{-1/2}\bQ_j\bD_j^{-1/2}\,.
\end{eqnarray*}

Hence,
\begin{eqnarray*}
Z_j^*&=& \frac{T-K-p+1}{p-1} \frac{(\bv_{21,j}^*)^\top(\bQ_{j}^*)^{-1}\bv_{21,j}^*}{v_{jj}^*}\\
&=& \frac{T-K-p+1}{p-1} \frac{\bv_{21,j}^\top\bD_j^{-1/2}\bD_j^{1/2}\bQ_{j}^{-1}\bD_j^{1/2}\bD_j^{-1/2}\bv_{21,j}/\omega_{jj}}{v_{jj}^*/\omega_{jj}}
=Z_j\,.
\end{eqnarray*}
As the joint distribution of $(Z_1^*,\ldots,Z_p^*)^\top$ is fully determined by the distribution of $\bV^*_{11}$ which does not depend on $\omega_{11},\ldots,\omega_{pp}$, and as the distribution of $(Z_1,\ldots,Z_p)^\top$ coincides with the distribution of $(Z_1^*,\ldots,Z_p^*)^\top$, we get that the distribution of $(Z_1,\ldots,Z_p)^\top$ is independent of $\omega_{11},\ldots,\omega_{pp}$. Finally, noting that the distribution of $T_{pr}$ is fully determined by the distribution of $(Z_1,\ldots,Z_p)^\top$, the statement of the theorem follows.
\end{proof}

\subsection*{Proof of Theorem \ref{th3a}}

\begin{proof}
The results of Theorem \ref{th3a}.(a) follows directly from Theorem \ref{th1a} and the fact that $\chi^2_q/q \stackrel{a.s.}{\longrightarrow} 1$ as $q\rightarrow \infty$.

Next we prove the statement of Theorem \ref{th3a}.(b). It holds that
\[
\sqrt{T_{ij}}=\sqrt{T-K-p+1}\frac{\sqrt{v_{jj}}}{\sqrt{\omega_{jj}}}\sqrt{\omega_{jj}}\frac{v_{ij}/v_{jj}}{\sqrt{v_{ii}-v_{ij}^2/v_{jj}}}\,.
\]

From the proof of Lemma \ref{lem2} we get
\[
\frac{v_{ij}}{v_{jj}}\Big|\left(v_{ii}-v_{ij}^2/v_{jj}\right) \sim \mathcal{N}\left(\frac{\omega_{ij}}{\omega_{jj}},\omega_{jj}^{-1}\left(v_{ii}-v_{ij}^2/v_{jj}\right)\right)
\]
and, hence,
\[
\sqrt{\omega_{jj}}\frac{v_{ij}/v_{jj}}{\sqrt{v_{ii}-v_{ij}^2/v_{jj}}}-\frac{\sqrt{\omega_{jj}}}{\sqrt{v_{ii}-v_{ij}^2/v_{jj}}}\frac{\omega_{ij}}{\omega_{jj}}
\Big|\left(v_{ii}-v_{ij}^2/v_{jj}\right) \sim \mathcal{N}\left(0,1\right)\,.
\]

Since the conditional distribution given in the last equation does not depend on the condition $\left(v_{ii}-v_{ij}^2/v_{jj}\right)$ it is also the unconditional distribution of the difference. Moreover, following the proof of Lemma \ref{lem2} we get
\[
\frac{\omega_{jj}}{v_{jj}}\sim \chi^2_{T-K-p+1}
\quad \text{and} \quad
\frac{\omega_{ii}-\omega_{ij}^2/\omega_{jj}}{v_{ii}-v_{ij}^2/v_{jj}}\sim \chi^2_{T-K-p+2}\,.
\]

Hence,
\[
\frac{1}{T-K-p+1}\frac{\omega_{jj}}{v_{jj}} \stackrel{a.s.}{\longrightarrow} 1
\quad \text{and} \quad
\frac{1}{T-K-p+2}\frac{\omega_{ii}-\omega_{ij}^2/\omega_{jj}}{v_{ii}-v_{ij}^2/v_{jj}} \stackrel{a.s.}{\longrightarrow} 1
\]
as $T-K-p+1 \longrightarrow \infty$. This leads to
\[\sqrt{T_{ij}}-\sqrt{T-K-p+2}\sqrt{\lambda_{ij}} \stackrel{d.}{\longrightarrow} \mathcal{N}(0,1)\]
and, hence,
\[\left(\sqrt{T_{ij}}-\sqrt{T-K-p+2}\sqrt{\lambda_{ij}}\right)^2 \stackrel{d.}{\longrightarrow} \chi^2_{1}\,.\]

The result in case of $T-K-p \longrightarrow d \in (0,\infty)$ is obtained in the same way.
\end{proof}

\subsection*{Proof of Theorem \ref{th3}}

\begin{proof}
\begin{enumerate}[(a)]
\item First, we consider the case $p/(T-K) \rightarrow c \in (0,1)$. Then it holds that
\[\sqrt{p-1}\left(T_j-1 \right)=\frac{\sqrt{p-1}\dfrac{\omega_{jj}\left(\frac{\bv^\top_{21,j}}{v_{jj}}\right)^\top \bQ_{j}^{-1}\left(\frac{\bv^\top_{21,j}}{v_{jj}}\right)}{p-1}-\sqrt{p-1}\dfrac{\omega_{jj}/v_{jj}}{T-K-p+1}}
{(\omega_{jj}/v_{jj})/(T-K-p+1)}\,.
\]

From the proof of Lemma \ref{lem2}, we get that $\omega_{jj}/v_{jj} \sim \chi^2_{T-K-p+1}$ and it is independent of
\[\omega_{jj}\left(\frac{\bv^\top_{21,j}}{v_{jj}}\right)^\top \bQ_{j}^{-1}\left(\frac{\bv^\top_{21,j}}{v_{jj}}\right) \sim \chi^2_{p-1}\,.\]
Hence, from the law of large numbers and the central limit theorem we get $(\omega_{jj}/v_{jj})/(T-K-p+1)\stackrel{a.s.}{\longrightarrow} 1$ as $T-K-p+1\rightarrow \infty$,
\begin{equation}\label{th3_as_den_H0}
\sqrt{T-K-p+1}\left(\frac{\omega_{jj}/v_{jj}}{T-K-p+1}-1\right)\stackrel{d.}{\longrightarrow} \mathcal{N}(0,2) ~~\text{for} ~~T-K-p\rightarrow \infty
\end{equation}
and
\begin{equation}\label{th3_as_num_H1}
\sqrt{p-1}\left(\dfrac{\omega_{jj}\left(\frac{\bv^\top_{21,j}}{v_{jj}}\right)^\top \bQ_{j}^{-1}\left(\frac{\bv^\top_{21,j}}{v_{jj}}\right)}{p-1}
-1\right) \stackrel{d.}{\longrightarrow} \mathcal{N}(0,2) ~~\text{for} ~~p\rightarrow \infty\,,
\end{equation}
as well as that both summands in the numerator are independent. Hence,
\begin{equation*}
\sqrt{p-1}\left(T_j-1\right)\stackrel{d.}{\longrightarrow} \mathcal{N}\left(0,\frac{2}{1-c}\right) ~~\text{for} ~~ \frac{p}{T-K} \rightarrow c \in (0,1) ~~\text{as}~~ T-K \rightarrow \infty \,.
\end{equation*}

Now, let $T-K-p \rightarrow d \in (0,\infty)$ as $T-K \rightarrow \infty$. Then, we get
\[\frac{\omega_{jj}\left(\frac{\bv^\top_{21,j}}{v_{jj}}\right)^\top \bQ_{j}^{-1}\left(\frac{\bv^\top_{21,j}}{v_{jj}}\right)}{p-1}
 \stackrel{a.s.}{\longrightarrow} 1 ~~ \text{for} ~~ p \rightarrow \infty\]
and
\[\frac{\omega_{jj}}{v_{jj}} \stackrel{d.}{\longrightarrow} \chi^2_{d+1} ~~ \text{for} ~~ T-K-p \rightarrow d \in (0,\infty)\,.\]
Putting these two results together we get the statement of the second part of Theorem \ref{th3}.(a).

\item The proof of Theorem \ref{th3}.(b) is achieved in the same way as the part (a) of this theorem. The only point which remains to be investigate is the asymptotic distribution of the numerator in the expression of $T_j$.

From the proof of Lemma \ref{lem2} we get
\begin{equation*}
\omega_{jj}\left(\frac{\bv^\top_{21,j}}{v_{jj}}\right)^\top \bQ_{j}^{-1}\left(\frac{\bv^\top_{21,j}}{v_{jj}}\right) | \bQ_j^{-1}=\bC \sim \chi^2_{p-1,\lambda_j(\bC)}\,,
\end{equation*}
where $\lambda_j(\bC)=\omega_{jj}^{-1} \bd_j^\top \bC \bd_j$. In the following we make use of:

\vspace{0.5cm}
\begin{lemma}\label{lem3}
Let $\mathbf{Y}=(Y_1,\ldots,Y_p)^\top\sim \mathcal{N}_p (\bol{\mu},\bI)$ with $\bol{\mu}=(\mu_1,\ldots,\mu_p)^\top$. Then for the random variable $Z^{(p)}=\mathbf{Y}^\top \mathbf{Y}\sim \chi^2_{p,\lambda}$ with $\lambda_p=\bol{\mu}^\top \bol{\mu}$ such that $\lim_{p\rightarrow \infty}\lambda_p/p < \infty$, we get
\begin{enumerate}[(a)]
\item
\begin{equation}
\frac{Z^{(p)}}{p}-1-\frac{\lambda_p}{p} \stackrel{a.s.}{\longrightarrow} 0 ~~ \text{for} ~~ p \rightarrow \infty \,.
\end{equation}
\item
\begin{equation}
\sqrt{p} \frac{\frac{Z^{(p)}}{p}-1-\frac{\lambda_p}{p}}{\sqrt{2\left(1+2\frac{\lambda}{p}\right)}} \stackrel{d.}{\longrightarrow} \mathcal{N}(0,1) ~~ \text{for} ~~ p \rightarrow \infty \,.
\end{equation}
\end{enumerate}
\end{lemma}

\vspace{0.5cm}
\begin{proof}
\begin{enumerate}[(a)]
\item It holds that
\begin{eqnarray*}
\frac{Z}{p}=\frac{1}{p}\sum_{i=1}^p Y_i^2=\frac{1}{p}\sum_{i=1}^p (Y_i-\mu_i)^2+\frac{1}{p}\sum_{i=1}^p (Y_i-\mu_i)\mu_i+\frac{\lambda}{p}\,.
\end{eqnarray*}
Since $Y_i-\mu_i\sim \mathcal{N}(0,1)$ from the law of large numbers we get that $\frac{1}{p}\sum_{i=1}^p (Y_i-\mu_i)^2 -1 \stackrel{a.s.}{\longrightarrow} 0$ as $p \rightarrow \infty$. Furthermore, it holds that $\sum_{i=1}^p (Y_i-\mu_i)\mu_i \sim \mathcal{N}(0,\lambda)$ and, consequently
\begin{equation*}
\frac{1}{p}\sum_{i=1}^p (Y_i-\mu_i)\mu_i \stackrel{a.s.}{\longrightarrow} 0 ~~\text{as}~~p \rightarrow \infty \,,
\end{equation*}
since $\lim_{p\rightarrow \infty}\lambda/p < \infty$. This completes the proof of the statement of Lemma \ref{lem3}.(a).

\item We get $Y_i^2 \sim \chi^2_{1,\mu_i^2}$, ${\rm E}(Y_i^2)=1+\mu_i^2$, ${\rm Var}(Y_i)=2(1+2\mu_i^2)$, and ${\rm E}(Y_i-1-\mu_i)^4=48(1+4 \mu_i^2)$. This leads to
\begin{equation}\label{CLT_Lyap}
\lim_{p \rightarrow \infty} \frac{\sum_{i=1}^p {\rm E}(Y_i-1-\mu_i)^4}{(\sum_{i=1}^p {\rm Var}(Y_i))^2}
=\lim_{p \rightarrow \infty} \frac{48 p\left(1+4\frac{\lambda}{p}\right)}{4p^2\left(1+2\frac{\lambda}{p}\right)^2}=0
\end{equation}
Then, an application of the Lyapunov central limit theorem (see, e.g., \citet[p. 362]{Billingsley1995}) gives
\begin{equation*}
\frac{\mathbf{Y}^\top \mathbf{Y}-p-\lambda}{\sqrt{2p\left(1+2\frac{\lambda}{p}\right)}}=
\frac{Z-p-\lambda}{\sqrt{2p\left(1+2\frac{\lambda}{p}\right)}}=
\sqrt{p} \frac{\frac{Z}{p}-1-\frac{\lambda}{p}}{\sqrt{2\left(1+2\frac{\lambda}{p}\right)}}
\stackrel{d.}{\longrightarrow}  \mathcal{N}(0,1)\,.
\end{equation*}
\end{enumerate}
\end{proof}

An application of Lemma \ref{lem3}.(b) leads to
\begin{equation*}
\sqrt{p-1}\frac{\frac{\omega_{jj}\left(\frac{\bv^\top_{21,j}}{v_{jj}}\right)^\top \bQ_{j}^{-1}\left(\frac{\bv^\top_{21,j}}{v_{jj}}\right)}{p-1}
-1-\frac{\lambda_j(\bQ_j^{-1})}{p-1}}
{\sqrt{2+4\frac{\lambda_j(\bQ_j^{-1})}{p-1}}} \stackrel{d.}{\longrightarrow}  \mathcal{N}(0,1)
\end{equation*}

Now, it holds that
\begin{eqnarray*}
\frac{\lambda_j(\bQ_j^{-1})}{p-1}=\frac{\omega_{jj}^{-1}\bd_j^\top \bXi_j^{-1} \bd_j}{p-1}\frac{\bd_j^\top \bQ_j^{-1} \bd_j}{\bd_j^\top \bXi_j^{-1} \bd_j}
=\frac{\lambda_j}{p-1}\frac{\bd_j^\top \bQ_j^{-1} \bd_j}{\bd_j^\top \bXi_j^{-1} \bd_j}\,.
\end{eqnarray*}
Because $\bQ_j^{-1} \sim \mathcal{W}_{p-1}(T-K,\bXi_j^{-1})$ we get (see, Theorem 3.2.8 in \citet{Muirhead})
\begin{equation*}
\frac{\bd_j^\top \bQ_j^{-1} \bd_j}{\bd_j^\top \bXi_j^{-1} \bd_j}
\sim \chi^2_{T-K} \,,
\end{equation*}
and, consequently,
\begin{eqnarray}\label{th3b_ref1}
\frac{\lambda_j(\bQ_j^{-1})}{p-1} \stackrel{a.s.}{\longrightarrow} \frac{\lambda_j}{c}~ \text{for}~~ \frac{p}{T-K} \rightarrow c \in (0,1) ~~ \text{as} ~~ T-K \rightarrow \infty \,.
\end{eqnarray}

Hence, from Slutsky's lemma (see, e.g., Theorem 1.5 in \citet{DasGupta2008}) we obtain
\begin{equation*}
\sqrt{p-1}\left(\frac{\omega_{jj}\left(\frac{\bv^\top_{21,j}}{v_{jj}}\right)^\top \bQ_{j}^{-1}\left(\frac{\bv^\top_{21,j}}{v_{jj}}\right)}{p-1}
-1\right) \stackrel{d.}{\longrightarrow} \mathcal{N}\left(\frac{\lambda_j}{c},2+4\frac{\lambda_j}{c}\right) \,,
\end{equation*}
and thus
\begin{equation*}
\sqrt{p-1}\left(\frac{\omega_{11,j}\left(\frac{\bv^\top_{21,j}}{v_{jj}}\right)^\top \bQ_{j}^{-1}\left(\frac{\bv^\top_{21,j}}{v_{jj}}\right)/(p-1)-\frac{\lambda_j}{c}}
{(\omega_{11,j}/v_{jj})/(T-K-p+1)}-1 \right)\stackrel{d.}{\longrightarrow} \mathcal{N}\left(0,\frac{2}{1-c}+4\frac{\lambda_j}{c}\right) \,.
\end{equation*}

In case of $T-K-p \longrightarrow d \in (0,\infty)$, we get from Lemma \ref{lem3}.(a)
\[
\frac{\omega_{jj}\left(\frac{\bv^\top_{21,j}}{v_{jj}}\right)^\top \bQ_{j}^{-1}\left(\frac{\bv^\top_{21,j}}{v_{jj}}\right)}{p-1}
-1-\frac{\lambda_j(\bQ_j^{-1})}{p-1} \stackrel{a.s.}{\longrightarrow} 0 ~~ \text{as} ~~ p \rightarrow \infty.
\]
Applying (\ref{th3b_ref1}) and $p/(T-K) \rightarrow 1$, we get the statement of the second part of Theorem \ref{th3}.(b).
\end{enumerate}
\end{proof}

\section*{Acknowledgments}
The authors are grateful to Professor Christian Genest, the associate editor and the referees for their suggestions, which have improved the presentation in the paper. We also thank David Bauder for his comments used in the preparation of the revised version of the paper.

\bibliography{Test-FM_ref}
%%%%%%%%%%%%%%%%%%%%%%%%%%%%%%%%%%%%%%%%%%%%%%%%%%%%%%%%%%%%%%%%%%%%%%%%%%%%%%%%%%%%%%%%%%%%%%%%%%Figures
%%%%%%%%%%%%%%%%%%%%%%%%%%%%%%%%%%%%%Figures Section 6.1
\newpage
\clearpage
\begin{landscape}
\begin{figure}[h!tb]
\begin{center}
\begin{tabular}{ccc}
\scalebox{0.4}{\includegraphics[]{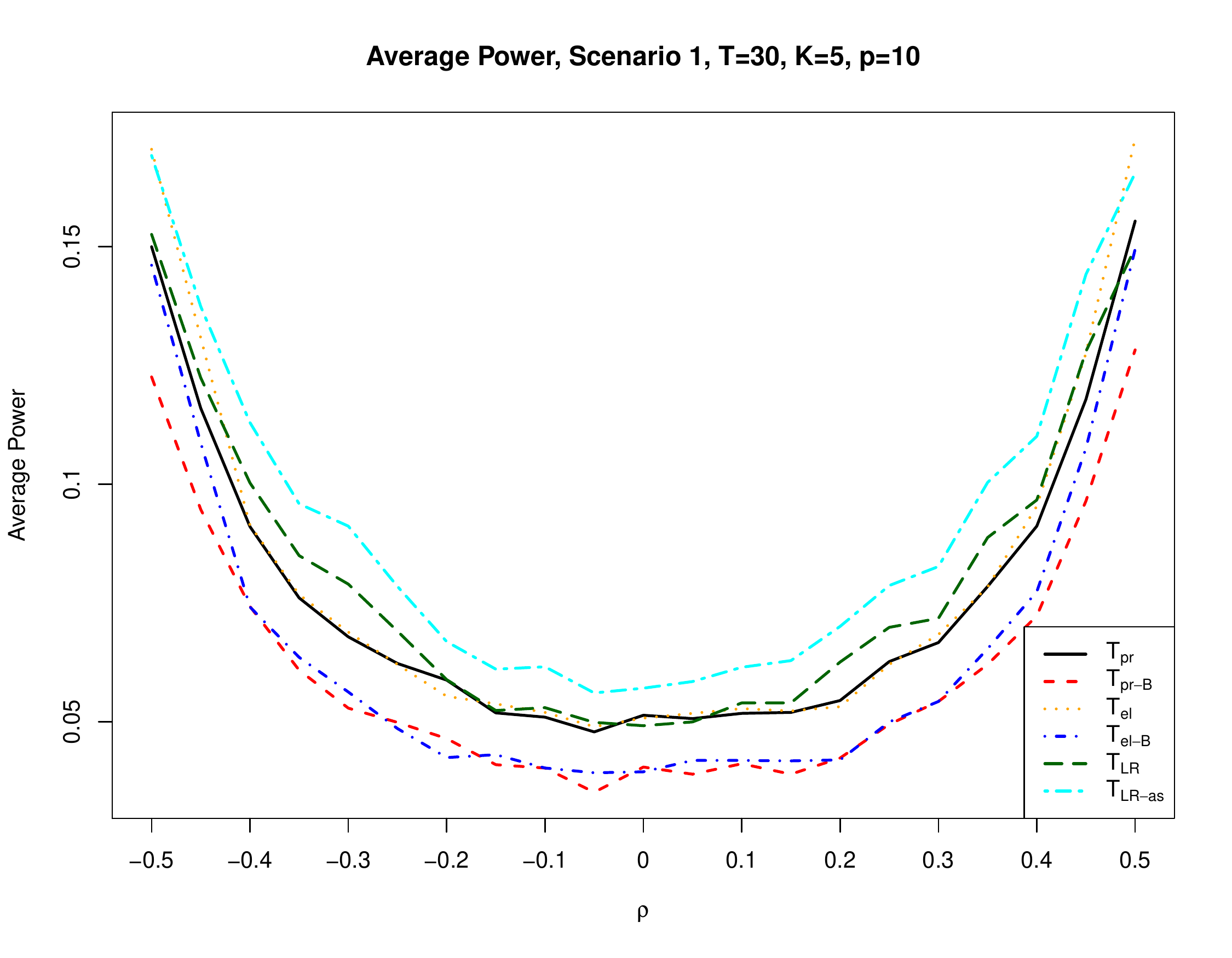}}&&\scalebox{0.4}{\includegraphics[]{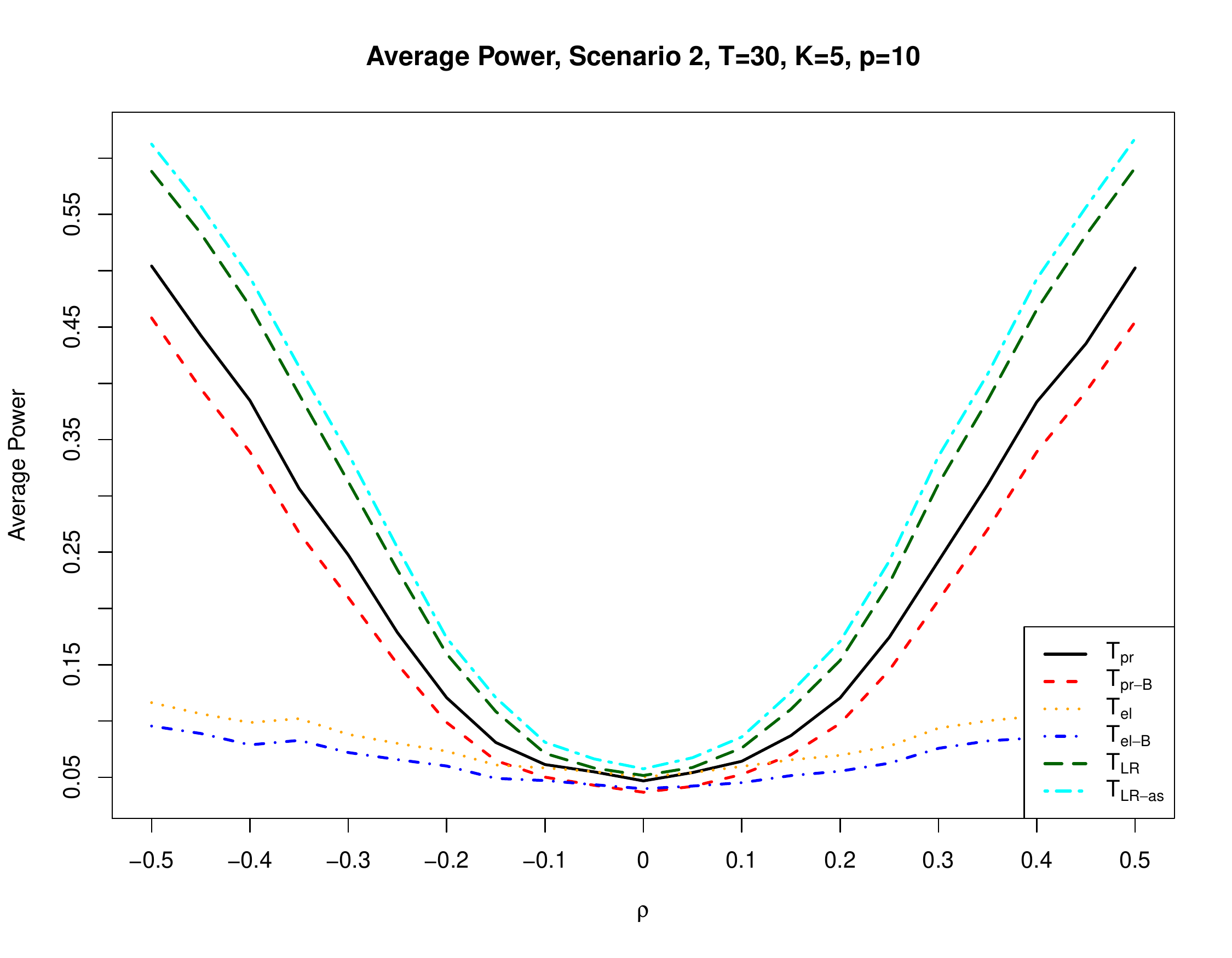}}\\
%&&\\
\scalebox{0.4}{\includegraphics[]{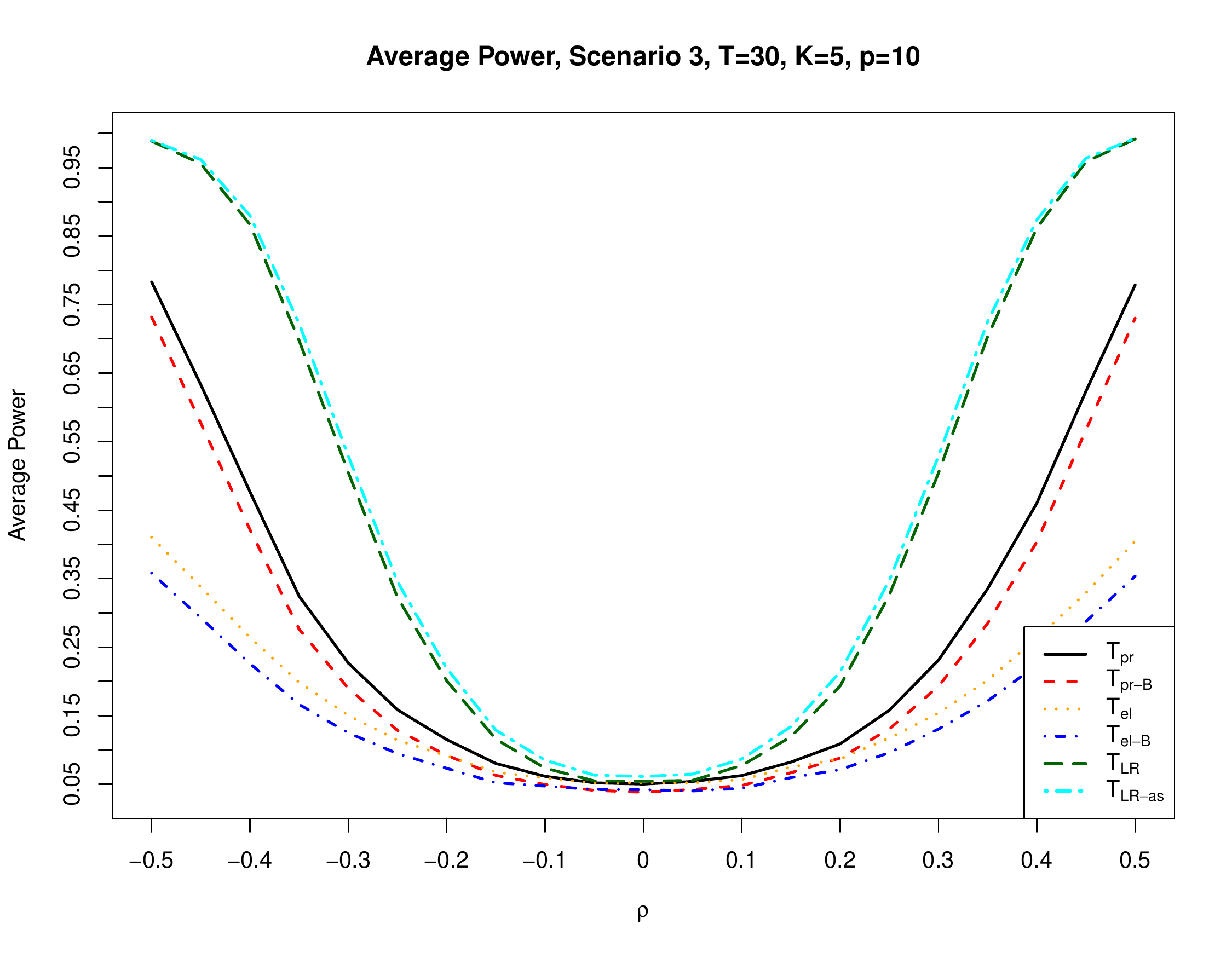}}&&\scalebox{0.4}{\includegraphics[]{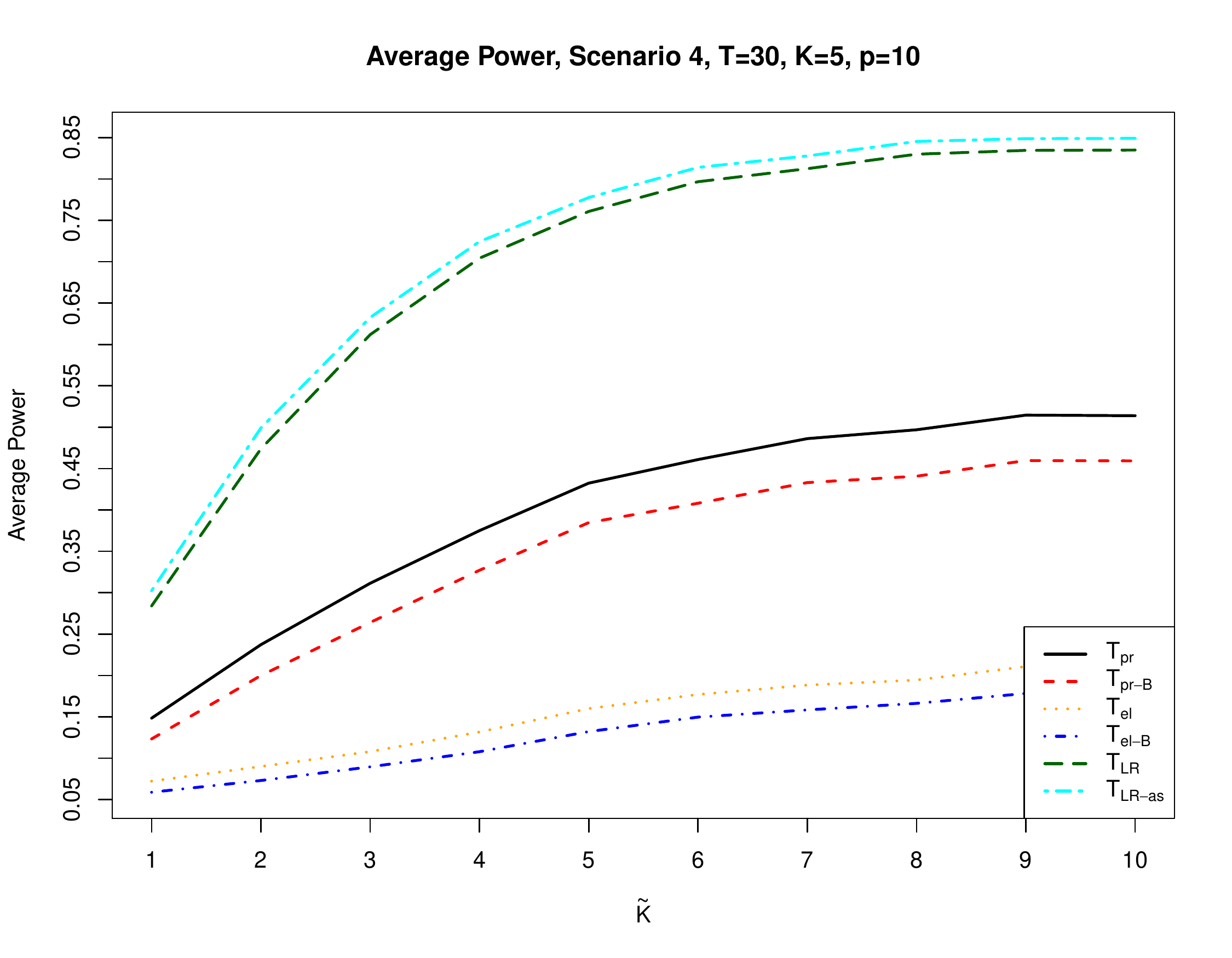}}\\
\end{tabular}
\end{center}
\caption{Power of $T_{el}$, $T_{el}$, and $T_{LR}$ based on the simulated critical values and of the corresponding tests whose critical values are determined by Bonferroni correction or asymptotic distribution ($T=30$, $K=5$, $p=10$). }
\label{Fig:T30K5p10}
\end{figure}
\end{landscape}

\newpage
\clearpage
\begin{landscape}
\begin{figure}[h!tb]
\begin{center}
\begin{tabular}{ccc}
\scalebox{0.4}{\includegraphics[]{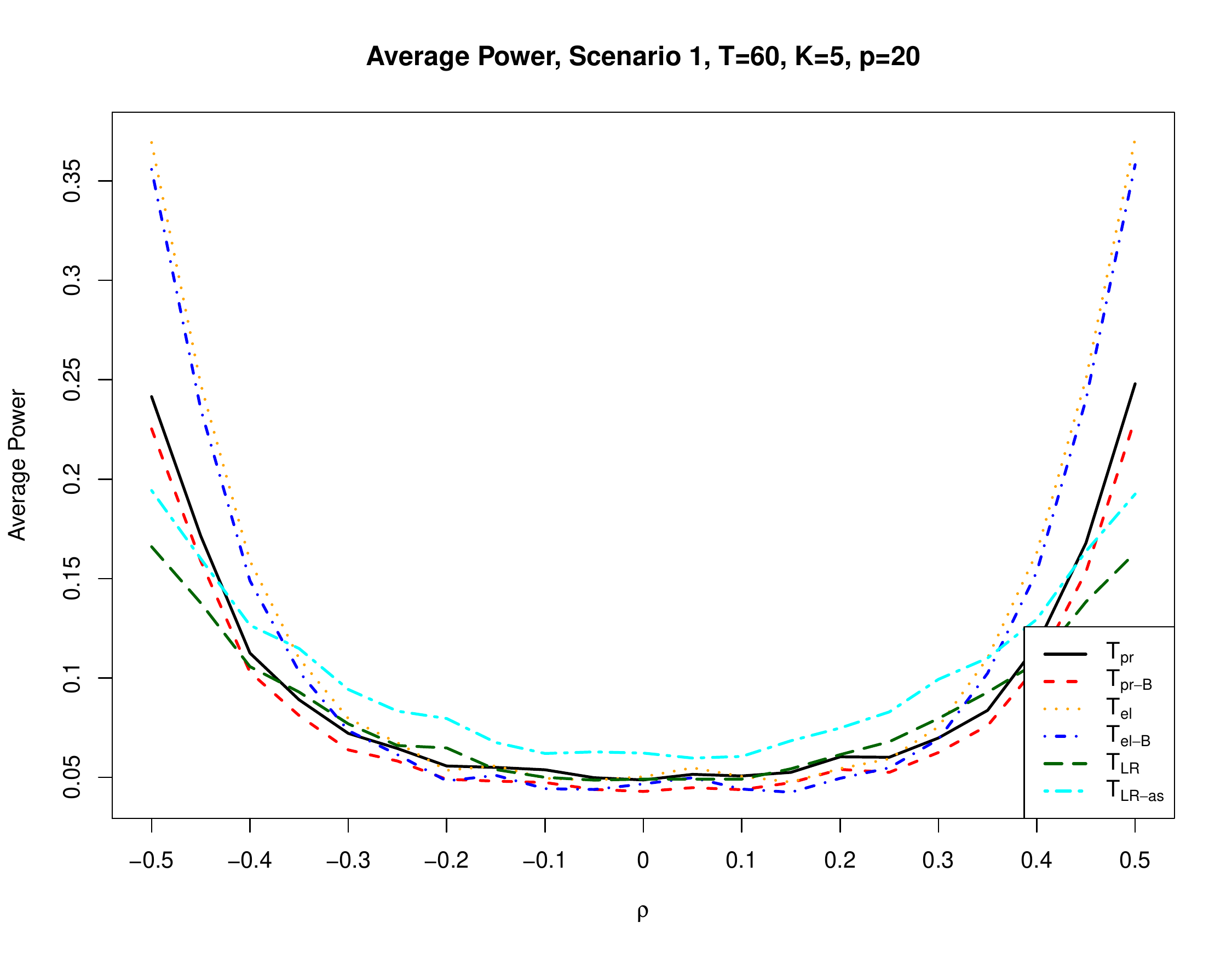}}&&\scalebox{0.4}{\includegraphics[]{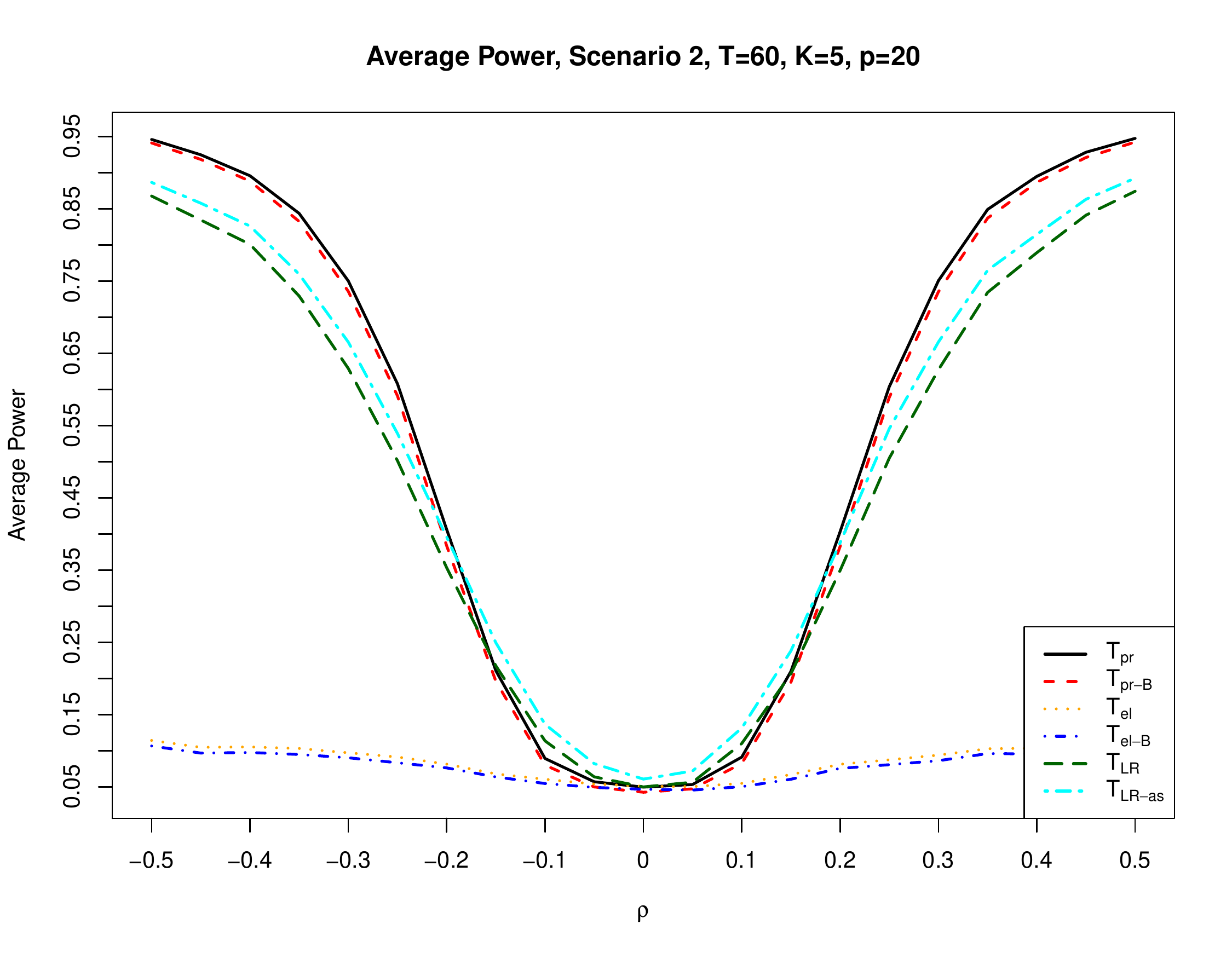}}\\
%&&\\
\scalebox{0.4}{\includegraphics[]{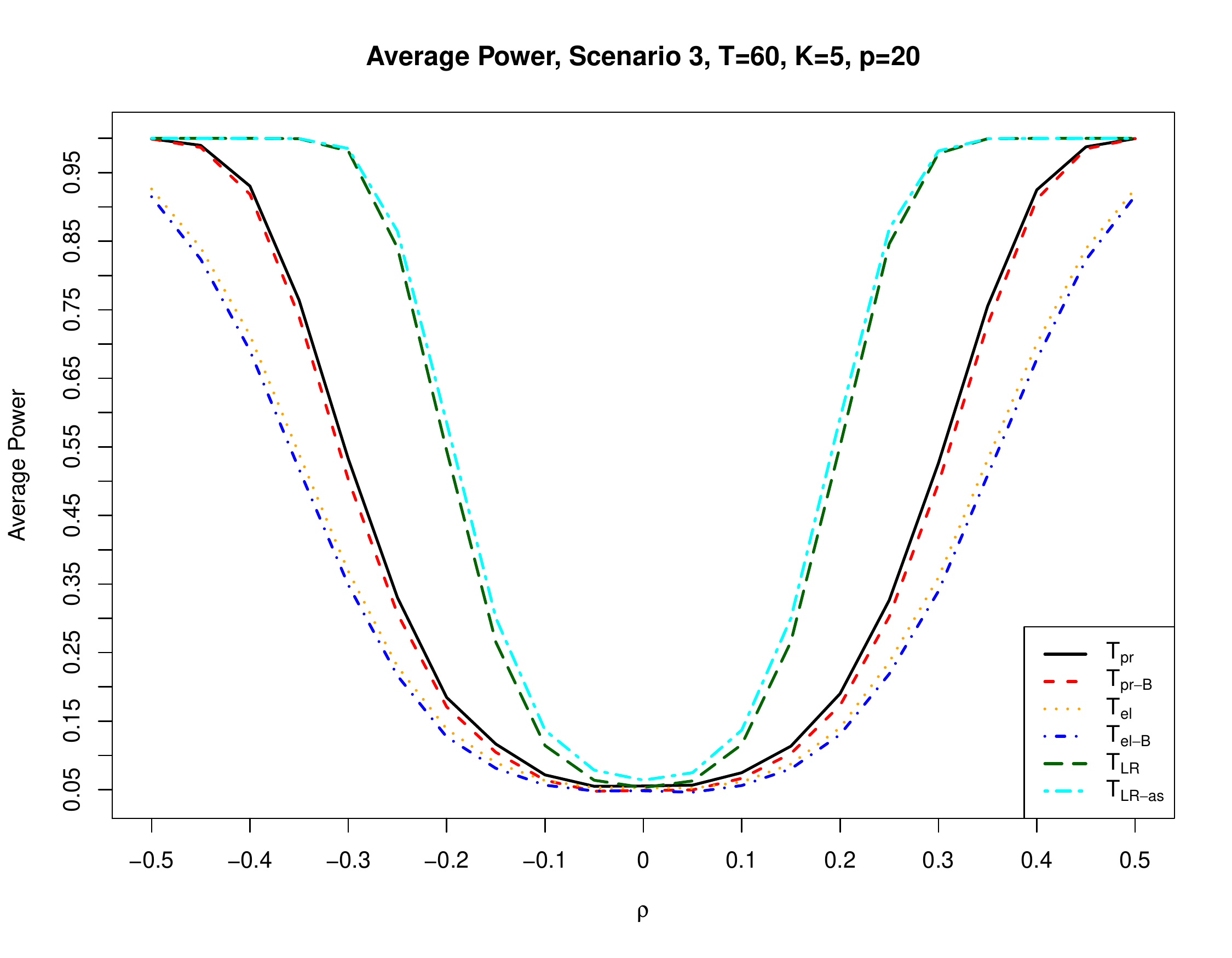}}&&\scalebox{0.4}{\includegraphics[]{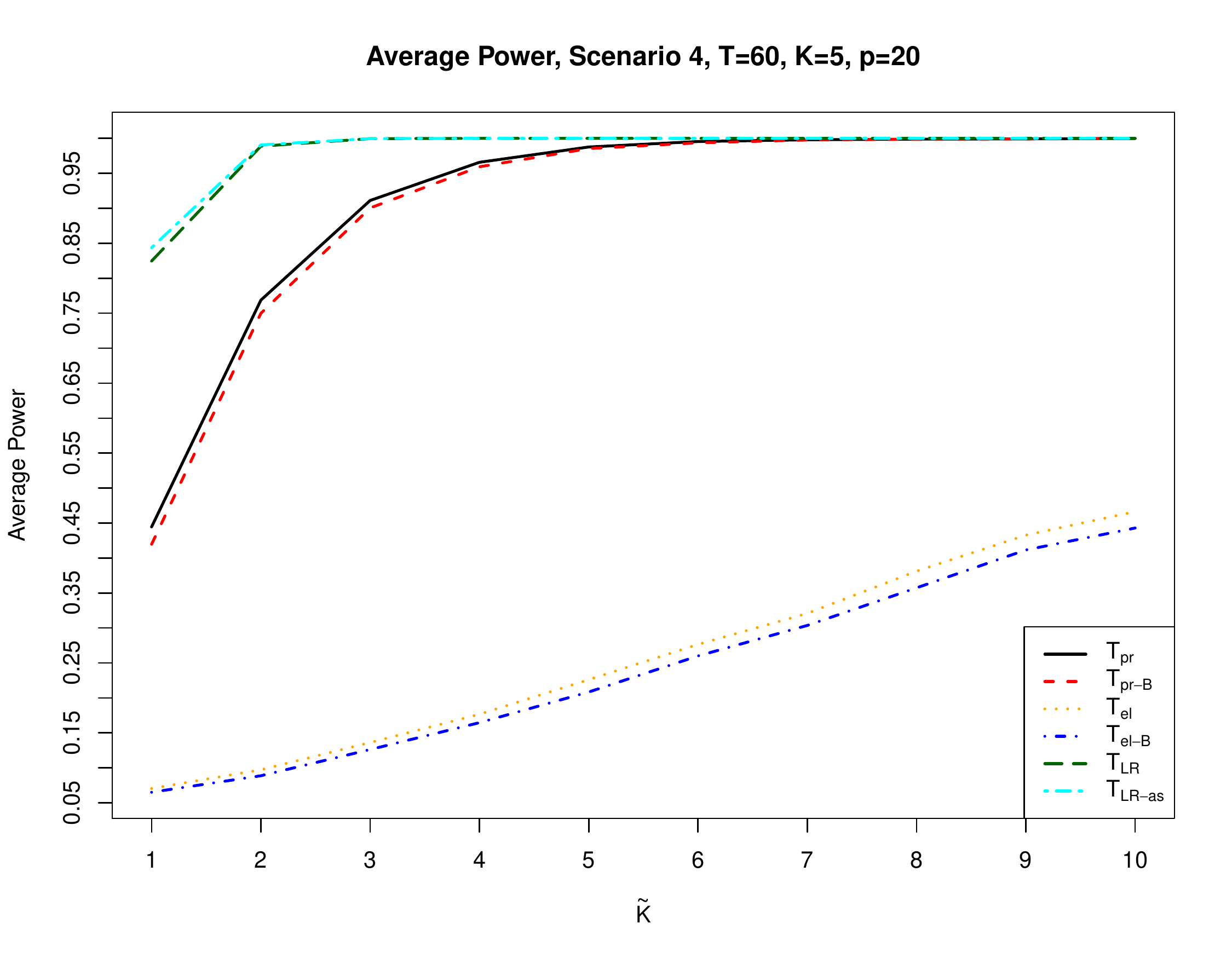}}\\
\end{tabular}
\end{center}
\caption{Power of $T_{el}$, $T_{el}$, and $T_{LR}$ based on the simulated critical values and of the corresponding tests whose critical values are determined by Bonferroni correction or asymptotic distribution ($T=60$, $K=5$, $p=20$). }
\label{Fig:T60K5p20}
\end{figure}
\end{landscape}

\newpage
\clearpage
\begin{landscape}
\begin{figure}[h!tb]
\begin{center}
\begin{tabular}{ccc}
\scalebox{0.4}{\includegraphics[]{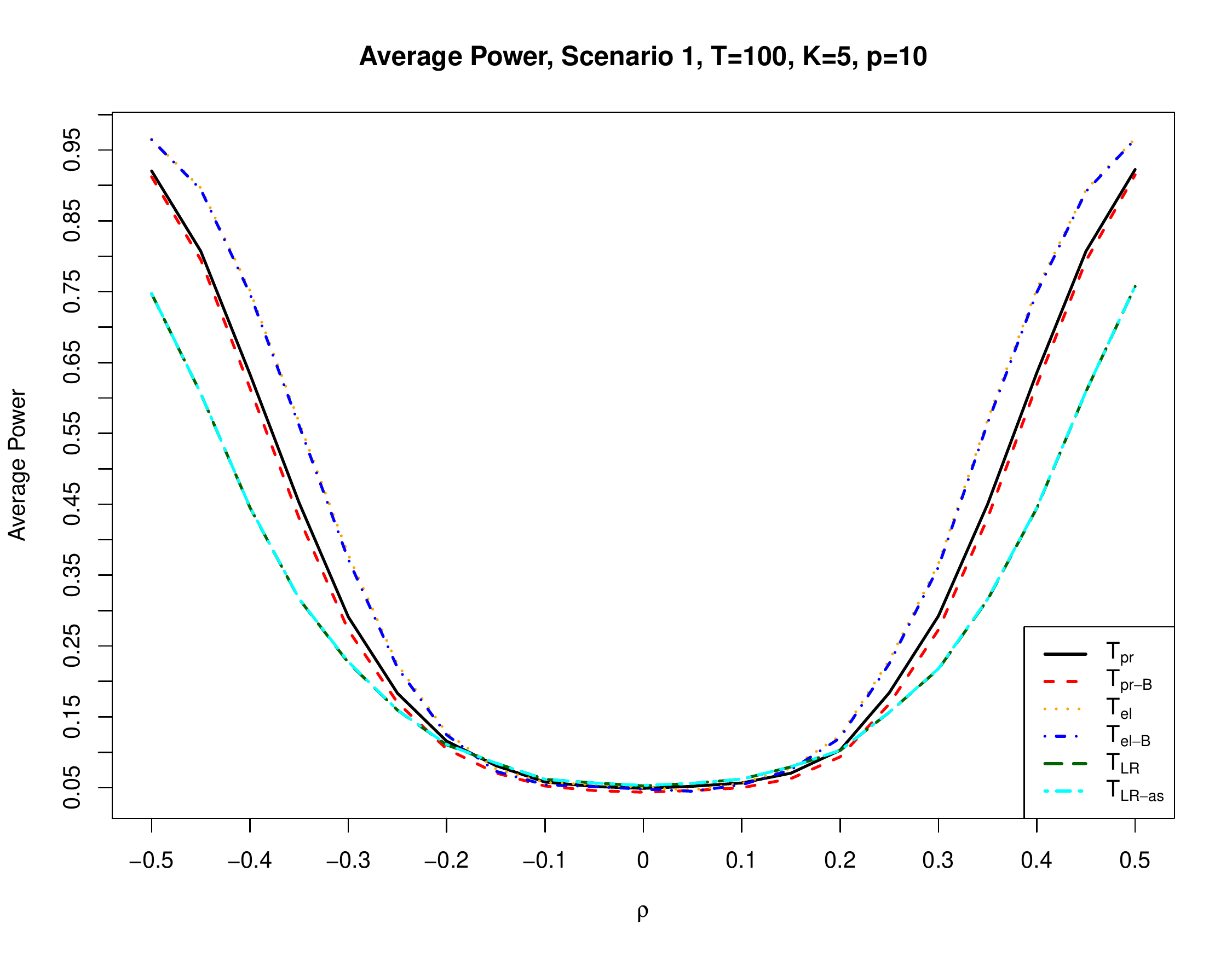}}&&\scalebox{0.4}{\includegraphics[]{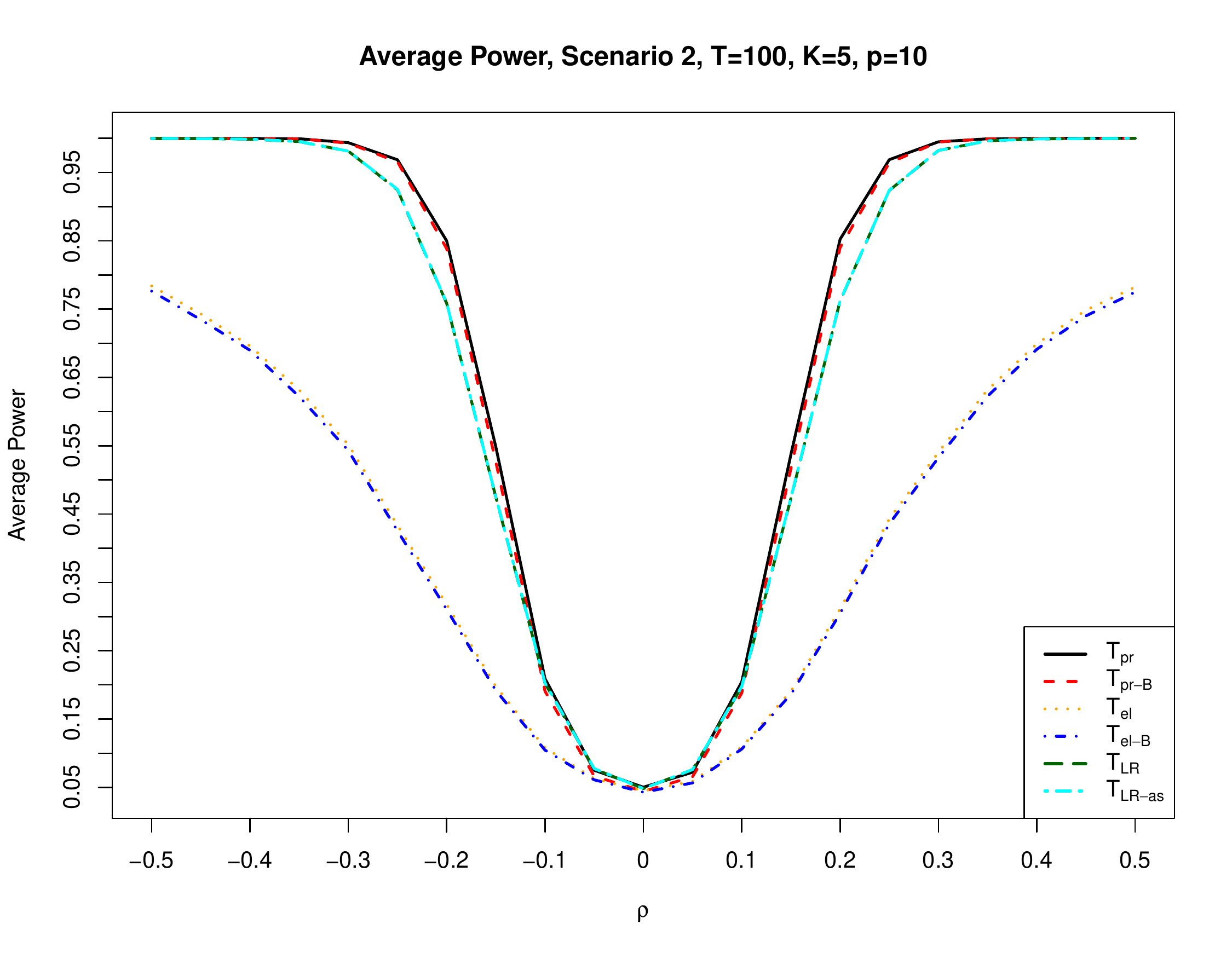}}\\
%&&\\
\scalebox{0.4}{\includegraphics[]{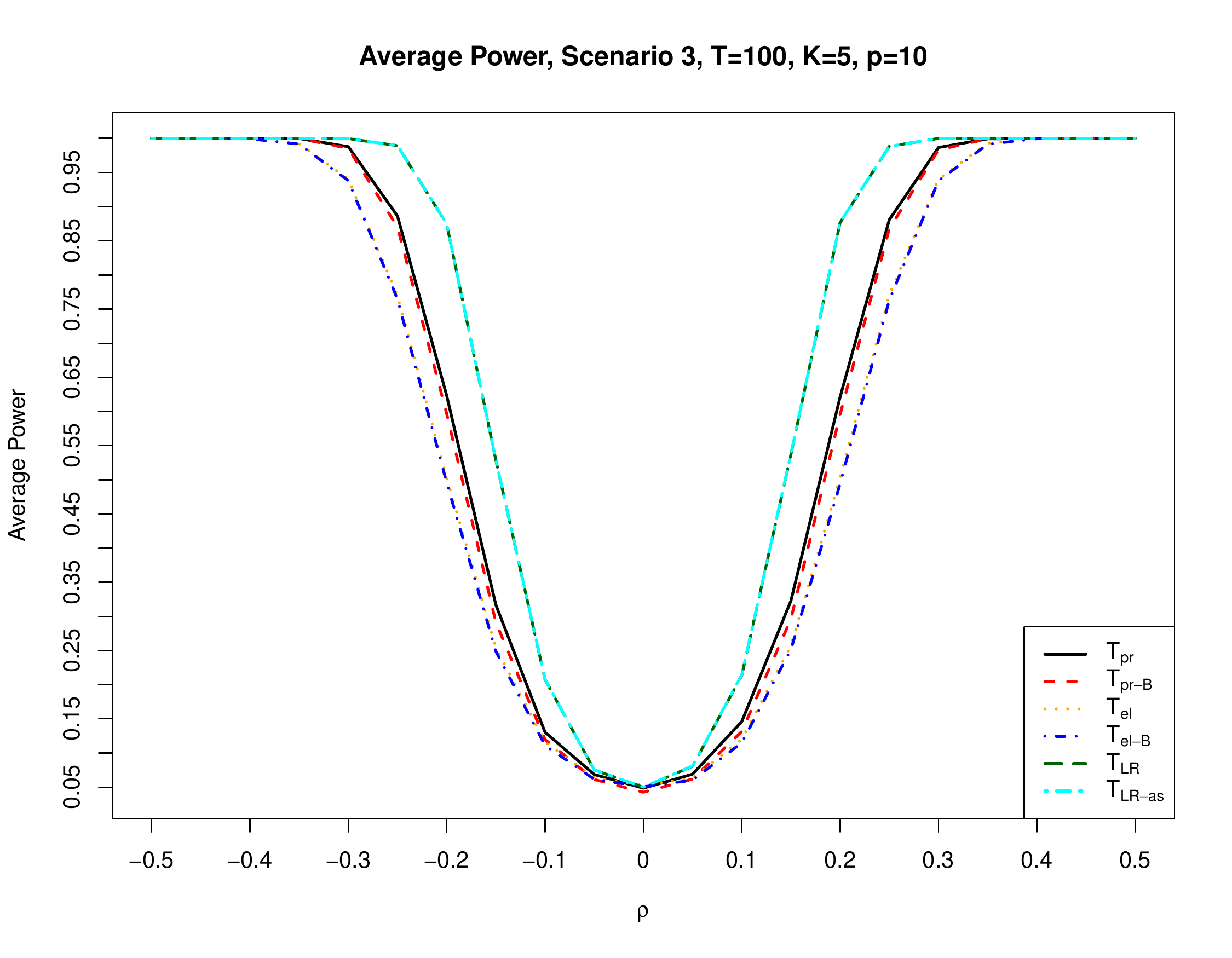}}&&\scalebox{0.4}{\includegraphics[]{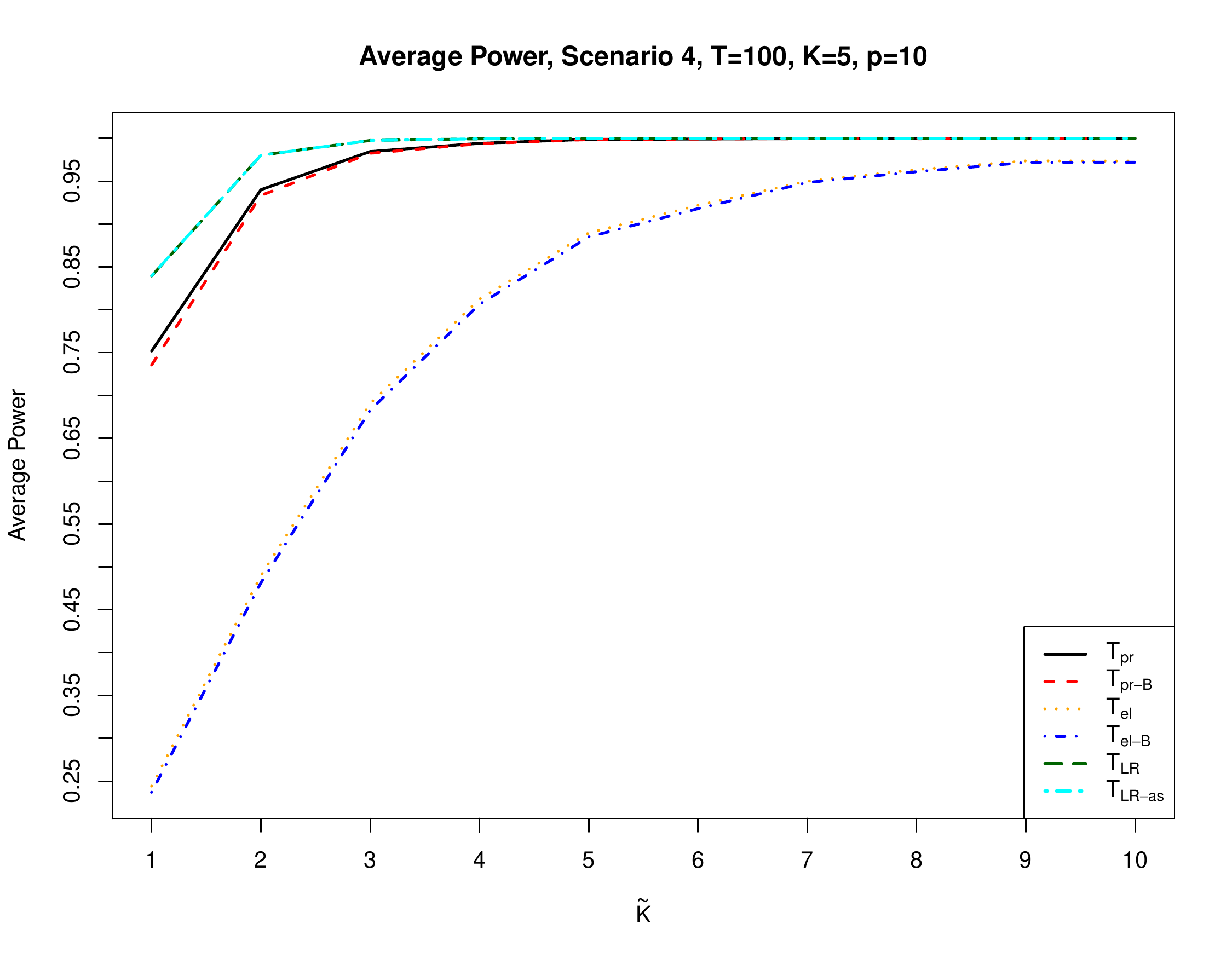}}\\
\end{tabular}
\end{center}
\caption{Power of $T_{el}$, $T_{el}$, and $T_{LR}$ based on the simulated critical values and of the corresponding tests whose critical values are determined by Bonferroni correction or asymptotic distribution ($T=100$, $K=5$, $p=10$). }
\label{Fig:T100K5p10}
\end{figure}
\end{landscape}

\newpage
\clearpage
\begin{landscape}
\begin{figure}[h!tb]
\begin{center}
\begin{tabular}{ccc}
\scalebox{0.4}{\includegraphics[]{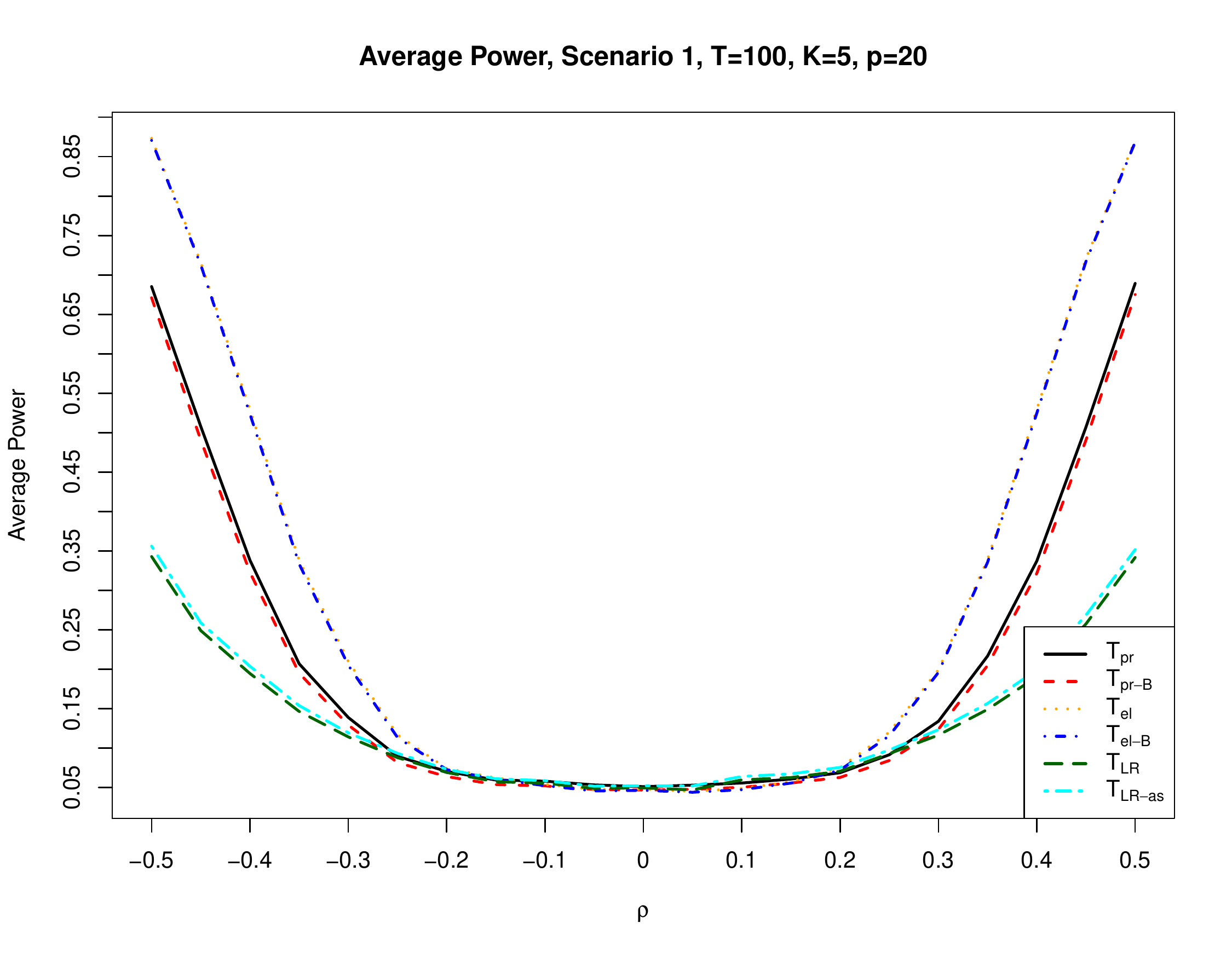}}&&\scalebox{0.4}{\includegraphics[]{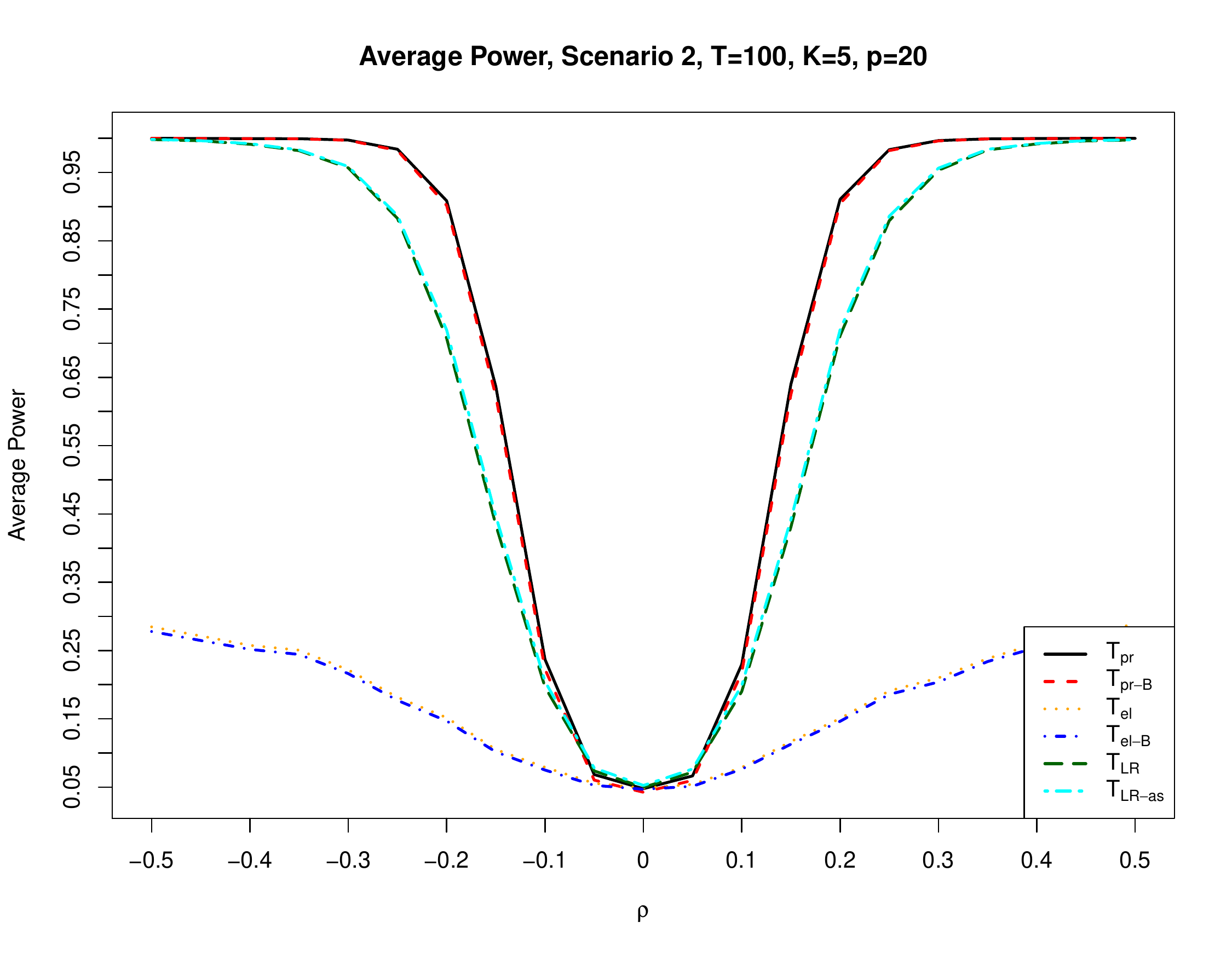}}\\
%&&\\
\scalebox{0.4}{\includegraphics[]{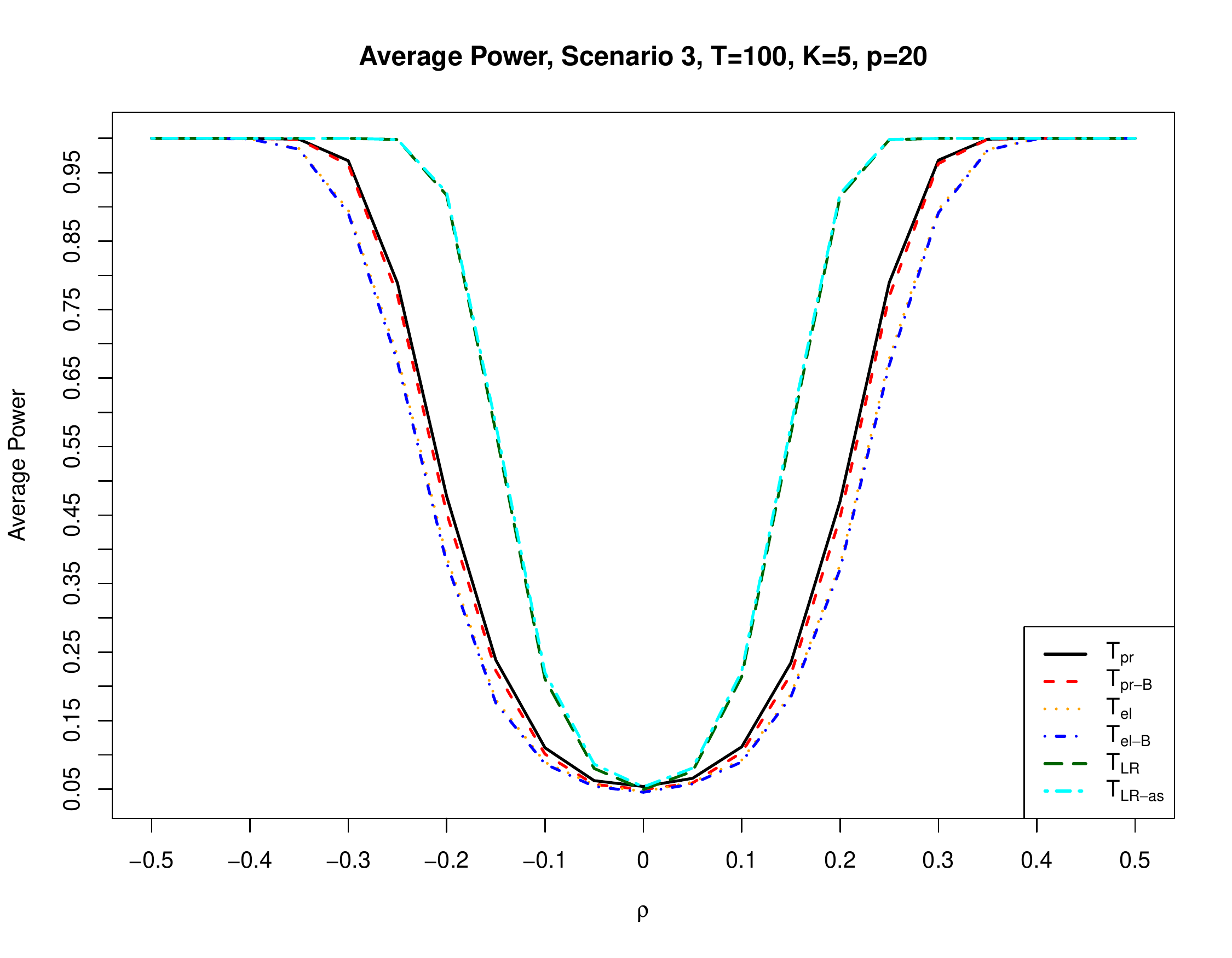}}&&\scalebox{0.4}{\includegraphics[]{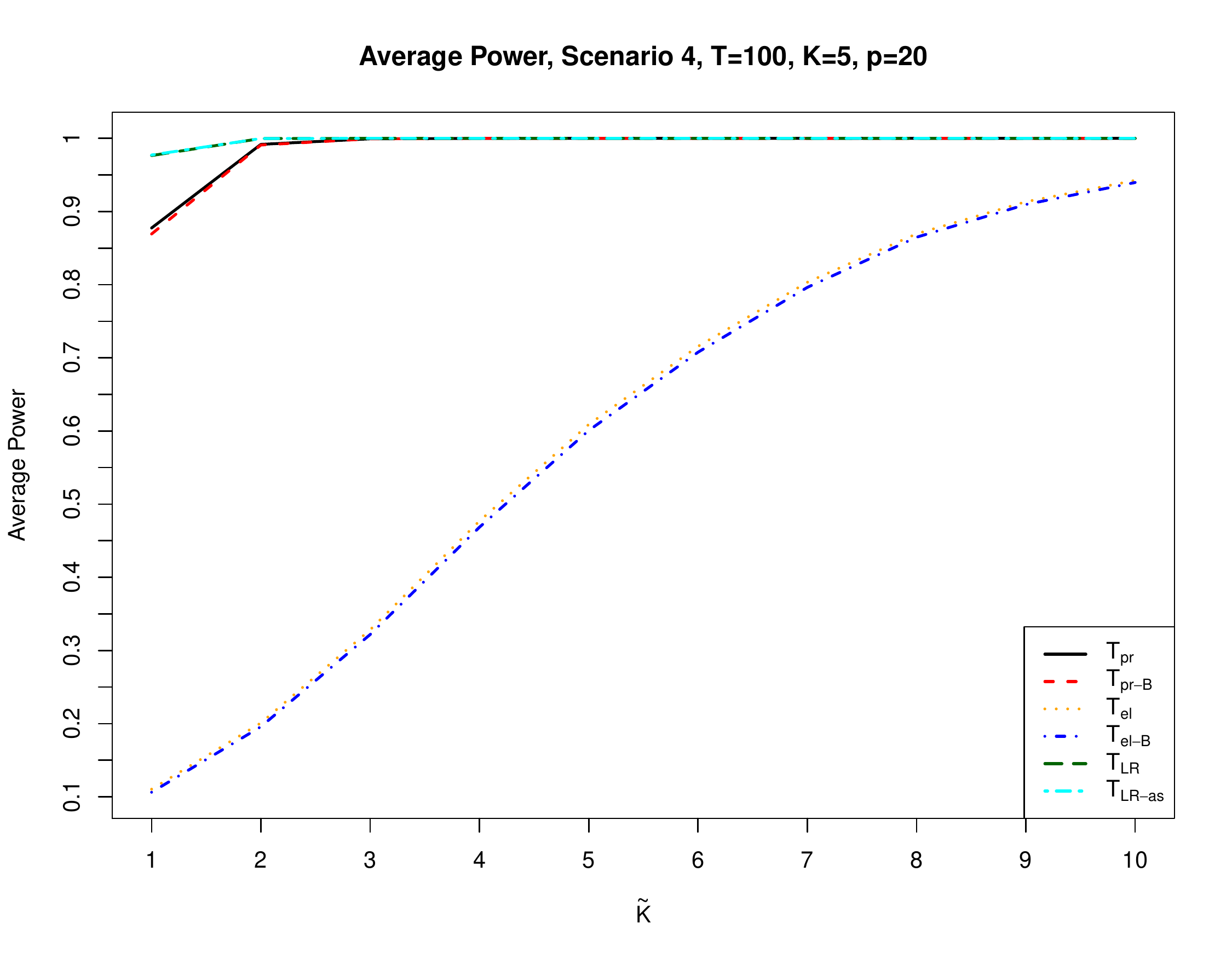}}\\
\end{tabular}
\end{center}
\caption{Power of $T_{el}$, $T_{el}$, and $T_{LR}$ based on the simulated critical values and of the corresponding tests whose critical values are determined by Bonferroni correction or asymptotic distribution ($T=100$, $K=5$, $p=20$). }
\label{Fig:T100K5p20}
\end{figure}
\end{landscape}

%%%%%%%%%%%%%%%%%%%%%%%%%%%%%%%%%%%%%Figures Section 6.2
\newpage
\clearpage
\begin{landscape}
\begin{figure}[h!tb]
\begin{center}
\begin{tabular}{ccc}
\scalebox{0.4}{\includegraphics[]{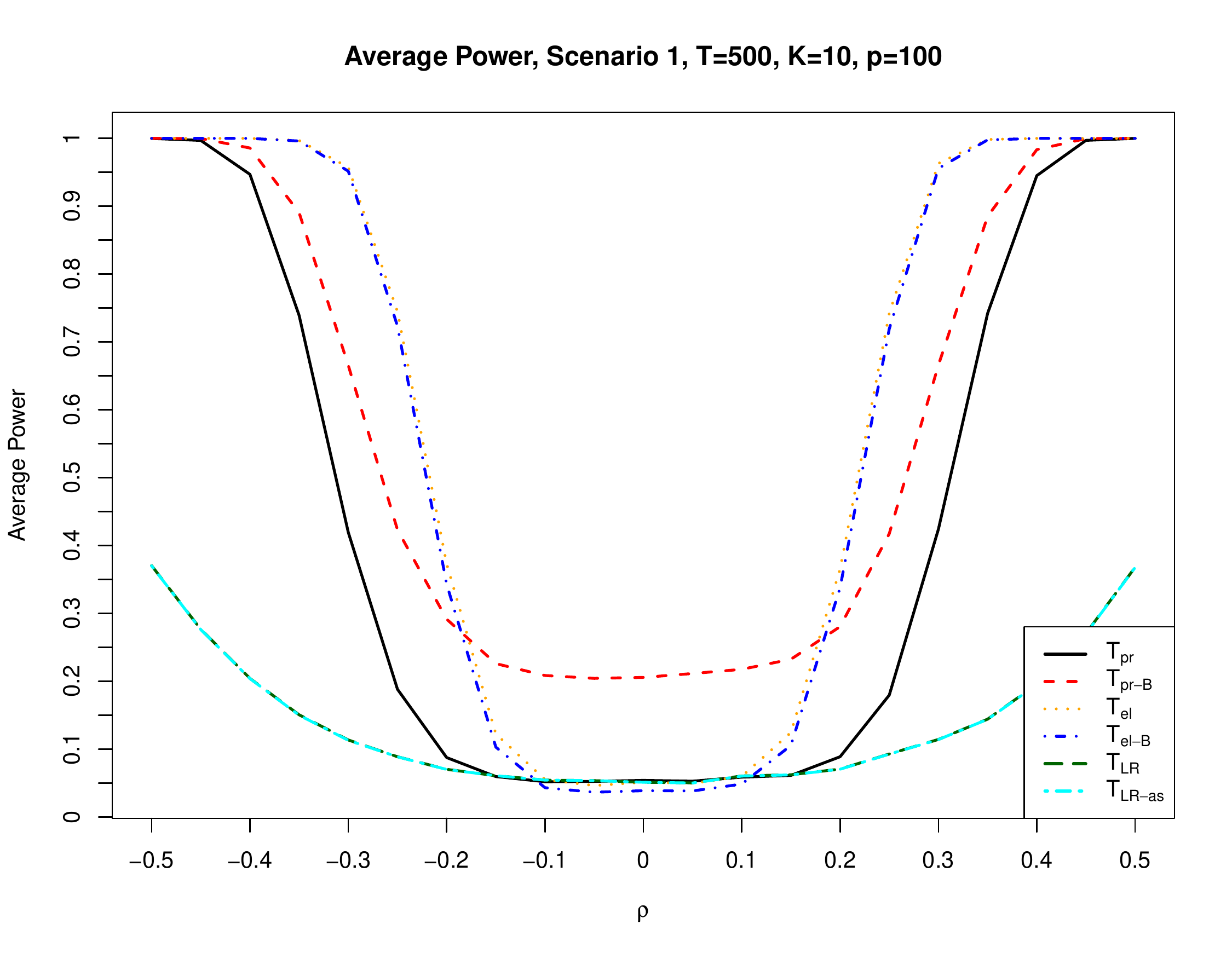}}&&\scalebox{0.4}{\includegraphics[]{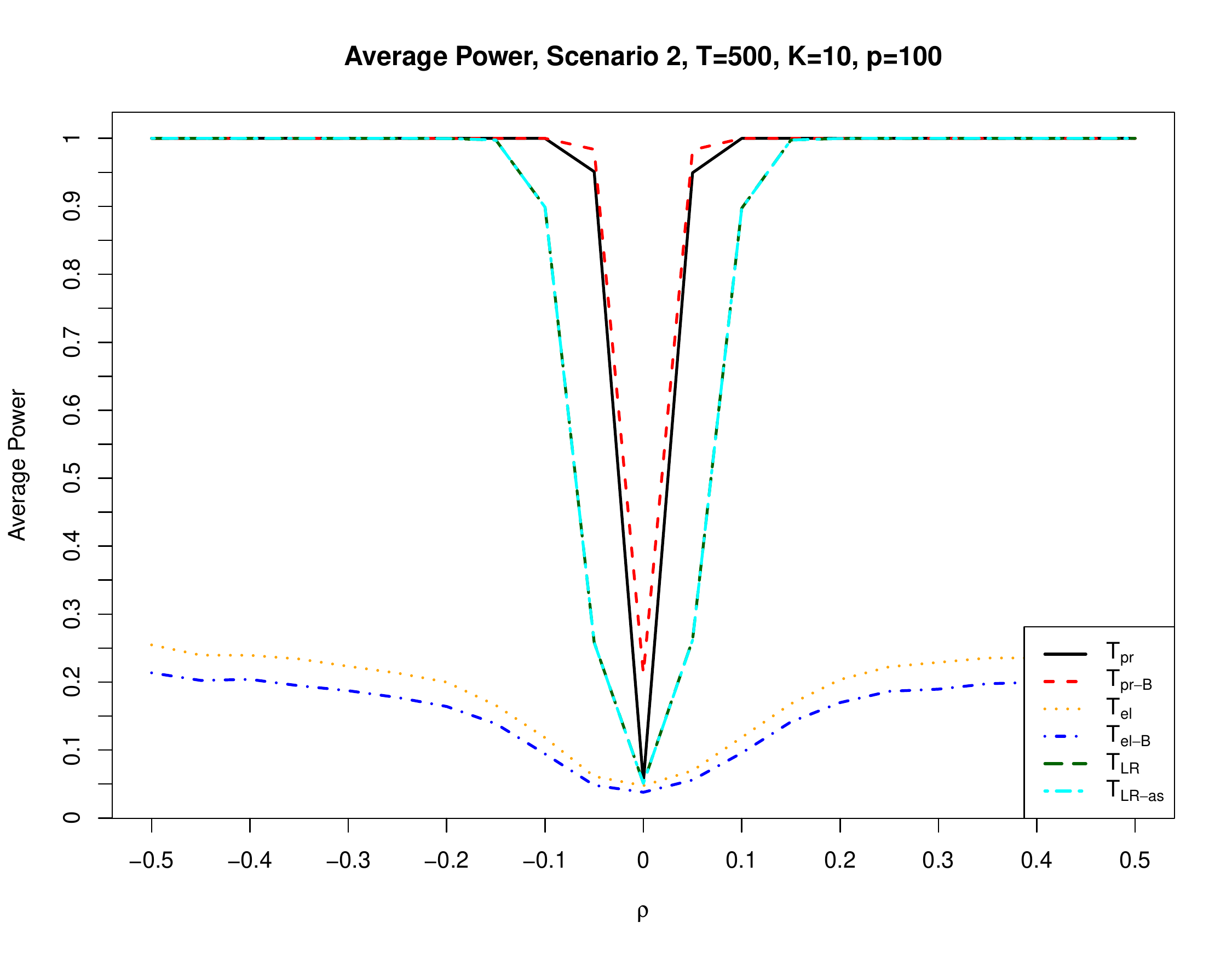}}\\
%&&\\
\scalebox{0.4}{\includegraphics[]{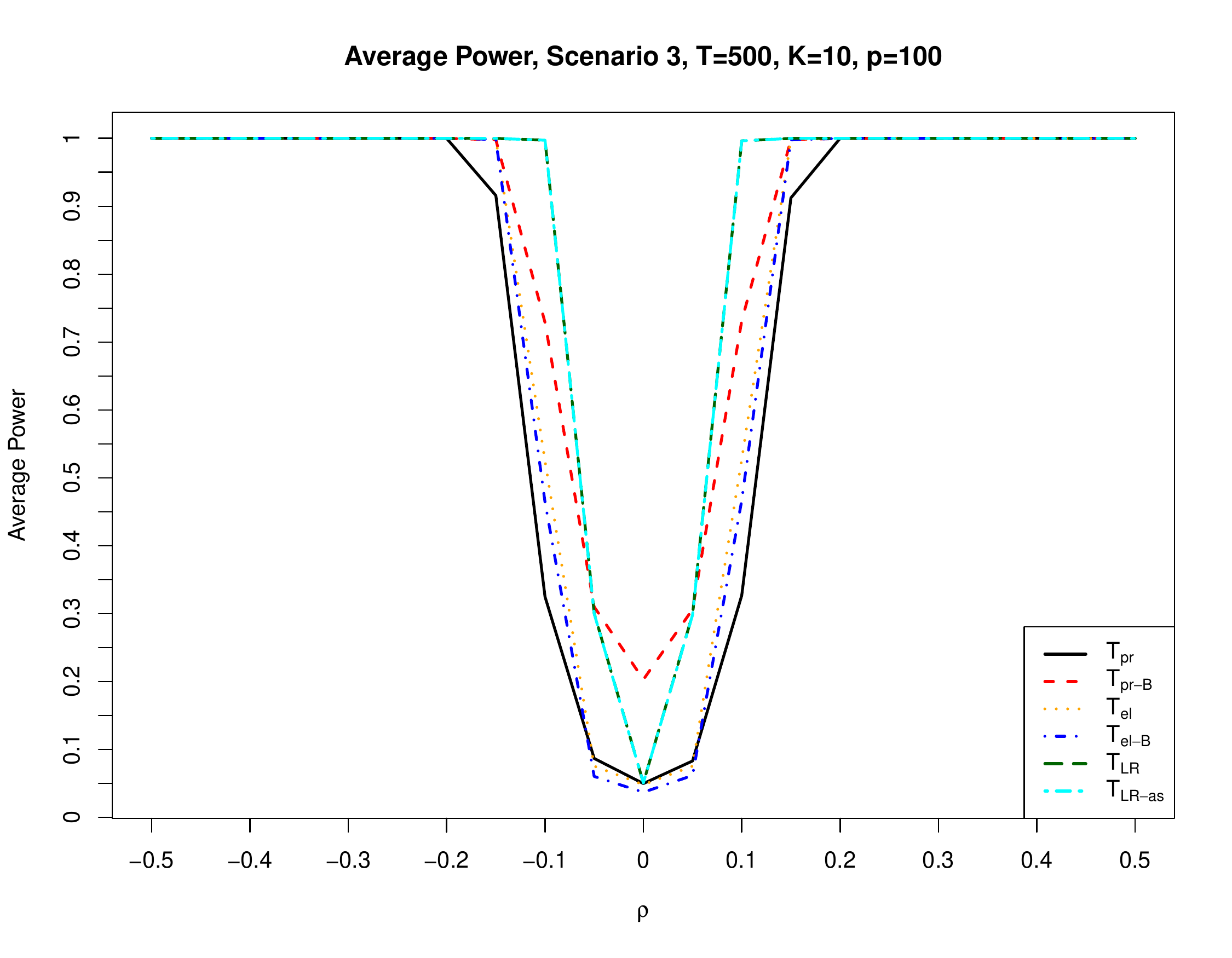}}&&\scalebox{0.4}{\includegraphics[]{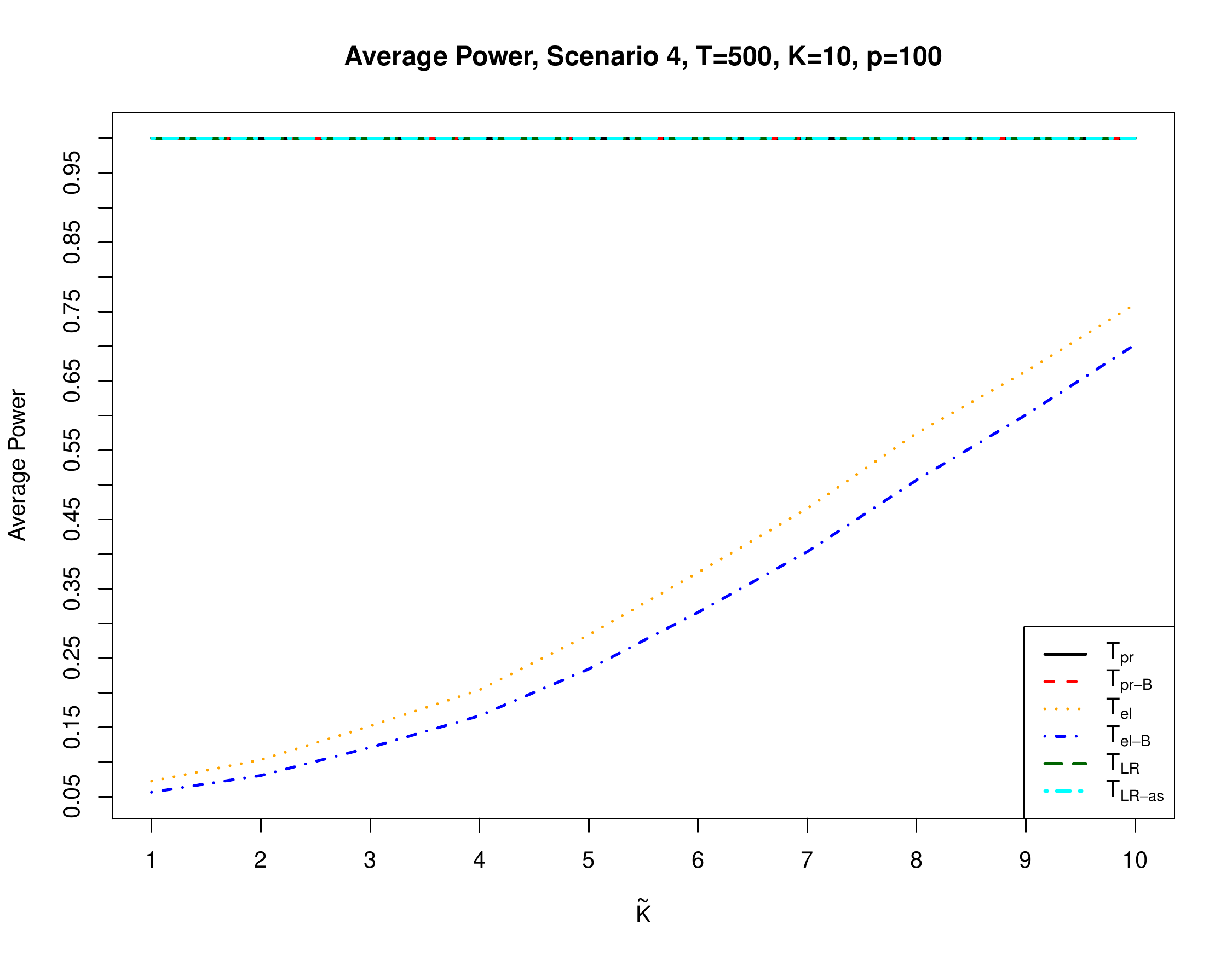}}\\
\end{tabular}
\end{center}
\caption{Power of $T_{el}$, $T_{el}$, and $T_{LR}$ based on the simulated critical values and of the corresponding tests whose critical values are determined by Bonferroni correction or asymptotic distribution ($T=500$, $K=10$, $p=100$). }
\label{Fig:T500K10p100}
\end{figure}
\end{landscape}

\newpage
\clearpage
\begin{landscape}
\begin{figure}[h!tb]
\begin{center}
\begin{tabular}{ccc}
\scalebox{0.4}{\includegraphics[]{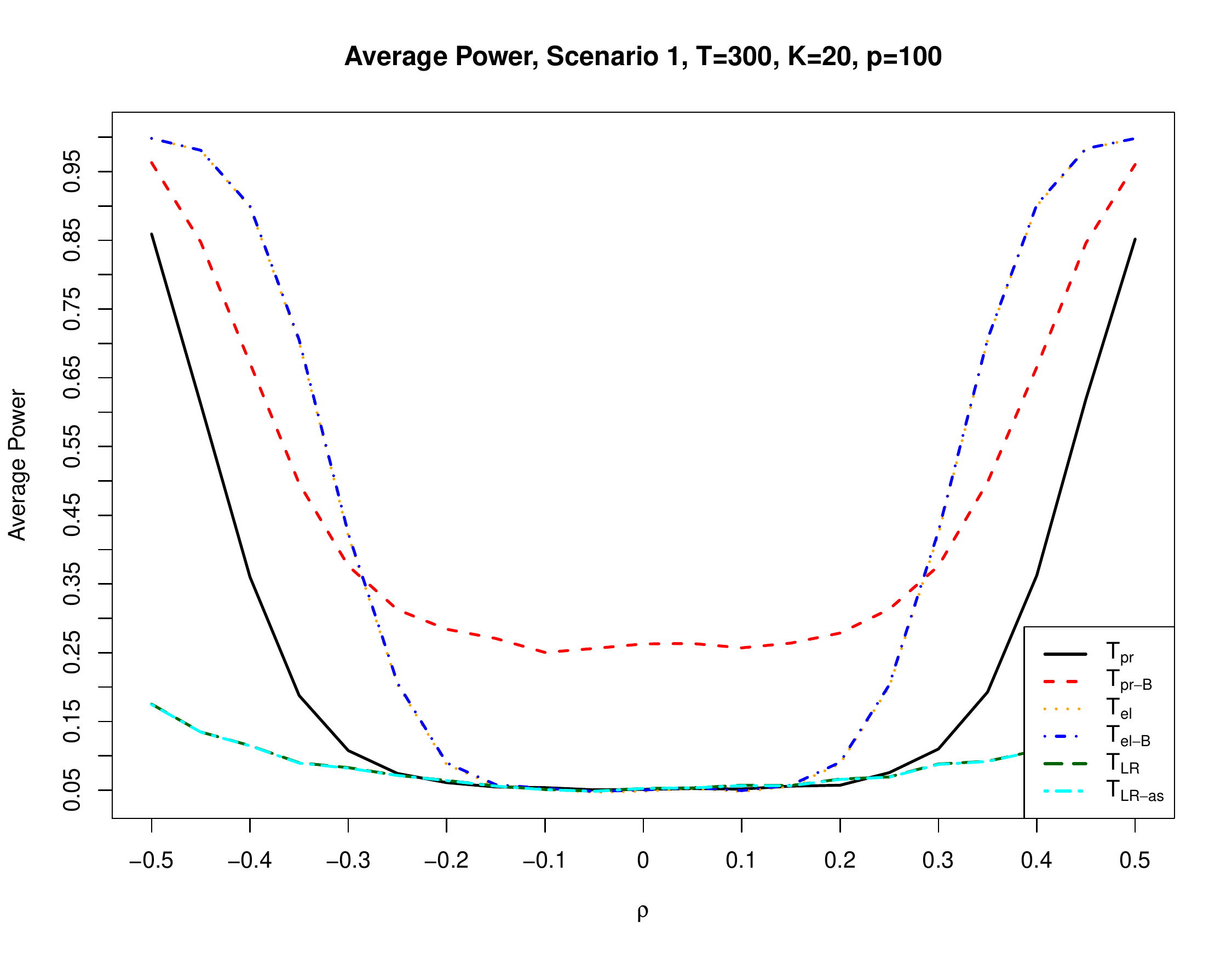}}&&\scalebox{0.4}{\includegraphics[]{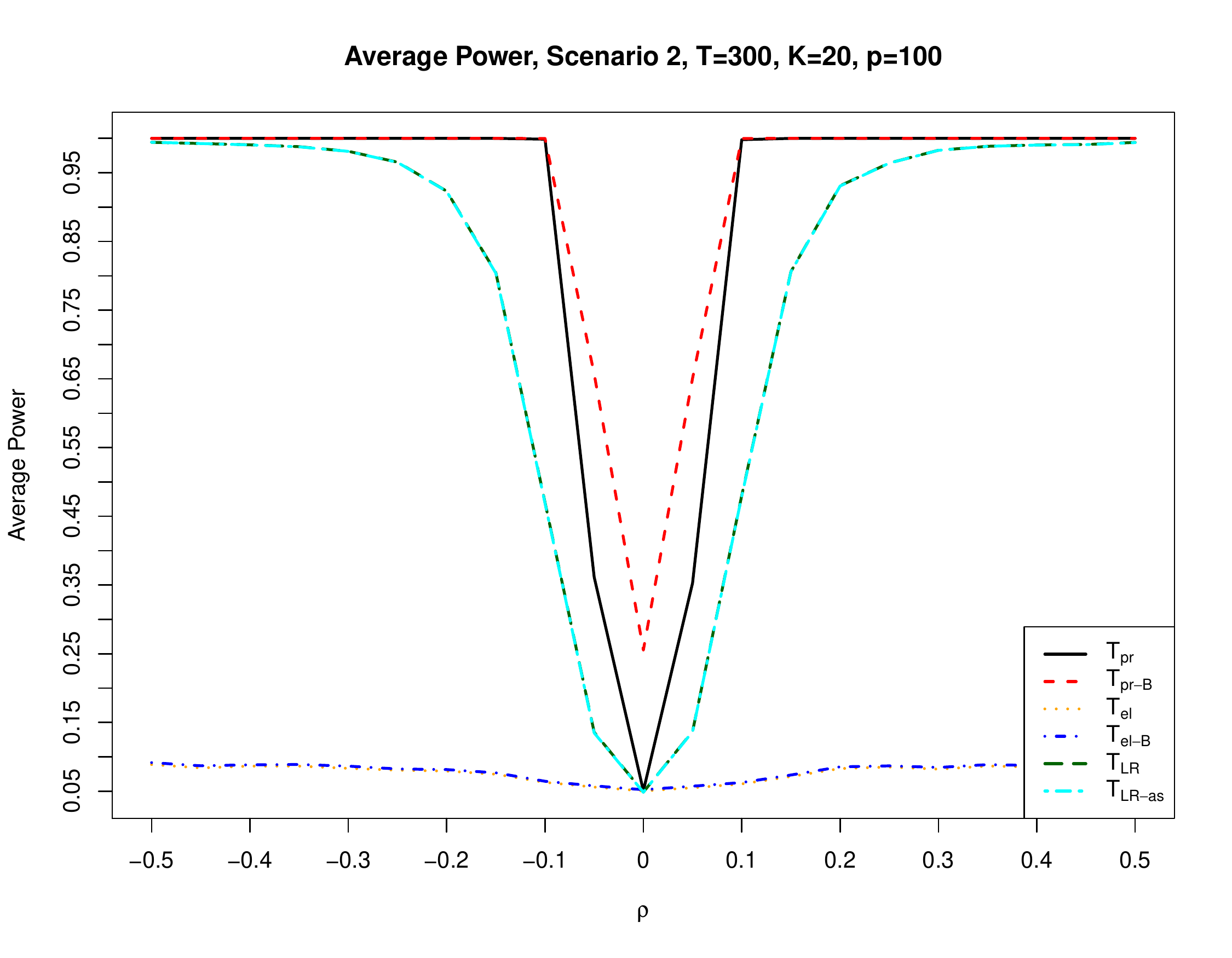}}\\
%&&\\
\scalebox{0.4}{\includegraphics[]{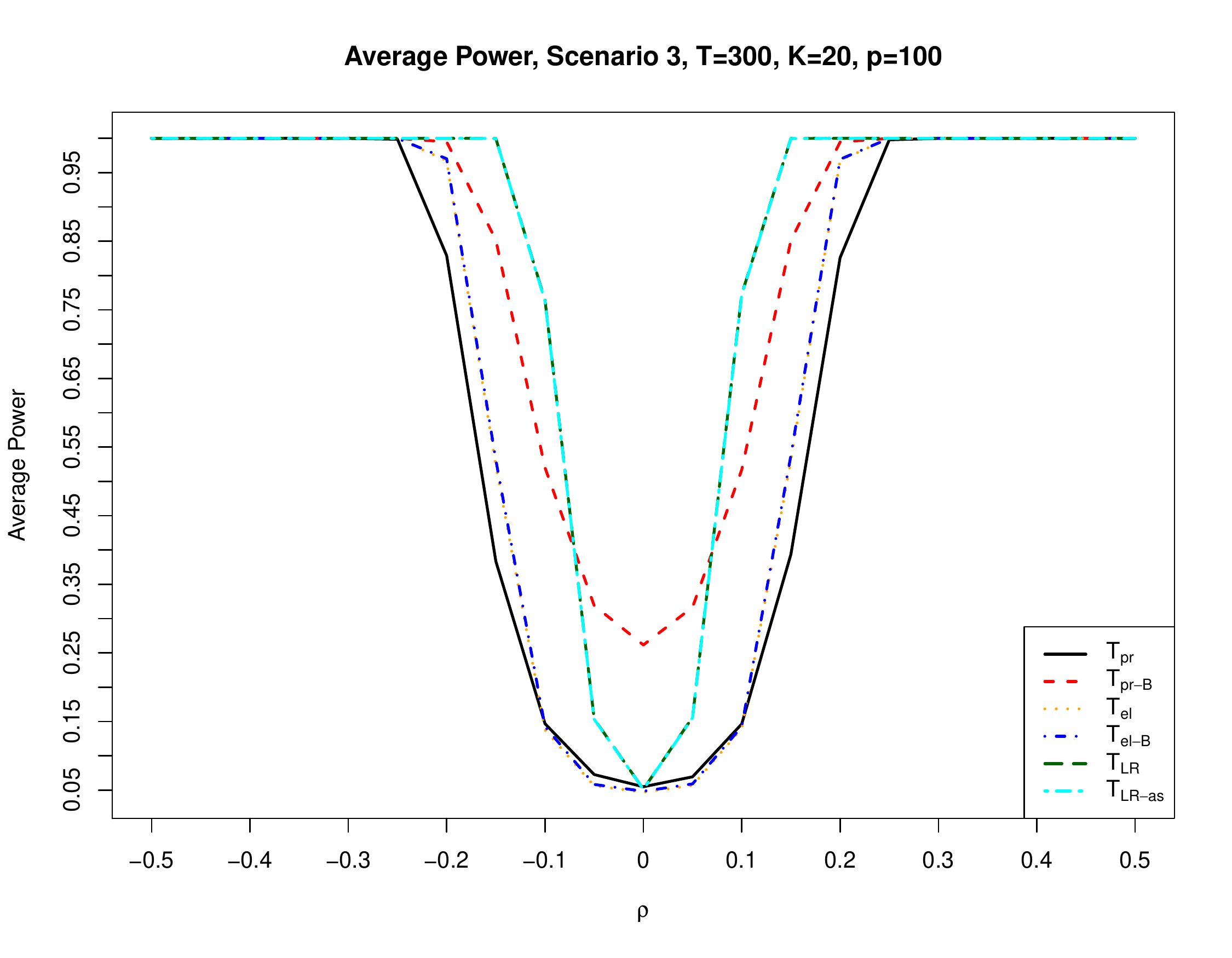}}&&\scalebox{0.4}{\includegraphics[]{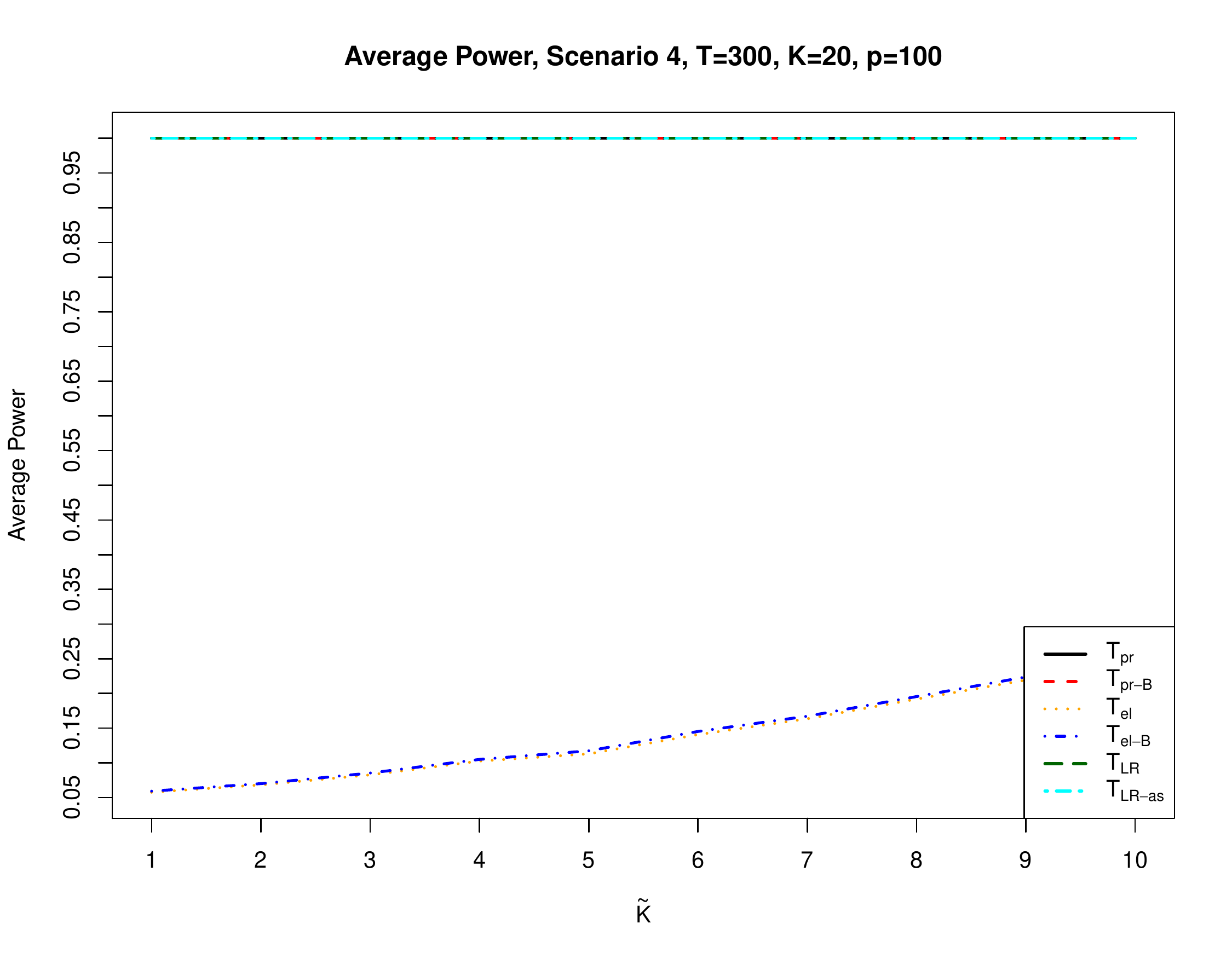}}\\
\end{tabular}
\end{center}
\caption{Power of $T_{el}$, $T_{el}$, and $T_{LR}$ based on the simulated critical values and of the corresponding tests whose critical values are determined by Bonferroni correction or asymptotic distribution ($T=300$, $K=20$, $p=100$). }
\label{Fig:T300K20p100}
\end{figure}
\end{landscape}

%%%%%%%%%%%%%%%%%%%%%%%%%%%%%%%%%%%%%Figures f-distribution
\newpage
\clearpage
\begin{landscape}
\begin{figure}[h!tb]
\begin{center}
\begin{tabular}{ccc}
\scalebox{0.4}{\includegraphics[]{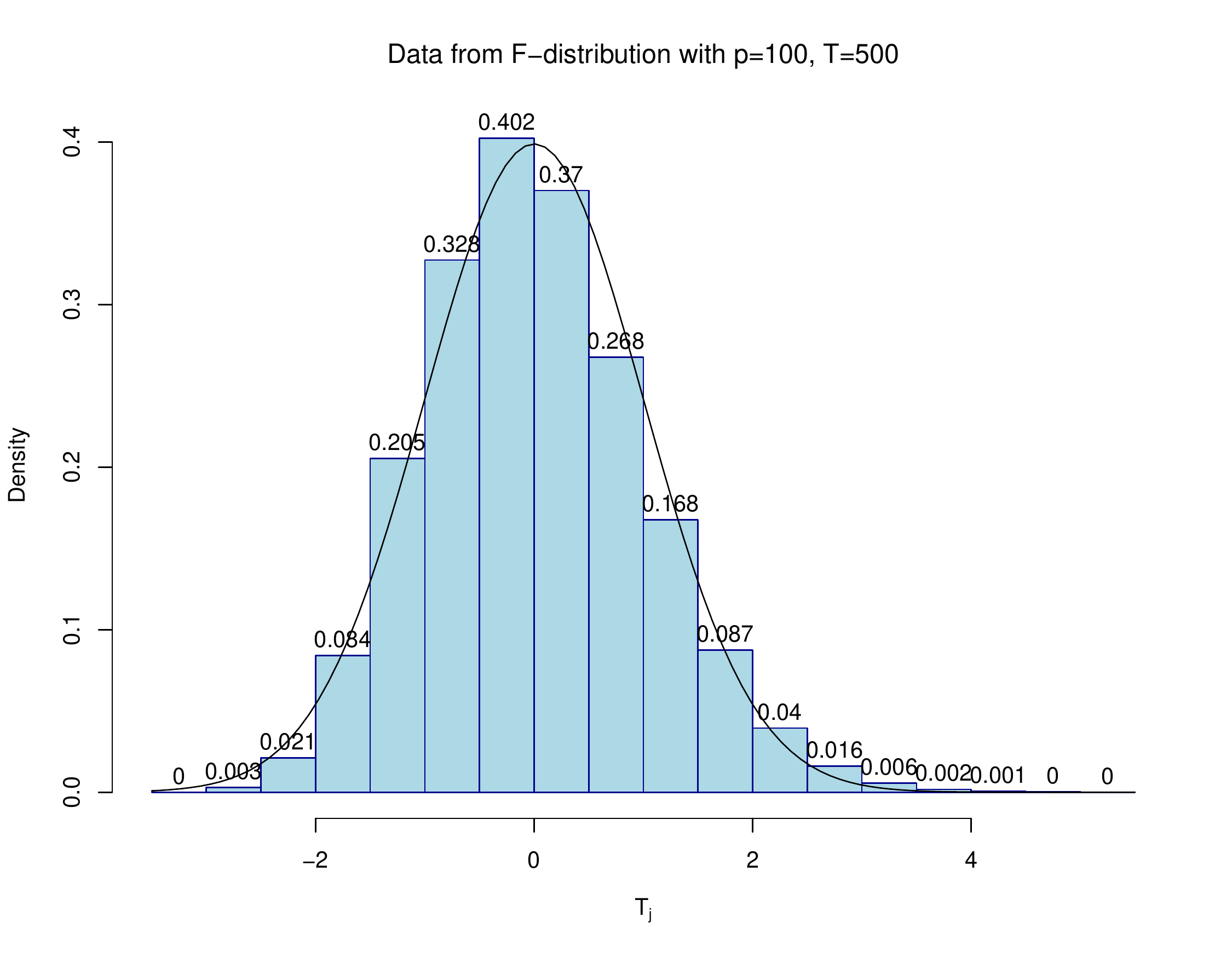}}&&\scalebox{0.4}{\includegraphics[]{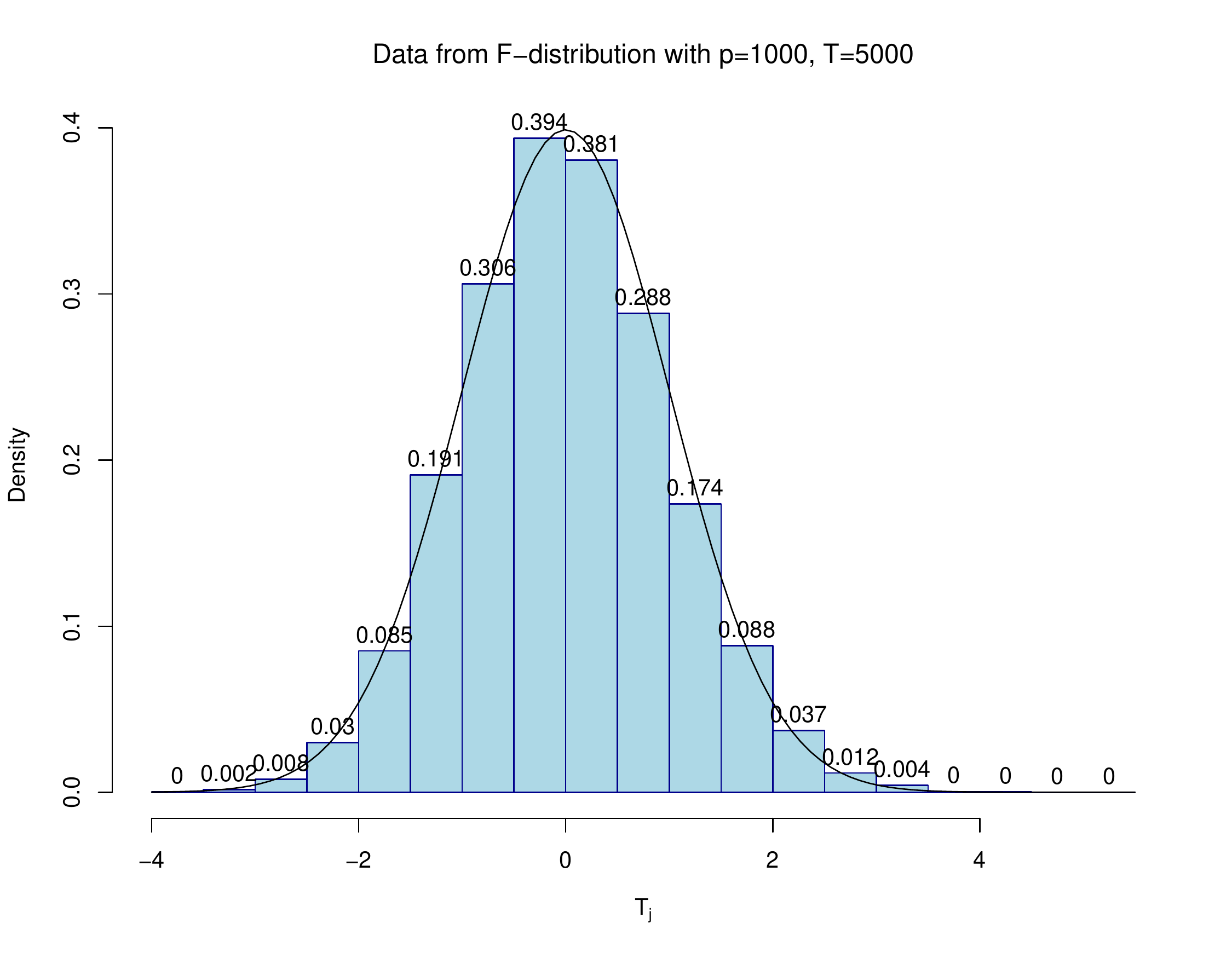}}\\
%&&\\
\scalebox{0.4}{\includegraphics[]{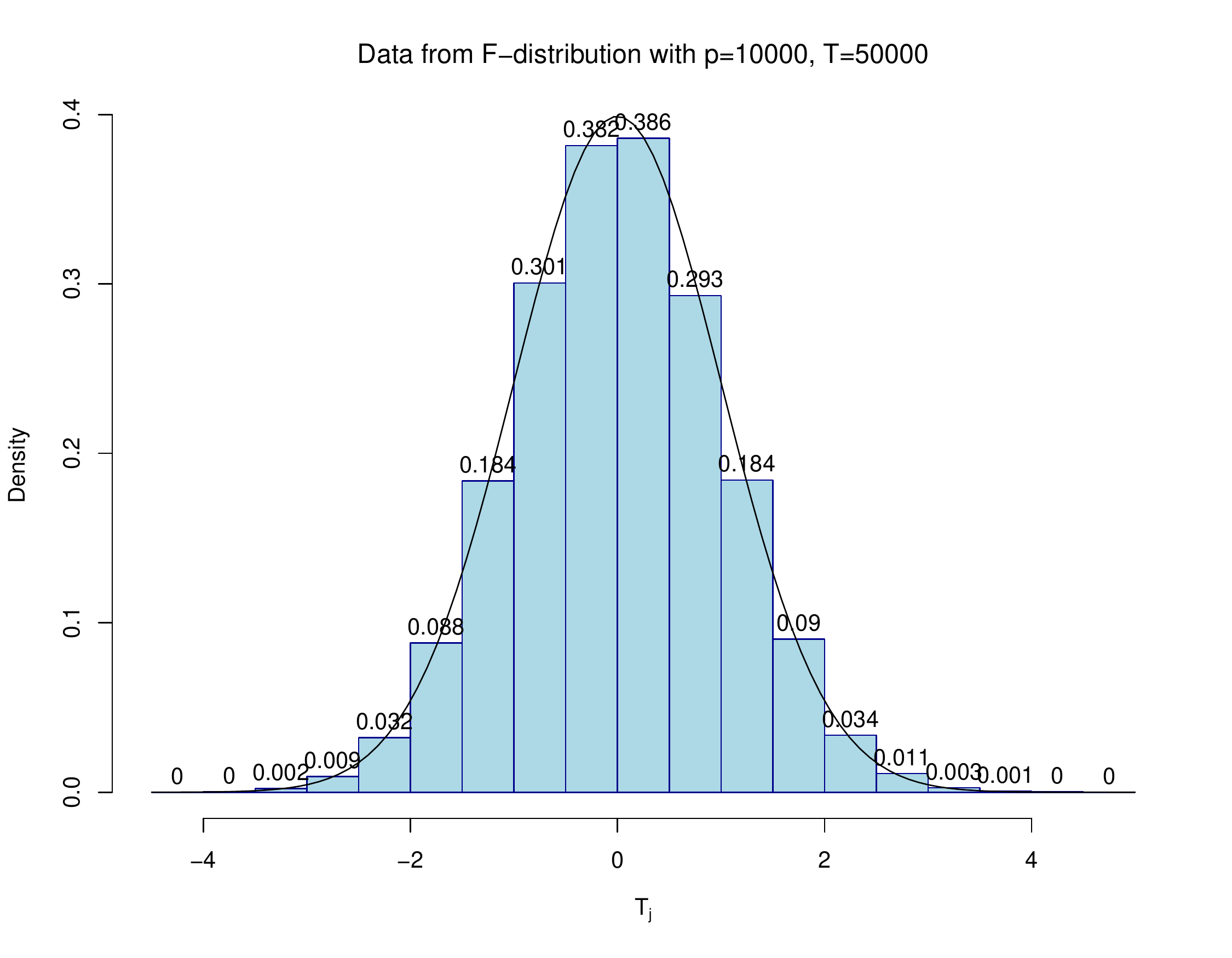}}&&\scalebox{0.4}{\includegraphics[]{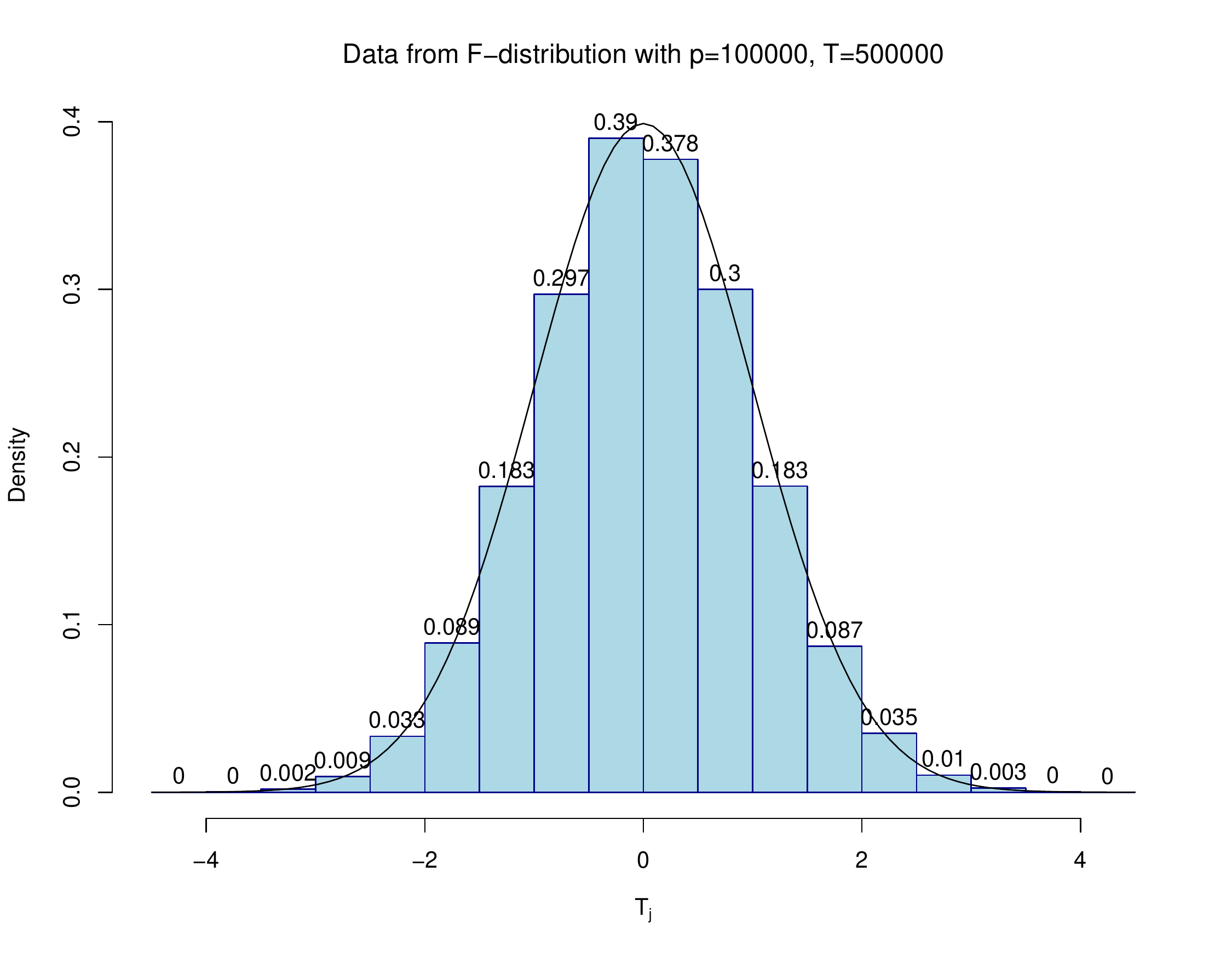}}\\
\end{tabular}
\end{center}
\caption{Approximation of the $\mathcal{F}_{d_1,d_2}$-distribution with large degrees of freedom by a normal distribution for $d_1 \in \{100,1000,10000,100000\}$ and $d_2=5d_1$. }
\label{Fig:Fdist}
\end{figure}
\end{landscape}

%%%%%%%%%%%%%%%%%%%%%%%%%%%%%%%%%%%%%Figures Section 7.1

\newpage
\clearpage

\begin{figure}[h!tb]
\begin{center}
\begin{tabular}{ccc}
\scalebox{0.35}{\includegraphics[]{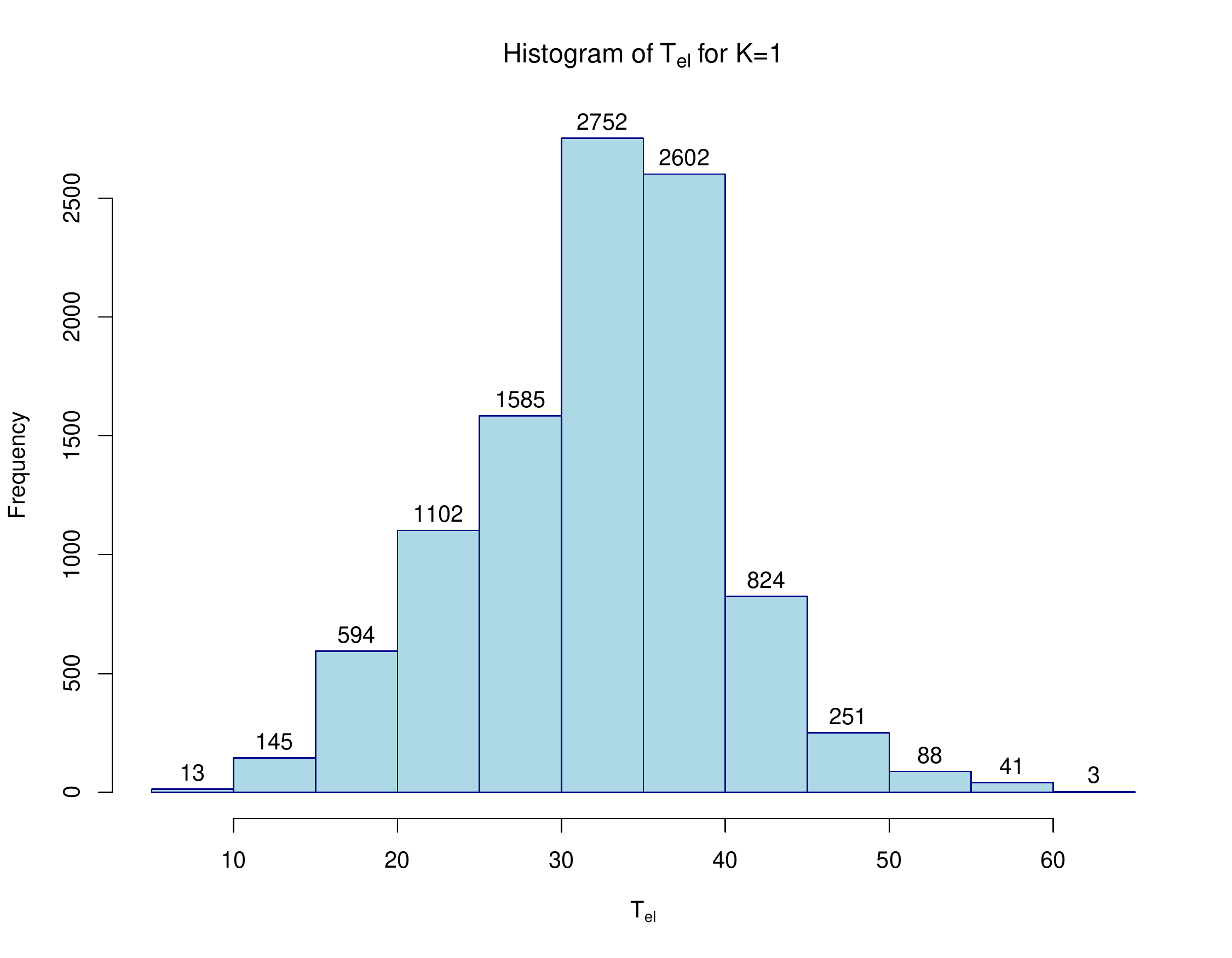}}&&\scalebox{0.35}{\includegraphics[]{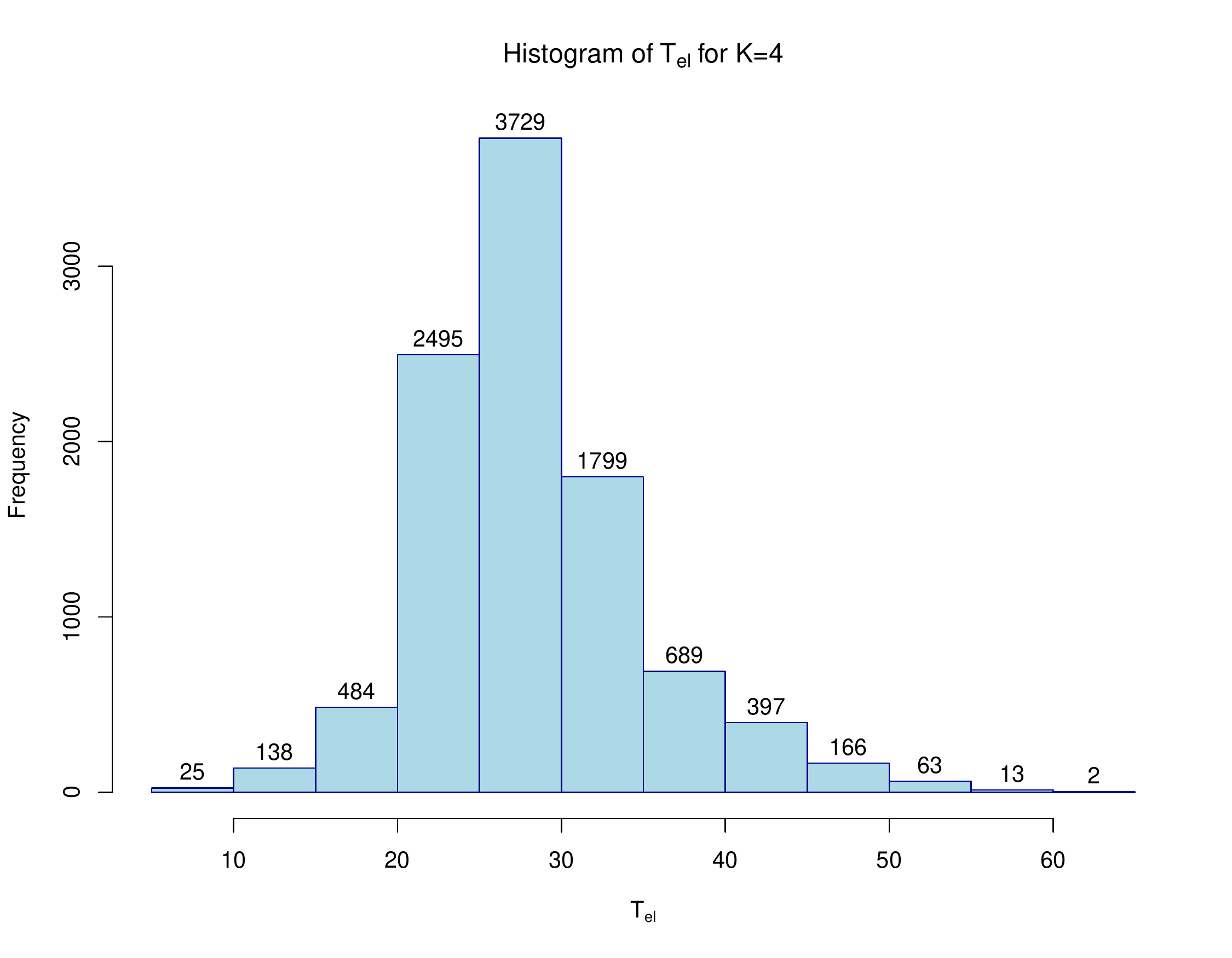}}\\
%&&\\
\scalebox{0.35}{\includegraphics[]{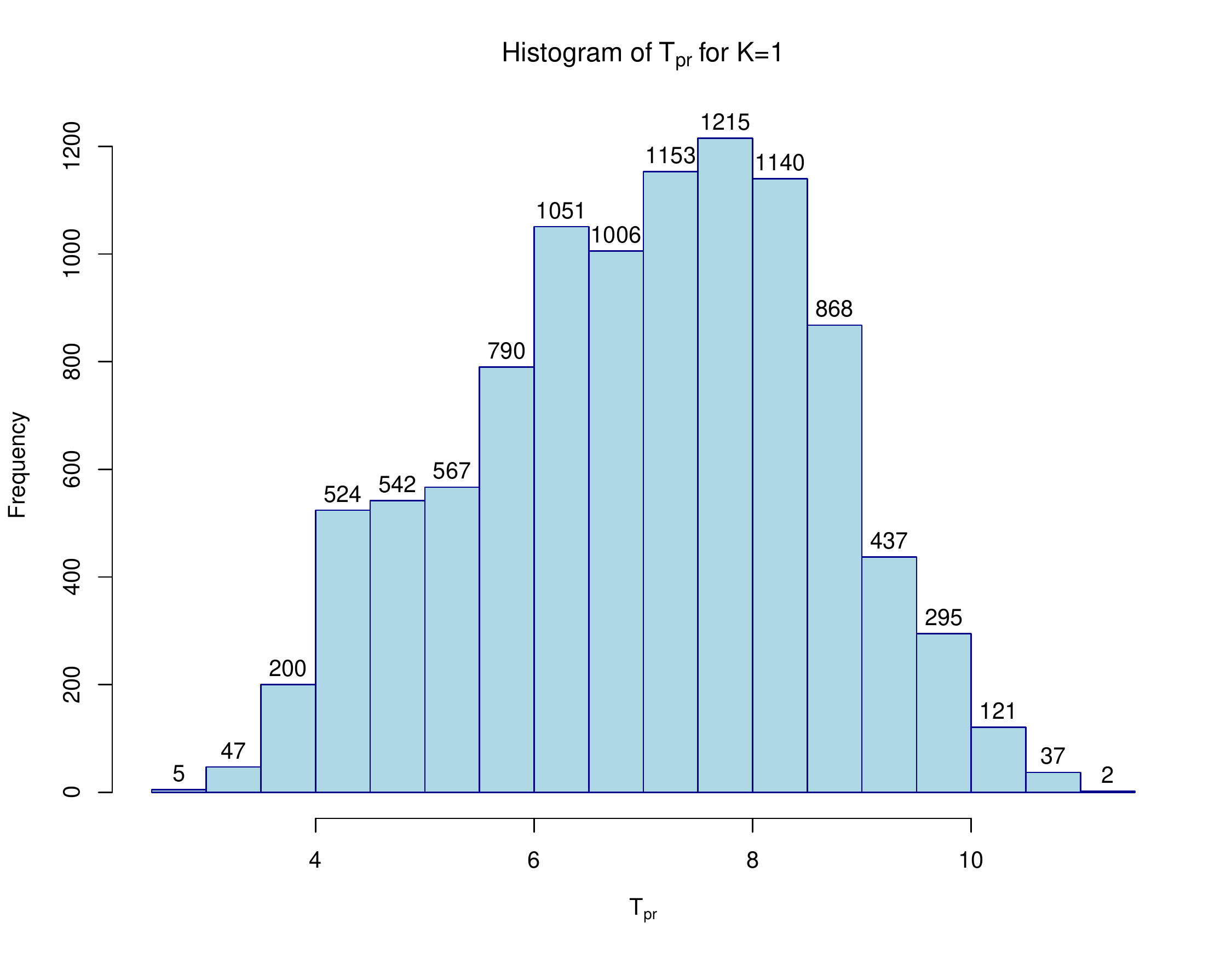}}&&\scalebox{0.35}{\includegraphics[]{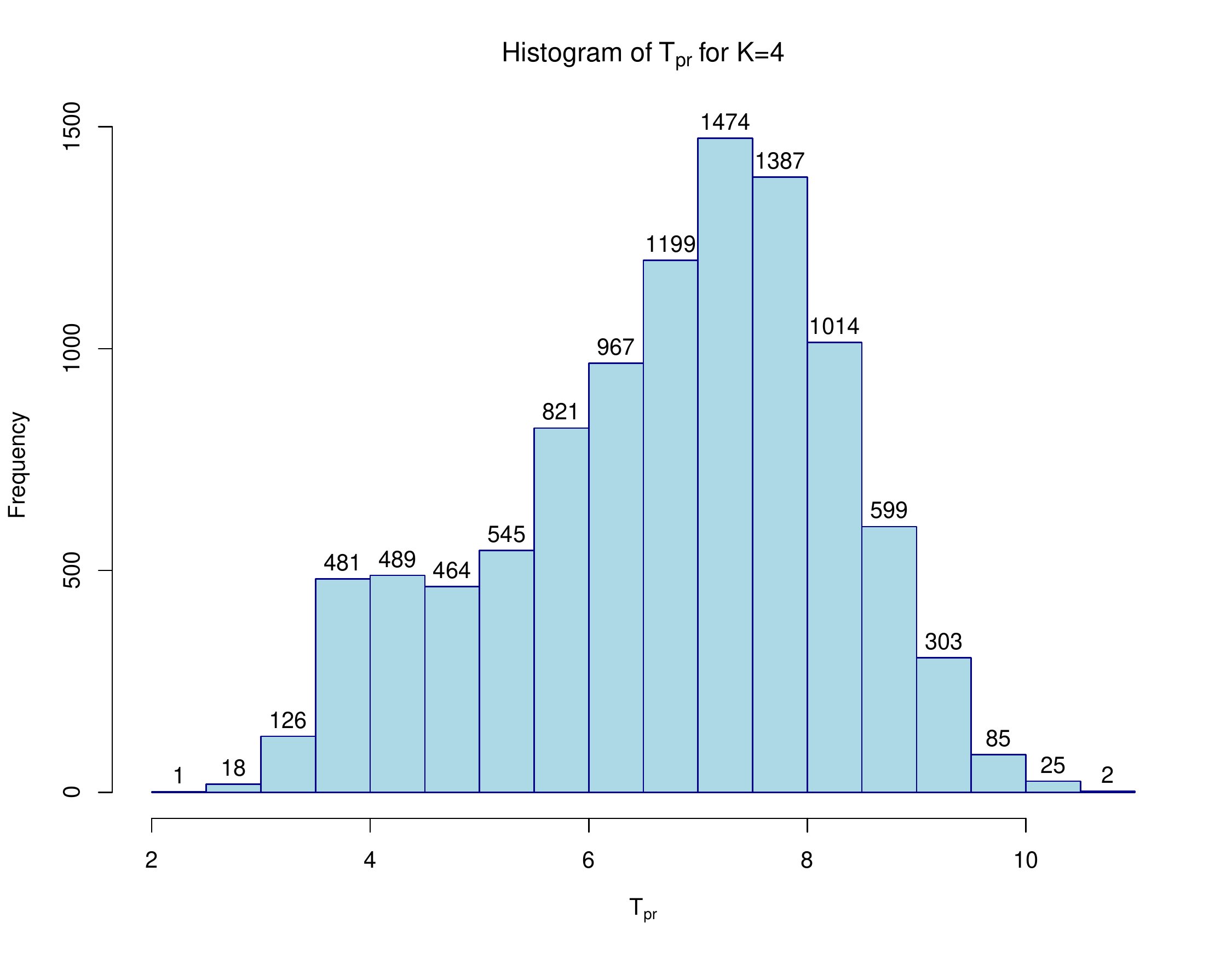}}\\
%&&\\
\scalebox{0.35}{\includegraphics[]{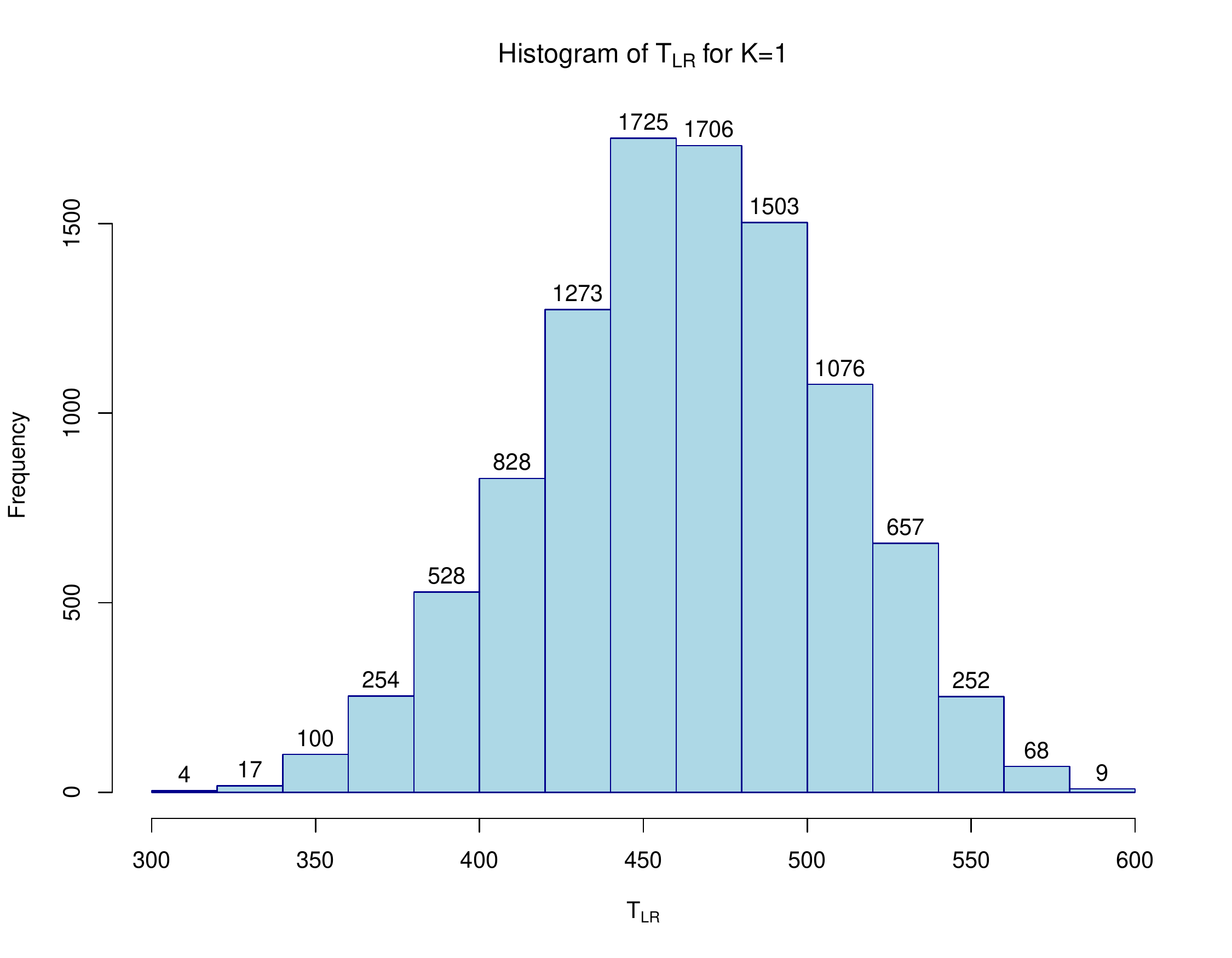}}&&\scalebox{0.35}{\includegraphics[]{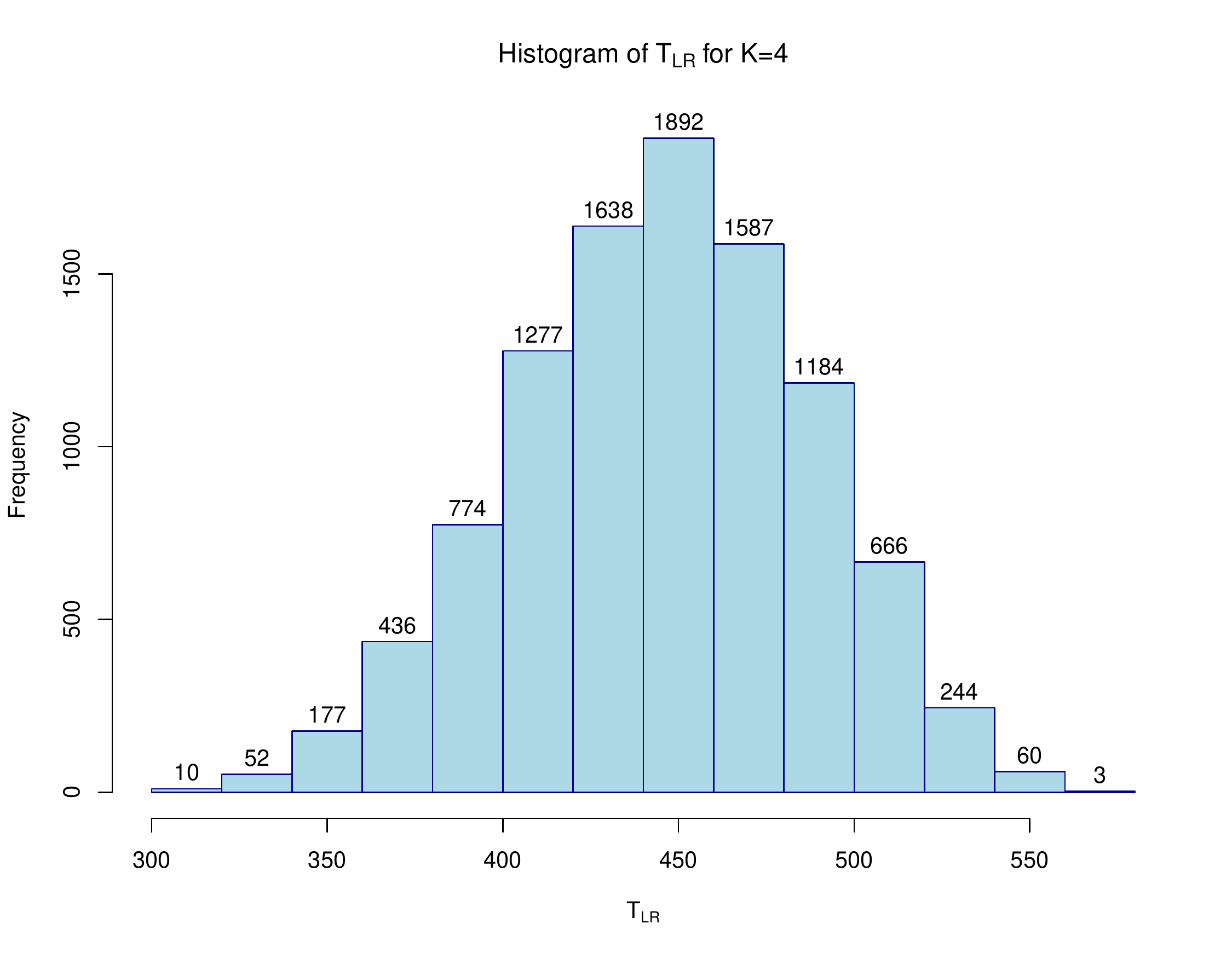}}\\
\end{tabular}
\end{center}
\caption{Histograms for the values of the test statistics $T_{el}$, $T_{pr}$, and $T_{LR}$ for portfolios of size $p=20$ constructed using the assets included into the DAX index. The data of weekly returns is used from the 11th of June 2012 to the 10th of June 2014 ($T=104$). The number of factors included into the model is equal to $K=1$ (left-hand side figures) and $K=4$ (right-hand side figures).}
\label{Fig:DAX}
\end{figure}

%%%%%%%%%%%%%%%%%%%%%%%%%%%%%%%%%%%%%Figures Section 7.2
\begin{figure}[h!tb]
\begin{center}
\begin{tabular}{ccc}
\scalebox{0.35}{\includegraphics[]{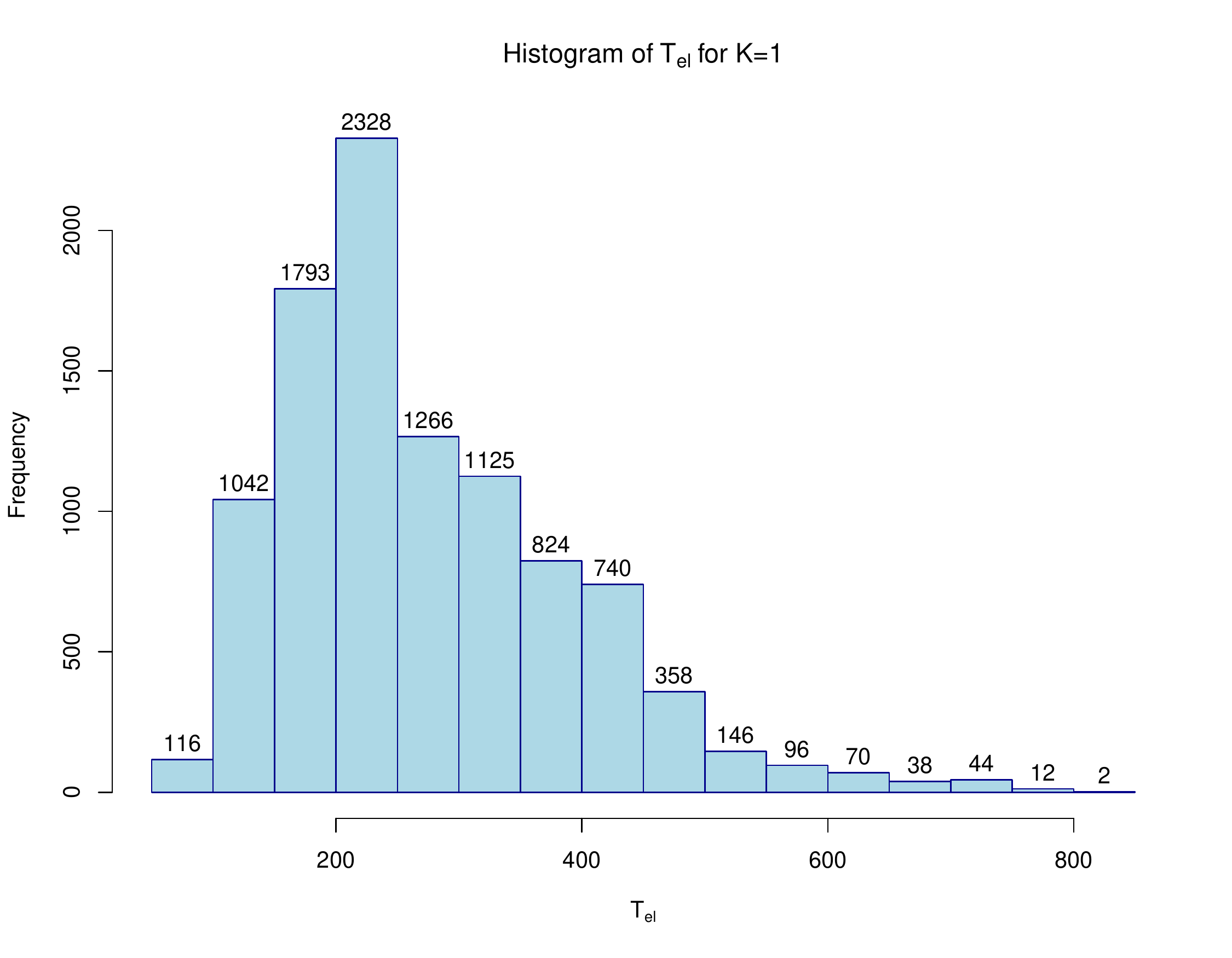}}&&\scalebox{0.35}{\includegraphics[]{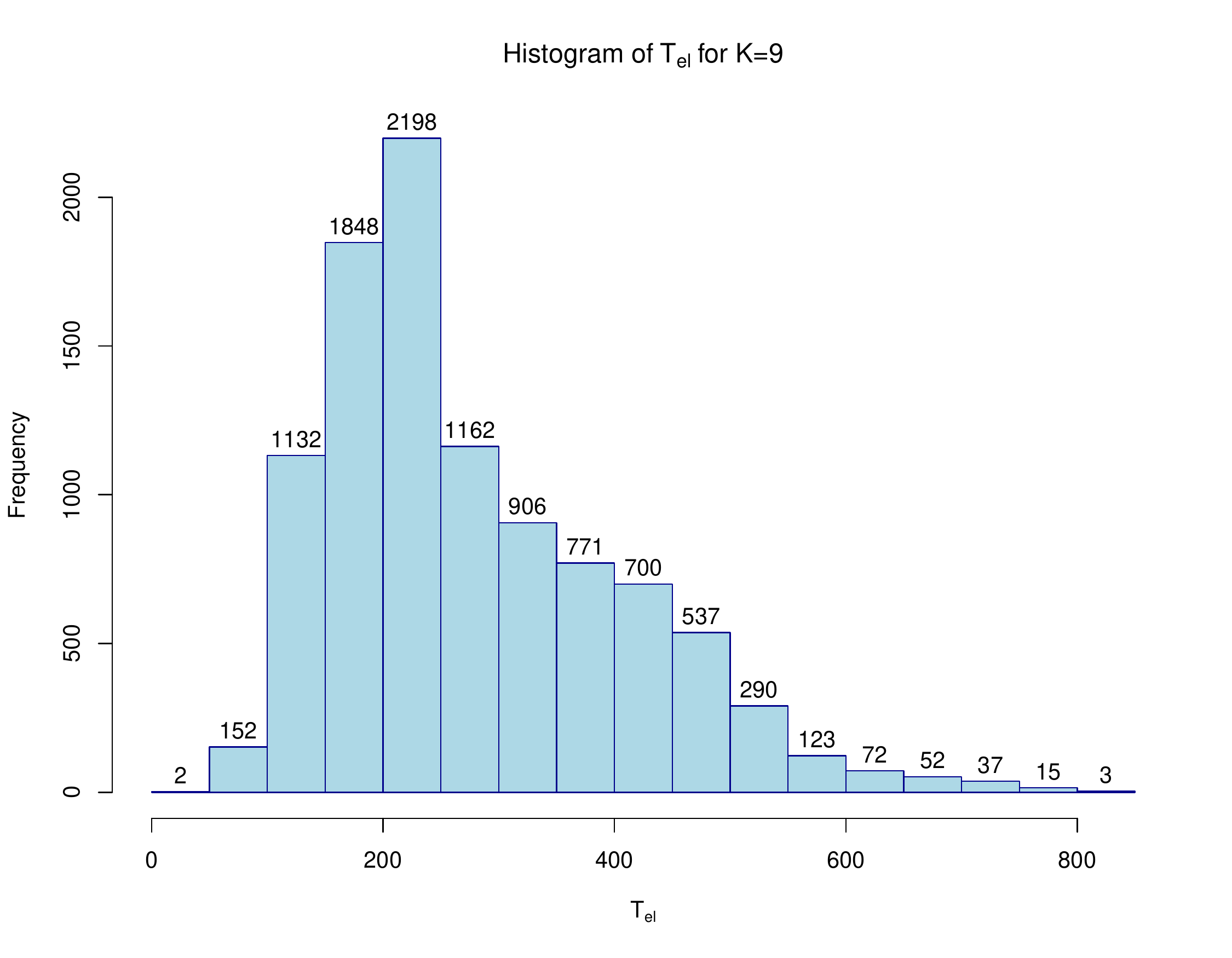}}\\
%&&\\
\scalebox{0.35}{\includegraphics[]{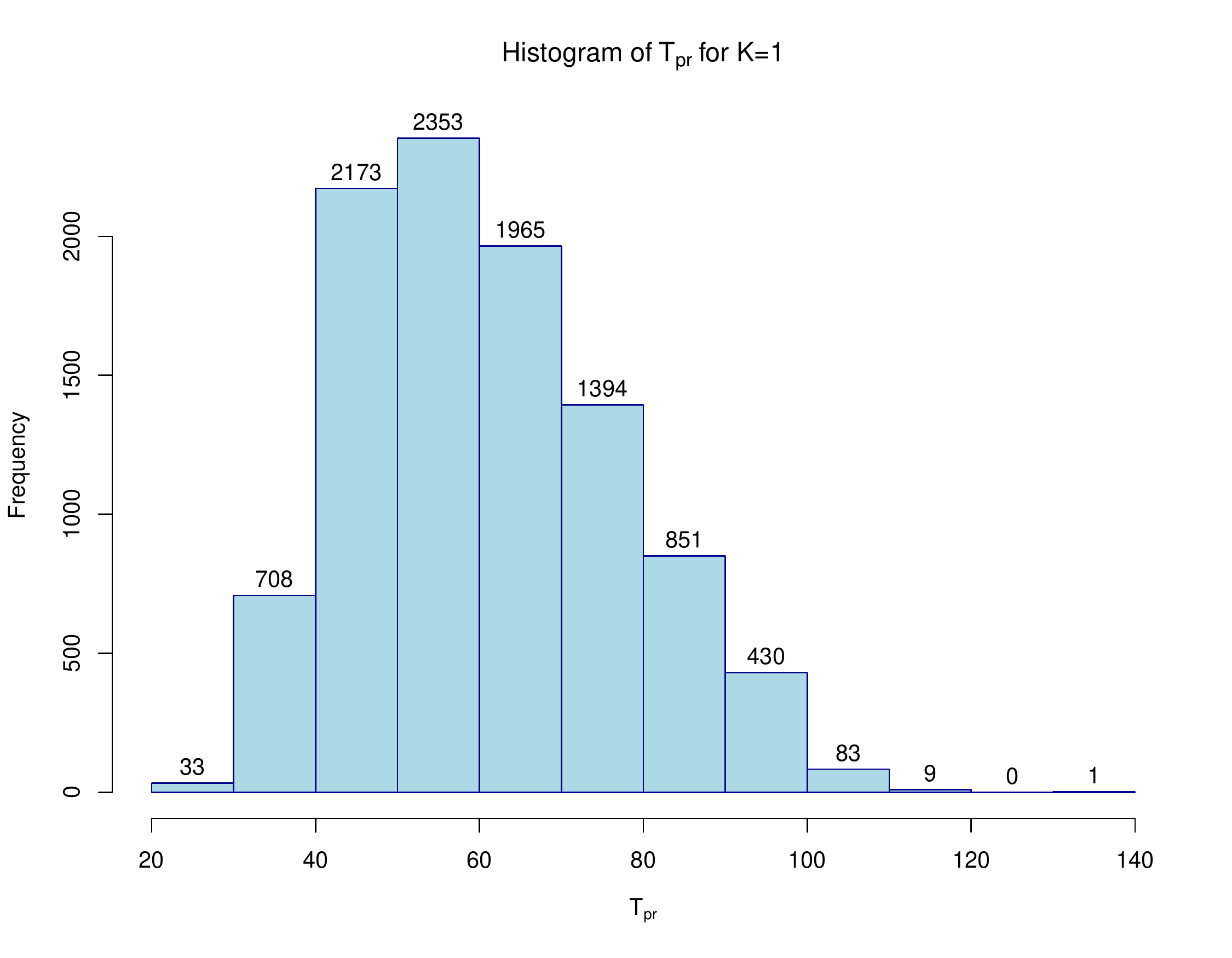}}&&\scalebox{0.35}{\includegraphics[]{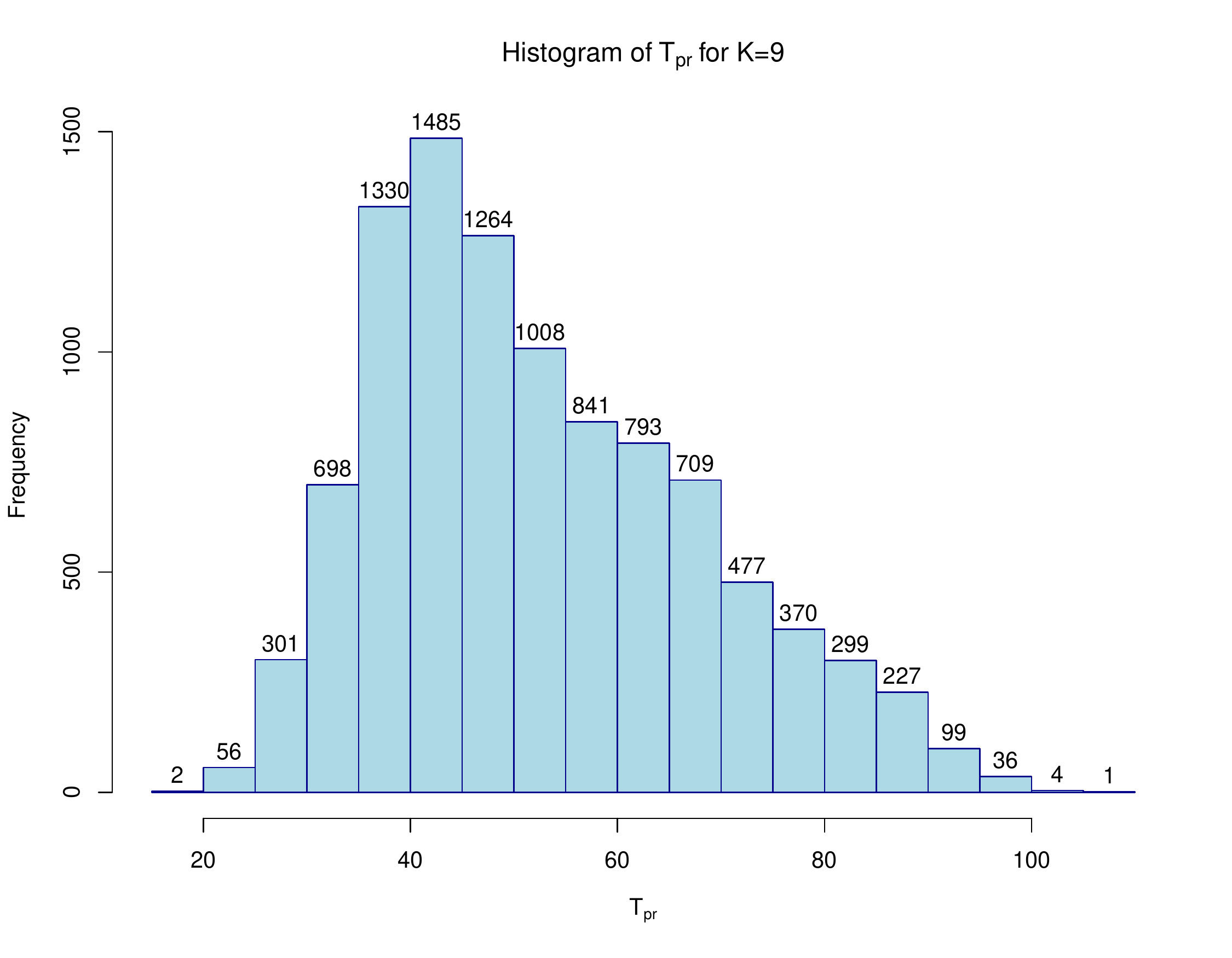}}\\
%&&\\
\scalebox{0.35}{\includegraphics[]{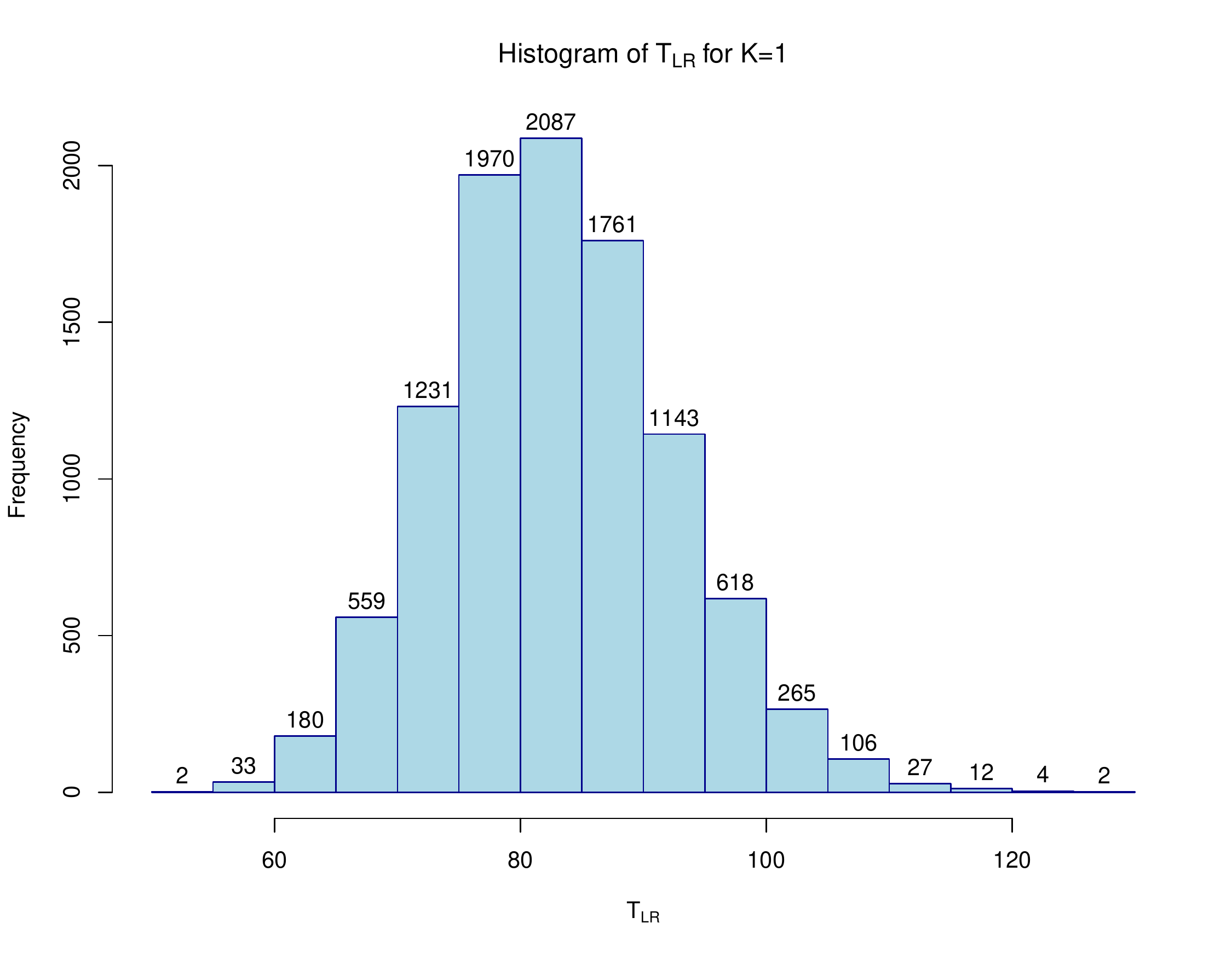}}&&\scalebox{0.35}{\includegraphics[]{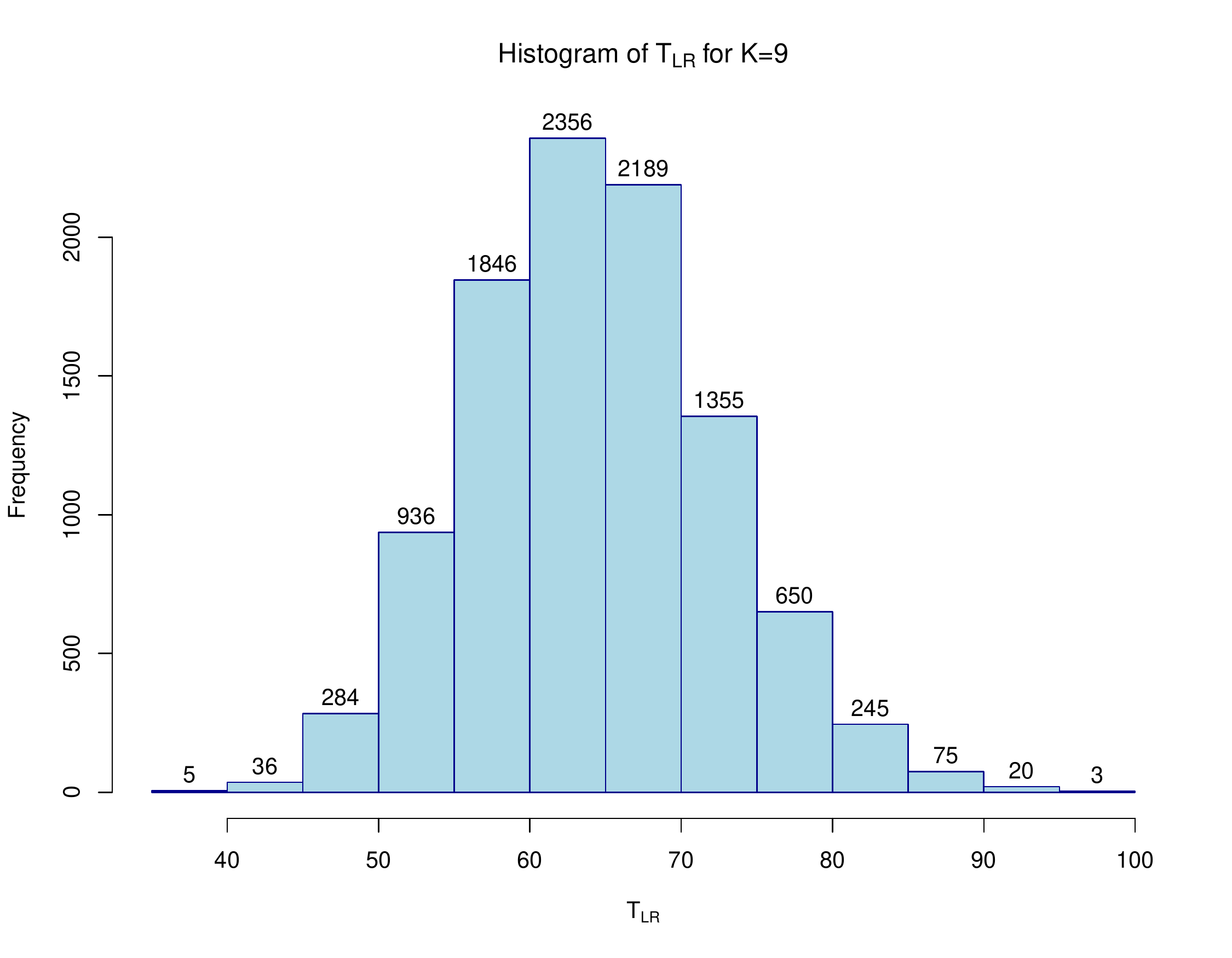}}\\
\end{tabular}
\end{center}
\caption{Histograms for the values of the test statistics $T_{el}$, $T_{pr}$, and $T_{LR}$ for portfolios of size $p=100$ constructed using the assets included into the SP index. The data of weekly returns is used from the 10th of June 2004 to the 10th of June 2014 ($T=518$). The number of factors included into the model is equal to $K=1$ (left-hand side figures) and $K=9$ (right-hand side figures).}
\label{Fig:SP}
\end{figure}

\end{document}